\documentclass[11pt]{article}
\usepackage{smile}

\usepackage{fullpage}
\usepackage{lscape}
\usepackage{framed}
\usepackage{mdframed}
\usepackage{enumerate}
\usepackage[inline]{enumitem}
\usepackage[T1]{fontenc}
\usepackage{moresize}
\usepackage{bm}
\usepackage{dsfont}
\usepackage{amsmath}
\usepackage{amssymb}
\usepackage{amsthm}
\usepackage{amsfonts}
\usepackage{stmaryrd}
\usepackage{array}
\usepackage{mathrsfs}
\usepackage{mathtools} 
\usepackage{extarrows}
\usepackage{stackrel}
\usepackage[nodisplayskipstretch]{setspace}
\usepackage{color}
\usepackage[usenames,dvipsnames]{xcolor}
\usepackage{cancel}
\usepackage{soul}
\usepackage{undertilde}
\usepackage{xfrac}
\usepackage{siunitx}
\usepackage{graphicx}
\usepackage{float}
\usepackage{rotating}
\usepackage{subcaption}
\usepackage{overpic}
\usepackage[all]{xy}
\usepackage{tikz}
\usetikzlibrary{arrows,matrix,positioning,calc,automata,patterns}
\usepackage{booktabs}
\usepackage{dcolumn}
\usepackage{multirow}
\usepackage{diagbox}
\usepackage{verbatim}
\usepackage{listings}
\usepackage{algorithm}
\usepackage{algpseudocode}
\usepackage{fancyvrb}
\usepackage{hyperref}
\usepackage[round]{natbib}
\usepackage{sectsty}
\usepackage{bbm}

\hypersetup{
    bookmarks=true,         
    unicode=false,          
    pdftoolbar=true,        
    pdfmenubar=true,        
    pdffitwindow=false,     
    pdfstartview={FitH},    
    pdftitle={My title},    
    pdfauthor={Author},     
    pdfsubject={Subject},   
    pdfcreator={Creator},   
    pdfproducer={Producer}, 
    pdfkeywords={key1, key2}, 
    pdfnewwindow=true,      
    colorlinks=true,        
    linkcolor=blue,         
    citecolor=blue,         
    filecolor=blue,         
    urlcolor=cyan           
}

\usepackage{stackengine}
\stackMath
\newcommand\tenq[2][1]{%
\def\useanchorwidth{T}%
\ifnum#1>1%
\stackunder[0pt]{\tenq[\numexpr#1-1\relax]{#2}}{\!\scriptscriptstyle\thicksim}%
\else%
\stackunder[1pt]{#2}{\!\scriptstyle\thicksim}%
\fi%
}

\def\tr{\mathop{\text{tr}}\kern.2ex}

\def\P{{\mathrm P}}
\def\Q{{\mathrm Q}}
\def\U{{\mathrm U}}
\def\E{{\mathrm E}}
\def\R{{\mathbbm R}}
\def\Z{{\mathbbm Z}}

\def\d{{\mathrm d}}

\newcommand{\sfi}{\mathsf{i}}
\newcommand{\ac}{\mathrm{ac}}
\newcommand{\sgn}{\mathrm{sgn}}
\newcommand{\dCov}{\mathrm{dCov}}

\renewcommand{\Pr}{\mathrm{P}}

\newcommand{\card}{\mathrm{card}}

\newcommand{\TV}{\mathrm{TV}}

\newcommand{\HL}{\mathrm{HL}}

\newcommand{\Arc}{\mathsf{Arc}}
\newcommand{\mbinom}{\binom}
\newcommand{\zahl}[1]{\llbracket #1\rrbracket}
\newcommand\yestag{\addtocounter{equation}{1}\tag{\theequation}}
\newcolumntype{L}[1]{>{\raggedright\let\newline\\\arraybackslash\hspace{0pt}}m{#1}}
\newcolumntype{C}[1]{>{  \centering\let\newline\\\arraybackslash\hspace{0pt}}m{#1}}
\newcolumntype{R}[1]{>{ \raggedleft\let\newline\\\arraybackslash\hspace{0pt}}m{#1}}
\newcolumntype{d}[1]{D{.}{.}{#1}}
\newcolumntype{H}{>{\setbox0=\hbox\bgroup}c<{\egroup}@{}}
\newcolumntype{Z}{>{\setbox0=\hbox\bgroup}c<{\egroup}@{\hspace*{-\tabcolsep}}}

\newcommand{\n}{^{(n)}}

\def\pms{\mspace{-1mu}{\scriptscriptstyle{\pm}}}

\numberwithin{equation}{section}

\newtheorem{theorem}{Theorem}[section]
\newtheorem{lemma}{Lemma}[section]
\newtheorem{proposition}{Proposition}[section]
\newtheorem{assumption}{Assumption}[section]
\newtheorem{corollary}{Corollary}[section]
\newtheorem{innercustomgeneric}{\customgenericname}
\providecommand{\customgenericname}{}
\newcommand{\newcustomtheorem}[2]{%
  \newenvironment{#1}[1]
  {%
   \renewcommand\customgenericname{#2}%
   \renewcommand\theinnercustomgeneric{##1}%
   \innercustomgeneric
  }
  {\endinnercustomgeneric}
}
\newcustomtheorem{customdefinition}{Definition}
\newcustomtheorem{customdefinitions}{Definitions}
\newcustomtheorem{customtheorem}{Theorem}
\newcustomtheorem{customassumption}{Assumption}
\newcustomtheorem{customlemma}{Lemma}
\newcustomtheorem{customexample}{Example}
\theoremstyle{definition}
\newtheorem{definition}{Definition}[section]
\newtheorem{example}{Example}[section]
\newtheorem{remark}{Remark}[section]

\begin{document}

\setlength{\abovedisplayskip}{5pt}
\setlength{\belowdisplayskip}{5pt}
\setlength{\abovedisplayshortskip}{5pt}
\setlength{\belowdisplayshortskip}{5pt}
\hypersetup{colorlinks,breaklinks,urlcolor=blue,linkcolor=blue}

\title{\LARGE On universally consistent and fully
  distribution-free\\ rank tests of vector independence}
  
\author{
Hongjian Shi\thanks{Department of Statistics, University of Washington, Seattle, WA 98195, USA; e-mail: {\tt hongshi@uw.edu}},~~
Marc Hallin\thanks{ECARES and Department of Mathematics, Universit\'e Libre de Bruxelles, Brussels, Belgium; email: {\tt mhallin@ulb.ac.be}},~~
Mathias Drton\thanks{Department of Mathematics, Technical University
  of Munich, 85748 Garching b.\  M\"unchen, Germany; e-mail: {\tt mathias.drton@tum.de}},~~and~
Fang Han\thanks{Department of Statistics, University of Washington, Seattle, WA 98195, USA; e-mail: {\tt fanghan@uw.edu}}
}

\date{}

\maketitle

\begin{abstract}
  Rank correlations have found many innovative applications in the
  last decade.  In particular, suitable 
  rank correlations
  have been used for consistent tests of independence between pairs of
  random variables.
  Using ranks is especially appealing for
  continuous data as tests become distribution-free.
  However, the
  traditional concept of ranks relies on ordering data and is, thus,
  tied to univariate observations.  As a result, it has long remained
  unclear how one may construct distribution-free yet consistent   tests of
  independence between 
  random vectors.  This is the
  problem addressed in this paper, in which we lay out a general
  framework for designing dependence measures that give tests of
  multivariate independence that are not only consistent and
  distribution-free but which we also prove to be statistically
  efficient.  Our framework leverages the recently introduced concept
  of center-outward ranks and signs, a multivariate generalization of
  traditional ranks, and adopts a common standard form for
  dependence measures that encompasses many popular examples.
  In a unified study, we derive a general asymptotic
  representation of center-outward rank-based test statistics under independence,
  extending to the multivariate setting the classical H\'{a}jek   asymptotic
  representation results.  This representation
  permits direct calculation of limiting null distributions
  and facilitates a local power
  analysis that provides strong support for the center-outward
  approach
  by establishing, for the first time,
  the  {nontrivial power of center-outward rank-based tests over root-$n$ neighborhoods
  within the class of quadratic mean differentiable alternatives}.
\end{abstract}


\section{Introduction}\label{sec:intro}
\color{black}Quantifying the dependence between two variables and  testing for their independence are among the oldest and most fundamental problems of statistical inference. The (marginal) distributions of the two variables under study, in that context, typically play the role of nuisances, and the need for a nonparametric approach naturally leads, when they are univariate,   to distribution-free methods based on their ranks. This paper  is dealing with the multivariate extension of that approach. \color{black}

\subsection{Measuring vector dependence and testing independence}

Consider two absolutely continuous random vectors $\mX_1$ and $\mX_2$, 
with values in $\R^{d_1}$ and~$\R^{d_2}$, respectively. The problems
of measuring the dependence between $\mX_1$ and $\mX_2$ and testing
their independence  when~$d_1=d_2=1$ (call this the univariate case)
have a long history that goes back more than a century
\citep{Pearson1895,Spearman1904}. The same problem when $d_1$ and
$d_2$  are possibly unequal and larger than one (the multivariate
case)  is of equal practical interest but considerably more   challenging. Following early attempts \citep{Wilks1935},  a large  literature has emerged,  with renewed interest in recent years.


When the marginal distributions of $\mX_1$ and $\mX_2$ are unspecified
and $d_1=d_2=1$,
rank 
correlations provide a natural and appealing
nonparametric approach to testing for independence, as initiated in
the work of
\cite{Spearman1904} and  \cite{Kendall1938}; cf.~Chapter III.6 in \cite{MR0229351}.  On one hand, ranks yield
distribution-free tests because, under the null hypothesis of  independence, their   distributions do not depend on the unspecified marginal distributions. On the other hand, they can be designed \citep{MR0029139,MR0125690,MR3178526,yanagimoto1970measures} 
to consistently  estimate 
dependence measures that vanish if and only if independence holds, and so detect any type of dependence---something  Spearman and Kendall's rank correlations cannot.


New subtleties arise, however, when attempting to extend the   rank-based approach to the
multivariate case.
While $d_k$ ranks can be constructed separately for each coordinate of
$\mX_k$, $k=1,2$, their joint distribution  depends on the  distribution
of the underlying $\mX_k$, preventing distribution-freeness of the $(d_1+d_2)$-tuple of ranks.  As a consequence, the existing  tests of
multivariate independence based on componentwise ranks  \citep[e.g.,][]{MR298847} are not distribution-free,
which has both computational implications (e.g., through a need for
permutation analysis) and statistical implications (as we shall detail
soon).


\subsection{Desirable properties}\label{sec:desiredsec}
In this paper, we develop a general framework 
for 
 multivariate analogues of popular rank-based measures of dependence for
 the univariate case.
 Our objective is  to achieve
   the following 
 five desirable properties. 

\begin{enumerate}[itemsep=0ex,label=(\arabic*), wide, labelwidth=!, labelindent=0pt]
\item\label{pty:1} {\it Full distribution-freeness.}  Many statistical tests exploit
asymptotic distribution-freeness for computationally efficient
distributional approximations  yielding pointwise asymptotic control
of their size.  This is the case, for instance, with \cite{MR1926170,MR1963662,MR2001322,MR2462206} due to estimation of a scatter matrix, or  with  \cite{
MR1965367,MR2088309}, \cite{MR2201019}.
Pointwise asymptotics yield, for any given significance level $\alpha\in(0,1)$, 
a sequence of tests $\phi^{(n)}_\alpha$ indexed by the sample size~$n$ such
that $\lim_{n\to\infty}\E_{\rm P} [\phi^{(n)}_\alpha]=\alpha$ for every
distribution $\rm P$ from a class  $\mathcal{P}$ of null distributions. 
Generally, however, the size fails to be controlled in a uniform
sense,
that is, it does not hold that 
$\lim_{n\to\infty}\sup_{{\rm P}\in\mathcal{P}}\E_{\rm P} [\phi^{(n)}_\alpha]\leq \alpha$,
which may explain poor finite-sample properties \citep[see,
e.g.,][]{MR1784901,MR2422862,MR3207983}.  While uniform inferential
validity is impossible to achieve for some problems, e.g., when 
testing for conditional independence
\citep{MR4124333,azadkia2019simple}, we shall see that it is
achievable for testing (unconditional) multivariate independence.
Indeed, for fully distribution-free tests, as obtained from our
rank-based approach,
pointwise validity automatically implies uniform validity.

\item\label{pty:2} {\it Transformation invariance.} A dependence measure $\mu$ is
said to be invariant under orthogonal transformations, shifts, and
global rescaling if
$$\mu(\mX_1,\mX_2)=\mu(\mv_1+a_1\fO_1\mX_1, \mv_2+a_2\fO_1\mX_2)$$ for
any scalars $a_k>0$, vectors $\mv_k\in\R^{d_k}$, and orthogonal
$d_k\times d_k$ matrices $\fO_k$,   $k=1,2$.  This invariance, here simply termed ``transformation invariance'', is a natural
requirement in cases where the components of $\mX_1,\mX_2$ do not have
specific meanings and observations could have been recorded in
another coordinate system.  Such invariance is
of
considerable interest in multivariate statistics 
\citep[see, e.g.,][]{MR1467849,MR1965367,MR2201019,MR3544291}.

\item\label{pty:3} {\it Consistency.}
  \citet{MR3842884} call a dependence measure $\mu$ 
  {\it
  I-consistent} within a family of distributions $\mathcal{P}$ if
independence between 
$\bm{X}_1$ and $\bm{X}_2$ with
joint distribution in $\mathcal{P}$   implies
$\mu(\bm{X}_1,\bm{X}_2) =~\!0$.  
If $\mu(\bm{X}_1,\bm{X}_2)
= 0$ implies independence of $\bm{X}_1$ and $\bm{X}_2$ (i.e., dependence of $\bm{X}_1$ and $\bm{X}_2$ implies
$\mu(\bm{X}_1,\bm{X}_2) \ne 0$), then $\mu$ is {\it D-consistent}
within $\mathcal{P}$.
  {Note that the measures considered in this paper do
   not necessarily take  maximal value 1 if and only if  one random
   vector is a measurable function of the other.} While any reasonable
 dependence measure should be I-consistent, prominent 
 examples (Pearson's correlation, Spearman's $\rho$, Kendall's $\tau$) fail to be D-consistent. If a dependence measure~$\mu$ is I- and D-consistent, then the consistency of tests based on an estimator $\mu\n$ of~$\mu$ is guaranteed by the (strong or weak) consistency of that estimator.  Dependence measures
that are both I- and D-consistent (within a large nonparametric
family) serve an important purpose as they are able to capture
nonlinear dependences.
Well-known I- and D-consistent measures for the univariate
case  include Hoeffding's $D$ \citep{MR0029139},
Blum--Kiefer--Rosenblatt's~$R$ \citep{MR0125690}, and
Bergsma--Dassios--Yanagimoto's $\tau^*$
\citep{MR3178526,yanagimoto1970measures,MR4185806}.  Multivariate
extensions have been proposed, e.g., in  \citet{gretton:aistats2005},
\citet{MR2382665}, \citet{MR2956796}, \citet{MR3068450},
\citet{NIPS2016_6220}, \citet{MR3737307}, \citet{MR3842884},
\citet{MR4185814}, \citet{deb2019multivariate},
\citet{shi2019distribution}, \cite{berrett2020optimal}.


\item\label{pty:4} {\it Statistical efficiency.}  Once its size is controlled, the
  performance of a test may be evaluated through 
  its power
against local alternatives.  
 {For the proposed tests,  our focus is on quadratic mean
  differentiable alternatives \citep[Sec.~12.2]{MR2135927}, which form a popular class for conducting local power
analyses; for related recent examples see 
\citet[Section 3]{MR3961499} and 
\citet[Section~4.4]{cao2020correlations}.
Our results then show the nontrivial local power of our tests in $n^{-1/2}$ neighborhoods
within this class.}
%

\item\label{pty:5} {\it Computational efficiency.}  Statistical
  properties aside, modern applications require the evaluation of a
  dependence measure and the corresponding test to be as computationally
  efficient as possible.   We thus prioritize measures leading to  low computational complexity.

\end{enumerate}

The main challenge, with this list of  five properties, lies in
combining the full distribution-freeness from property~\ref{pty:1}
with  properties~\ref{pty:2}--\ref{pty:5}.
 The solution, as we shall see, involves an adequate multivariate
 extension of the univariate concepts  of ranks and signs. 

\subsection{Contribution of this paper}

This paper proposes a class of dependence measures and tests that achieve the five properties 
from Section~\ref{sec:desiredsec} by leveraging
the   recently introduced multivariate center-outward  ranks and signs \citep{MR3611491,hallin2017distribution}; see 
\citet{hallin2020distribution} for a complete account. In contrast to earlier related
concepts such as componentwise ranks \citep{MR0298844}, spatial ranks
\citep{MR2598854,MR3803462}, depth-based ranks \citep{MR1212489,MR2329471}, and
pseudo-Mahalanobis ranks and signs \citep{MR1926170}, the new concept
yields statistics that enjoy full distribution-freeness (in finite
samples and, thus, asymptotically) 
as soon as the
underlying probability measure is Lebesgue-absolutely continuous. 
 This allows for a general
 multivariate strategy, in which the 
 observations are replaced
by functions of  their center-outward ranks~and signs when forming dependence
measures and corresponding test statistics.  This is also the idea put
forth in \citet{shi2019distribution} and, in a slightly different
way, in \citet{deb2019multivariate}, where the focus is on distance
covariance between center-outward ranks and signs.

Methodologically, we are  generalizing this
approach in two important ways.  First, we introduce a class of   
\emph{generalized symmetric covariances}  (GSCs) 
along with their center-outward rank versions,  of which  the distance covariance
concepts from \citet{deb2019multivariate} and
\citet{shi2019distribution} are but particular cases.  Second,
we show how considerable additional flexibility and   power results
from incorporating score functions in the definition. 
 Our
simulations in Section \ref{sec:simulation} exemplify the benefits of
this  ``score-based'' approach.

From a theoretical point of view, we 
offer a new approach to asymptotic theory for the
proposed 
rank-based statistics.  Indeed, handling this general class 
with the methods of \citet{shi2019distribution} or
\citet{deb2019multivariate} would be highly nontrivial. Moreover,
these methods would not provide any insights into local power---an
issue 
receiving much attention also in other
contexts
\citep{hallin2019center,MR4154446,hallin2020multivariate,hallin2020efficient}.
We thus develop a completely different method, based on a
general asymptotic representation result applicable to all
center-outward rank-based GSCs under the null hypothesis of
independence and contiguous alternatives of  dependence. 
Our result (Theorem~\ref{thm:smallop}) is a multivariate extension of
H\'{a}jek's  classical asymptotic representation 
for univariate linear rank statistics 
\citep{MR0229351} and also simplifies 
the  derivation of limiting null distributions. 
Combined with a nontrivial use of Le~Cam's
third lemma in a context of non-Gaussian limits, our approach allows for 
the first  {local power} results in the area; the  {statistical efficiency} of the
tests of \citet{deb2019multivariate}  and  \citet{shi2019distribution} follows as a special
case.  {In Proposition \ref{prop:sGSC}, we establish the strong consistency of our rank-based tests against any fixed alternative under a regularity condition on the score function. Thanks to a recent result by \cite{deb2021efficiency}, that assumption can be relaxed: our tests, thus, enjoy  {\it universal consistency} against fixed dependence alternatives. }  

\vspace{-2mm}

\paragraph*{Outline of the paper}

The 
paper begins with a review of
important dependence measures from the
literature (Section~\ref{sec:sgc}).  Generalizing the idea of symmetric rank covariances
put forth in \citet{MR3842884}, we 
 show that a single
formula unifies them all;   
 we term the concept  {\it generalized symmetric covariance} (GSC).
As further background, Section~\ref{sec:hallin} introduces 
the notion of
center-outward ranks and signs.  Section~\ref{sec:measure} 
presents our streamlined approach of defining multivariate dependence
measures, along with sample counterparts, and highlights
some of their basic properties. Section \ref{sec:test} treats
tests of independence and develops  a theory of asymptotic
representation for center-outward rank-based GSCs
(Section~\ref{sec:representation}) as well as the local
power analysis of the corresponding tests  {against classes of quadratic mean differentiable alternatives} (Section~\ref{sec:local-power}).
 {Specific alternatives are exemplified in
  Section~\ref{sec:example-new}, and} 
benefits of  {choosing standard} score functions (
such as normal scores) are illustrated in the numerical study in
Section \ref{sec:simulation}. All proofs are deferred to the appendix.

\paragraph*{Notation}
 {For integer~$m\ge 1$, put~$\zahl{m}:=~\!\{1,2,\ldots,m\}$, and let 
$\kS_m$ be  the symmetric group, i.e., the group of all permutations
of $\zahl{m}$. We write $\sgn(\sigma)$ for the sign of
$\sigma\in\kS_m$.  In the sequel, the subgroup 
\begin{align}
  \label{eq:H*m}
&H_{*}^{m}:=\langle(1~4),(2~3)\rangle=\{(1),(1~4),(2~3),(1~4)(2~3)\}
                 \;\subset\; \kS_m
\end{align}
will play an important role.  Here, we have made use of the cycle
notation (omitting 1-cycles) so that, e.g.,
$(1)$ denotes the identity permutation and
\[
  (1~4)\equiv\bigg(\begin{matrix}1&2&3&4&5&6&\cdots&m\\4&2&3&1&5&6&\cdots&m\end{matrix}\bigg), \qquad
  (1~4)(2~3)\equiv\bigg(\begin{matrix}1&2&3&4&5&6&\cdots&m\\4&3&2&1&5&6&\cdots&m\end{matrix}\bigg),\]
where the right-hand sides are in classical two-line notation listing
$\sigma(i)$ below $i$,  $i\in\zahl{m}$.
}

A set with distinct elements
$x_1,\dots,x_n$ is written  either as~$\{x_1,\dots,x_n\}$ or~$\{x_i\}_{i=1}^{n}$.  The corresponding sequence is denoted 
by~$[x_1,\dots,x_n]$ or
$[x_i]_{i=1}^{n}$. 
An arrangement of $\{x_i\}_{i=1}^{n}$ 
is a sequence
$[x_{\sigma(i)}]_{i=1}^{n}$, where $\sigma\in\kS_n$.  An
$r$-arrangement 
is 
a sequence
$[x_{\sigma(i)}]_{i=1}^{r}$ for $r\in\zahl{n}$. 
Write $I^{n}_{r}$ for the family of all $(n)_r:=n!/(n-r)!$ possible~$r$-arrangements of $\zahl{n}$.

 {The set of nonnegative reals is denoted $\R_{\geq
    0}$, and $\bm0_d$ stands for the} origin in $\R^{d}$. For~$\bmu,\mv\in~\!\R^d$, we write $\bmu\preceq\mv$ if~$u_{\ell}\le v_{\ell}$
for all $\ell\in\zahl{d}$, { and $\bmu\not\preceq\mv$ otherwise}. 
Let $\Arc(\bmu,\mv):=(2\pi)^{-1}\arccos\{\bmu^{\top}\mv/\\(\lVert\bmu\rVert\lVert\mv\rVert)\}$ if $\bmu,\mv\ne\bm0_d$; 
 $\Arc(\bmu,\mv):=0$ otherwise.  {Here,  $\norm{\cdot}$ stands for the
 Euclidean norm.} For  vectors
$\mv_1,\ldots,\mv_k$, we use $(\mv_1,\ldots,\mv_k)$ as a shorthand for
$(\mv_1^\top,\ldots,\mv_k^\top)^\top$. We write $\fI_d$ for the ${d\times d}$ 
identity matrix. 
 For a function~$f:\cX\to \R$, we
define $\norm{f}_{\infty}:=\max_{x\in\cX}|f(x)|$.  The
symbols $\lfloor \cdot\rfloor$ and $\ind(\cdot)$  stand for the floor and indicator functions.  


The cumulative
distribution function and the probability distribution of a real-valued random variable/vector $\mZ$ are
denoted as $F_{\mZ}(\cdot)$ and ${\rm P}_{\mZ}$, respectively. 
The class of
probability measures on~$\R^d$ that are
absolutely continuous (with respect to the Lebesgue measure) is denoted as $\cP_d^{\ac}$.  We
use~$\rightsquigarrow$ and $\stackrel{\sf a.s.}{\longrightarrow}$ to
denote convergence in distribution and almost sure convergence,
respectively.  For any symmetric kernel $h(\cdot)$ on $(\R^d)^m$, any
integer $\ell\in\zahl{m}$,  and any probability measure $\Pr_{\mZ}$,
we write $h_{\ell}(\mz_1\ldots,\mz_{\ell}; \Pr_{\mZ})$ for $\E
h(\mz_1\ldots,\mz_{\ell},\mZ_{\ell+1},\ldots,\mZ_m)$ where
$\mZ_1,\ldots,\mZ_m$ are $m$ independent copies of
${\mZ}\sim\Pr_{\mZ}$, and $\E h:=\E h(\mZ_1,\ldots,\mZ_m)$.  The
product measure of two distributions $\P_1$ and~$\P_2$ is denoted
$\P_1\otimes\P_2$.


\section{Generalized symmetric covariances}
%
\label{sec:sgc}

Let $\mX_1$ and $\mX_2$ be two random vectors with values in
$\mathbbm{R}^{d_1}$ and $\mathbbm{R}^{d_2}$, respectively, and assume throughout this
paper that they are both absolutely continuous with respect to the
Lebesgue measure. \citet[Def.~3]{MR3842884} introduced a general
approach to defining rank-based measures of dependence
via 
signed sums of indicator functions
that are acted upon by  subgroups of the symmetric group.
In this section, we highlight that their resulting family of {\it symmetric rank
  covariances}
 can be extended to cover a much
wider range of dependence measures including, in particular, the
celebrated {\it distance covariance} \citep{MR2382665}.  This
enables us to handle a broad family of dependence
measures in the following common standard form.

\begin{definition}[Generalized symmetric covariance]\label{def:sc} 
{\rm 
  A  measure of dependence $\mu$ is said to be
  an $m$-th order 
{\it generalized symmetric covariance (GSC)} if there exist
  two kernel functions~$f_1:(\R^{d_1})^m\to\R_{\geq 0}$ and $f_2:(\R^{d_2})^m\to\R_{\geq 0}$,
  and a subgroup $H\subseteq\kS_m$ containing an equal number of even and odd permutations such that
\[
\mu(\mX_1,\mX_2)=\mu_{f_1,f_2,H}(\mX_1,\mX_2):= \E [k_{f_1,f_2,H}((\mX_{11},\mX_{21}),\dots,(\mX_{1m},\mX_{2m}))].
\]
Here $(\mX_{11},\mX_{21}),\dots,(\mX_{1m},\mX_{2m})$ are $m$
independent copies of $(\mX_1,\mX_2)$, and
the dependence kernel function $k_{f_1,f_2,H}(\cdot)$ is defined as
\begin{align*}
k_{f_1,f_2,H}&\Big((\mx_{11},\mx_{21}),\dots,(\mx_{1m},\mx_{2m})\Big)\\
:=\;&
\Big\{\sum_{\sigma\in H}\sgn(\sigma) f_1(\mx_{1\sigma(1)},\dots,\mx_{1\sigma(m)})\Big\}
\Big\{\sum_{\sigma\in H}\sgn(\sigma) f_2(\mx_{2\sigma(1)},\dots,\mx_{2\sigma(m)})\Big\}.
\yestag\label{eq:sc} 
\end{align*}
}
\end{definition}

As the group $H$ is required to have equal numbers of even and odd
permutations, the order of a GSC satisfies $m\ge 2$. This requirement
also justifies the term ``generalized covariance'' through the
following property; compare 
\citet[Prop.~2]{MR3842884}.


\begin{proposition}
  \label{prop:GSC-Iconsistency}
  All GSCs are I-consistent.  More precisely, the GSC
  $\mu_{f_1,f_2,H}(\mX_1,\mX_2)$   is I-consistent in the family
  of  distributions such that  
   $\E[
   f_k]:=\E[ f_k(\mX_{k1},\ldots,\mX_{km})]<\infty$, $k=1,2$,  {where
    $\mX_{k1},\ldots,\mX_{km}$ are $m$ independent copies
  of $\mX_k$}.
\end{proposition}

The concept of GSC unifies a surprisingly large number of well-known
dependence measures.  
We consider here  {five noteworthy   examples, namely,} the distance covariance of \cite{MR2382665} and
\cite{MR3053543}, the multivariate version of Hoeffding's $D$ based on
marginal ordering  \citep[Section~2.2, p.~549]{MR3842884}, and the
projection-averaging 
extensions of Hoeffding's $D$ 
\citep{MR3737307}, of Blum--Kiefer--Rosenblatt's $R$ \citep[Proposition~D.5]{MR4185814supp}, and of Bergsma--Dassios--Yanagimoto's $\tau^*$
\citep[Theorem.~7.2]{MR4185814}. 
 {Only one type of subgroup, namely,
  $H_*^m:=\langle(1~4),(2~3)\rangle\subseteq\kS_m$ for~$m\ge 4$ is needed; recall~\eqref{eq:H*m}.}
\textcolor{black}{For simplicity, 
 we write
  $\mw = (\mw_1,\ldots, \mw_m)\mapsto f_k (\mw)$   for the kernel
  functions of an  $m$th order multivariate GSC for which the
  dimension of $\mw _\ell$, $\ell =1,\ldots, m$, is~$d_k$, hence may
  differ for $k=1$ and~$k=2$.  Not all components of $\mw$ need to
  have an impact on $ f_k (\mw)$. For instance, the kernels  of
  distance covariance, a
  4th order GSC, map~$\mw=(\mw_1,\ldots,\mw_4)$ to
  $\R_{\geq 0}$ but depend neither on $\mw_3$ nor~$\mw_4$. }

\begin{example}[Examples of multivariate GSCs]\label{prop:sgc} \mbox{ }
\begin{enumerate}[itemsep=-.5ex,label=(\alph*),labelindent=0pt]
\item\label{example:dCov} Distance covariance is a 4th order GSC
  with $H=H_*^4$ and
  $$f_k^{\dCov}(\mw)=\frac12\lVert \mw_1-\mw_2\rVert\ \text{ on } \
  (\mathbbm{R}^{d_k})^4,\quad k=1,2.$$
   {Indeed, with $c_d:=\pi^{(1+d)/2}/\Gamma((1+d)/2)$, we have
\begin{align*}
&\mkern-20mu\mu_{f_1^{\dCov},f_2^{\dCov},H_*^4}(\mX_1,\mX_2)\\
&=\frac14
\E[(\lVert\mX_{11}-\mX_{12}\rVert
    -\lVert\mX_{11}-\mX_{13}\rVert-\lVert\mX_{14}-\mX_{12}\rVert
    +\lVert\mX_{14}-\mX_{13}\rVert)\\
& \mkern57mu
    \times (\lVert\mX_{21}-\mX_{22}\rVert
    -\lVert\mX_{21}-\mX_{23}\rVert-\lVert\mX_{24}-\mX_{22}\rVert
    +\lVert\mX_{24}-\mX_{23}\rVert)]\\
&=\frac{1}{c_{d_1}c_{d_2}}\int_{\bR^{d_1}\times \bR^{d_2}}
\frac{\lvert\varphi_{(\mX_1,\mX_2)}(\mt_1,\mt_2)-\varphi_{\mX_1}(\mt_1)\varphi_{\mX_2}(\mt_2)\rvert^2}
{\lVert\mt_1\rVert^{d_1+1}\lVert\mt_2\rVert^{d_2+1}}
{\rm d}\mt_1 {\rm d}\mt_2\yestag\label{eq:dCovID}.
\end{align*}
Identity \eqref{eq:dCovID} was established in 
\citet[Remark~3]{MR2382665},
\citet[Thm.~8]{MR2752127}, and
\citet[Sec.~3.4]{MR3178526};}
\item\label{example:WDM}   Hoeffding's multivariate marginal ordering $D$  
   is a 5th order GSC with $H=H_*^5$ and
  $$f_k^{M}(\mw)=\frac12\ind(\mw_1,\mw_2\preceq\mw_5)\ \text{ on } \
  (\mathbbm{R}^{d_k})^5,\quad k=1,2 {,}$$
 {
since, by  \citet[Prop.~1]{MR3842884},
}\begin{align*}
 {\mkern-15mu\mu_{f_1^{M},f_2^{M},H_*^5}(\mX_1,\mX_2)\!=\!\int_{\bR^{d_1}\times \bR^{d_2}}
\!\!\!\!\!\{F_{(\mX_1,\mX_2)}(\bmu_1,\bmu_2)\!-\!F_{\mX_1}(\bmu_1)F_{\mX_2}(\bmu_2)\}^2\d F_{(\mX_1,\mX_2)}(\bmu_1,\bmu_2);}
\end{align*}
\item\label{example:mD}  Hoeffding's multivariate projection-averaging $D$ 
   is a 5th order GSC with $H=H_*^5$ and
  $$f_k^{D}(\mw)=\frac12\Arc(\mw_1-\mw_5,\mw_2-\mw_5)\ \text{ on } \
  (\mathbbm{R}^{d_k})^5,\quad k=1,2.$$ 
 {Indeed, by \citet[Equation~(3)]{MR3737307}, we have
\begin{multline*}
\mu_{f_1^{D},f_2^{D},H_*^5}(\mX_1,\mX_2)
=\int_{\cS_{d_1-1}\times \cS_{d_2-1}}\int_{\bR^2}
\{F_{(\malpha_1^{\top}\mX_1,\malpha_2^{\top}\mX_2)}(u_1,u_2)\\
-F_{\malpha_1^{\top}\mX_1}(u_1)F_{\malpha_2^{\top}\mX_2}(u_2)\}^2 
\d F_{(\malpha_1^{\top}\mX_1,\malpha_2^{\top}\mX_2)}(u_1,u_2)
\d \lambda_{d_1}(\malpha_1)\d \lambda_{d_2}(\malpha_2),
\end{multline*}
with $\lambda_d$ the uniform measure on the unit sphere $\cS_{d-1}$;}
\item\label{example:mR}  Blum--Kiefer--Rosenblatt's multivariate projection-averaging $R$ is a 6th order GSC  with~$H=~\!H_*^6$ and
  \begin{align*}
    f_1^{R}(\mw)&=\frac12\Arc(\mw_1-\mw_5,\mw_2-\mw_5)\ \text{ on } \
              (\mathbbm{R}^{d_1})^6,\\
    f_2^{R}(\mw)&=\frac12\Arc(\mw_1-\mw_6,\mw_2-\mw_6)\ \text{ on } \
              (\mathbbm{R}^{d_2})^6 {;}
  \end{align*}
 {this follows from \citet[Prop.~D.5]{MR4185814supp}, who showed 
}\begin{multline*}
 {\mu_{f_1^{R},f_2^{R},H_*^6}(\mX_1,\mX_2)
=\int_{\cS_{d_1-1}\times \cS_{d_2-1}}\int_{\bR^2}
\{F_{(\malpha_1^{\top}\mX_1,\malpha_2^{\top}\mX_2)}(u_1,u_2)}\\
 {-F_{\malpha_1^{\top}\mX_1}(u_1)F_{\malpha_2^{\top}\mX_2}(u_2)\}^2 
\d F_{\malpha_1^{\top}\mX_1}(u_1)
\d F_{\malpha_2^{\top}\mX_2}(u_2)
\d \lambda_{d_1}(\malpha_1)\d \lambda_{d_2}(\malpha_2);}
\end{multline*}
\item\label{example:mBDY}   Bergsma--Dassios--Yanagimoto's
  multivariate projection-averaging $\tau^*$  is a 4th order GSC  with $H=~\!H_*^4$ and
  $$f_k^{\tau^*}(\mw)=\Arc(\mw_1-\mw_2,\mw_2-\mw_3)+\Arc(\mw_2-\mw_1,\mw_1-\mw_4)\ \text{ on } \
              (\mathbbm{R}^{d_k})^4,\quad k=1,2 {,}$$
 {since, by \citet[Theorem~7.2]{MR4185814}, we have
\begin{multline*}
\mu_{f_1^{\tau^*},f_2^{\tau^*},H_*^4}(\mX_1,\mX_2)
=\int_{\cS_{d_1-1}\times \cS_{d_2-1}}
\E \{a_{\sign}(\malpha_1^{\top}\mX_{11},\malpha_1^{\top}\mX_{12},\malpha_1^{\top}\mX_{13},\malpha_1^{\top}\mX_{14})\\
\times a_{\sign}(\malpha_2^{\top}\mX_{21},\malpha_2^{\top}\mX_{22},\malpha_2^{\top}\mX_{23},\malpha_2^{\top}\mX_{24})\}
\d \lambda_{d_1}(\malpha_1)\d \lambda_{d_2}(\malpha_2),
\end{multline*}
with $a_{\sign}(w_1,w_2,w_3,w_4):=\sign(|w_1-w_2|-|w_1-w_3|-|w_4-w_2|+|w_4-w_3|).$}
\end{enumerate}
\end{example} 


\begin{remark}\label{rmk:hsic}
 { \citet{MR3127866} recognize distance covariance as an example of an HSIC-type
  statistic
  \citep{gretton:aistats2005,MR2255909,MR2249882,MR2320675}.  The HSIC-type statistics are all 4th order
  multivariate GSCs, and
  we note that 
  our results for distance covariance readily extend to other HSIC-type statistics.}
\end{remark}

\begin{remark}\label{rmk:sgc:1d}
   {In the univariate case, the GSCs from
    Example~\ref{prop:sgc}\ref{example:WDM}--\ref{example:mBDY} reduce
  to the~$D$ of
\cite{MR0029139}, $R$ of \cite{MR0125690}, and $\tau^*$ of
\cite{MR3178526}, respectively. As shown by \cite{MR4185806}, the
latter is connected to the work of
\cite{yanagimoto1970measures}.  In Appendix~\ref{appendix:sectionB1},
we simplify the kernels for the univariate case, and show that the GSC
framework also covers the $\tau$ of \cite{Kendall1938}.
}
\end{remark}

All  {the multivariate} dependence
measures we have introduced are D-consistent, albeit with some
variations in  the  families of distributions for
which this holds; see, e.g., the discussions in Examples 2.1--2.3 of \citet{MR4185806}.  As these dependence
measures all involve the group $H_*^m$, we highlight the following fact.

\begin{lemma}\label{lem:kernel:kernel2}
  A GSC $\mu=\mu_{f_{1},f_{2},H_{*}^m}$ with $m\ge 4$ is
  D-consistent in a family $\mathcal{P}$ if and only if the pair~$(f_1,f_2)$ is D-consistent in $\mathcal{P}$---namely, if and only if
  \begin{align*}
  \E\Big[&\prod_{k=1}^{2}
        \Big\{f_{k}(\mX_{k1},\mX_{k2},\mX_{k3},\mX_{k4},\mX_{k5},\dots,\mX_{km})
              -f_{k}(\mX_{k1},\mX_{k3},\mX_{k2},\mX_{k4},\mX_{k5},\dots,\mX_{km})\\
             &-f_{k}(\mX_{k4},\mX_{k2},\mX_{k3},\mX_{k1},\mX_{k5},\dots,\mX_{km})
              +f_{k}(\mX_{k4},\mX_{k3},\mX_{k2},\mX_{k1},\mX_{k5},\dots,\mX_{km})\Big\}\Big]
\end{align*}
is finite,  nonnegative, and equal to $0$ only if $\mX_1$ and $\mX_2$ are independent.
\end{lemma}


\begin{theorem}\label{thm:kernel:kernel2}
 {All the} multivariate GSCs in  {Example}~\ref{prop:sgc}   are D-consistent within the family $\big\{{\rm P}\in\cP_{d_1+d_2}^{\ac} \big\vert\, \E_{\rm P}[f_k(\mX_{k1},\ldots,\mX_{km})]<\infty  ,\ k=1,2
  \big\}$ (with~$f_k$,~$k=1,2$ denoting their respective kernels).
\end{theorem}

The invariance/equivariance properties of GSCs depend on those of their kernels. We say that a kernel   function $f: (\R^d)^m\to\R$ is   {\it orthogonally invariant} if, for any orthogonal matrix~$\fO \in\R^{d\times d}$ and any~$\mw_1,\ldots,\mw_m\in~\!(\R^d)^m$,  
 $
f(\mw_{1},\dots,\mw_{m})=f(\fO\mw_{1},\dots,\fO\mw_{m}).
$

\begin{lemma}\label{lem:kernel:kernel1}
If $f_1$ and $f_2$ both  are orthogonally invariant, then 
any GSC of the  form~$\mu=\mu_{f_1,f_2,H}$ is orthogonally
invariant, i.e., 
 $
\mu(\mX_1,\mX_2) = \mu(\fO_1\mX_1, \fO_2\mX_2)
$ 
for any pair of random vectors $(\mX_1,\mX_2)$ and 
orthogonal matrices $\fO_1 \in\R^{d_1\times d_1}$  and~$\fO_2
\in\R^{d_2\times d_2}$.
\end{lemma}


\begin{proposition}\label{prop:kernel:kernel1}
  \label{prop:kernel:ortho:invariance} The kernels \ref{example:dCov},\ref{example:mD}--\ref{example:mBDY}  in  {Example}
  \ref{prop:sgc}, hence the corresponding GSCs, are
  orthogonally invariant.
\end{proposition}

Turning from theoretical dependence measures to their empirical counterparts, it is clear that 
any GSC admits a natural unbiased estimator in
the form of a U-statistic, which we call the {\it sample generalized symmetric
   covariance} (SGSC).

\begin{definition}[Sample  generalized symmetric covariance] \label{def:SGSC} 
{\rm 
The sample generalized symmetric covariance of $\mu=\mu_{f_1,f_2,H}$ is $\hat\mu\n= \hat\mu\n ([(\mx_{1i},\mx_{2i})]_{i=1}^n; f_1,f_2,H)$, of the form
\begin{align*}
\hat\mu\n= {n \choose m}^{-1}\sum_{i_1<i_2<\cdots<i_m}\overline k_{f_1,f_2,H}\Big((\mx_{1i_1},\mx_{2i_1}),\dots,(\mx_{1i_m},\mx_{2i_m})\Big),
\end{align*}
where $\overline k_{f_2,f_2,H}$ is the ``symmetrized'' version of $k_{f_2,f_2,H}$:
\begin{align*}
\overline k_{f_1,f_2,H}\Big(\big[(\mx_{1\ell},\mx_{2\ell})\big]_{\ell=1}^{m}\Big):=\frac{1}{m!}\sum_{\sigma\in\kS_{m}}
k_{f_1,f_2,H}\Big(\big[(\mx_{1\sigma(\ell)},\mx_{2\sigma(\ell)}\big]_{\ell=1}^{m}\Big).
\end{align*}
}
\end{definition}

If the kernels $f_1$ and $f_2$ are orthogonally invariant,
then it also holds that all  SGSCs of the form 
$\hat\mu\n ( \;\cdot\;; f_1,f_2,H)$ are 
orthogonally invariant, in the sense of remaining unaffected when
 the   input $[(\mx_{1i},\mx_{2i})]_{i=1}^n$ is transformed  into
$[(\fO_1\mx_{1i},\fO_2\mx_{2i})]_{i=1}^n$ where  $\fO_1 \in\R^{d_1\times d_1}$ and~$\fO_2
\in~\!\R^{d_2\times d_2}$ are arbitrary orthogonal matrices. 
Proposition~\ref{prop:kernel:ortho:invariance} thus also implies the 
orthogonal invariance of SGSCs associated with  kernels \ref{example:dCov} and \ref{example:mD}--\ref{example:mBDY}  in  {Example}~\ref{prop:sgc}.\medskip

The SGSCs associated with the examples listed in  {Example}~\ref{prop:sgc}, unfortunately, all fail to satisfy  the  crucial property of  distribution-freeness. However, as we will 
show in Section~\ref{sec:measure}, distribution-freeness, along with
 transformation invariance, can be
obtained by computing SGSCs from (functions of) the center-outward
ranks and signs of the observations.  

\section{Center-outward ranks and signs} \label{sec:hallin}

This section   briefly  introduces  the concepts of
center-outward ranks and signs to be used in the sequel.  The main purpose  
is to fix notation and terminology; for a 
comprehensive coverage, we refer to~\citet{hallin2020distribution}.

We are concerned with defining multivariate ranks for a sample of
$d$-dimensional observations drawn from a distribution in
 the class $\cP_d^{\ac}$ of absolutely continuous probability measures
on $\R^d$ with~$d\ge 2$. 
Let $\bS_d$ and $\cS_{d-1}$ 
denote the open unit ball and the unit sphere in $\R^d$, respectively. 
  Denote by~$\U_d$ the spherical uniform measure on~$\bS_d$, that is, 
  the product of the uniform measures on~$[0,1)$ (for the distance to the origin)
and on $\cS_{d-1}$ (for the direction).  
 The push-forward of a measure~$\Q$ by a measurable transformation~$T$   
is denoted as $T\sharp\Q$.

\begin{definition}[Center-outward distribution
  function] \label{def:centerdistr}  
  {\rm 
  The \emph{center-outward distribution function} of a probability
  measure $\P\in\cP_d^{\ac}$ is the $\P$-a.s. unique function $\fF_{\pms}$
  that (i) maps
  $\R^d$ to the open unit ball $\bS_d$, (ii) is the gradient of a
  convex function on $\R^d$, and (iii) pushes~$\P$ forward
  to $\U_d$ (i.e., such that $\fF_{\pms}\sharp\P=\U_d$).
}
\end{definition}

 {The center-outward distribution  function ${\bf F}_{\pms}$ of $\rm P$ entirely characterizes $\rm P$ provided that~$\rm P\in \cP_d^{\ac}$; cf.~\citet[Prop.~2.1(iii)]{hallin2020distribution}. 
Also, $\fF_{\pms}$ is invariant under shift, global rescaling, and orthogonal transformations.
We refer the readers to 
Appendix~\ref{appendix:sectionB2} for details about  these elementary properties of center-outward distribution functions.}

The sample counterpart $\fF_{\pms}\n$ of
$\fF_{\pms}$ is based on an $n$-tuple of data points $\mz_1,\dots,\mz_n\in\R^d$. 
The key idea is to construct $n$ grid points in the unit ball $\bS_d$
such that the corresponding discrete uniform distribution converges weakly to 
$\U_d$ as $n\to\infty$.
For $d\geq 2$, the 
   construction  proposed  in
\citet[Sec.~4.2]{hallin2017distribution}  
%
 starts by factorizing $n$ into 
\[
  n = n_R n_S+n_0,\qquad
  n_R,n_S\in\Z_{>0},\qquad 0\le n_0<\min\{n_R,n_S\}, 
\]
where in asymptotic scenarios $n_R$ and $n_S\to\infty$, hence $n_0/n\to0$, as $n\to\infty$.
Next consider   the intersection points between 
\begin{itemize}[itemsep=-.5ex,label=--]\label{dsp:grid}
\item the $n_R$ hyperspheres centered at $\bm0_d$, with radii $r/(n_R+1)$, $r\in\zahl{n_R}$, and
\item   $n_S$  {rays given by} distinct unit vectors
  $\{\ms_{s}^{(n_S)}\}_{s\in\zahl{n_S}}$ that divide the unit circle
  into arcs of equal length~$2\pi/n_S$ for $d=2$, and are
  distributed as regularly as possible on the unit sphere $\cS_{d-1}$
  for $d\ge~\!3$; asymptotic statements merely require 
 that the discrete uniform distribution
 over~$\{\ms_{s}^{(n_S)}\}_{s=1}^{n_S}$ converges weakly  to the
 uniform distribution on $\cS_{d-1}$ as $n_S\to\infty$. 
\end{itemize}
Letting $\mn:=(n_R,n_S,n_0)$, the grid
$\kG_{\mn}^{d}$ is defined as the set 
 of  $n_Rn_S$ 
  points
$\big\{\frac{r}{n_R+1}\ms_{s}^{(n_S)}\big\}$ with~$r\in\zahl{n_R}$ and $s\in\zahl{n_S}$
as described above  {along with,} whenever $n_0>1$,  the 
 $n_0$  points~$\big\{\frac{1}{2(n_R+1)}\ms_{s}^{(n_S)}\big\}$, ${s\in \cS}$ where $\cS$ is chosen as a random sample  {of size $n_0$} without replacement from $\zahl{n_S}$. 
 For $d=1$, letting $n_S=2$, $n_R=\lfloor n/n_S\rfloor$, $n_0=n-n_Rn_S=0$ or 1, 
 $\kG^{d}_{\mn}$ reduces to
the points $\{\pm r/(n_R+1): r\in\zahl{n_R}\}$, along with the origin $0$ in case $n_0=1$.  


The empirical version $\fF_{\pms}^{(n)}$ 
of $\fF_{\pms}$ is then defined as the optimal coupling between the
observed data points and the  grid
$\kG_{\mn}^{d}$.

\begin{definition}[Center-outward ranks and signs]
\label{def:empdistr}
{\rm 
Let $\mz_1,\dots,\mz_n$ be distinct data points in $\R^d$.
Let~$\cT$ be the collection of all bijective mappings between the set $\{\mz_{i}\}_{i=1}^{n}$ and the  grid~$\kG^{d}_{\mn} {=\{\bmu_{i}\}_{i=1}^{n}}$. The \emph{sample center-outward distribution function} is defined as
\begin{equation}\label{eq:assignment}
\fF_{\pms}^{(n)}:=\argmin_{T\in\cT}\sum_{i=1}^{n}\Big\lVert\mz_{i}-T(\mz_{i})\Big\rVert^2,
\end{equation}
and $(n_R+1)\lVert {\bf F}^{(n)}_{\pms}(\mz_i)\rVert$ and ${\bf
  F}^{(n)}_{\pms}(\mz_i)/\lVert {\bf F}^{(n)}_{\pms}(\mz_i)\rVert$ are
called the \emph{center-outward rank} and \emph{center-outward sign}
of $\mz_i$, respectively.
}
\end{definition}

\begin{remark}\label{remark:31}  {The particular 
way the grid $\kG_{\mn}^{d}$ is constructed here produces
center-outward ranks and signs
that enjoy all the properties --- uniform distributions and mutual
independence --- that are expected from ranks and signs (see Section
\ref{appendix:sectionB2} of the online Appendix). These properties,
however, are not required for the finite-sample validity and
asymptotic properties of the rank-based tests we are pursuing in the
subsequent sections. Any sequence of grids $\kG_{\mn}^{d}$, whether
stochastic (defined over a different probability space than the observations) or deterministic, is fine provided that the corresponding empirical distribution converges to the spherical uniform ${\rm U}_d$. In addition, for the reasons developed, e.g., in \cite{hallin2021measure}, we deliberately only consider the spherical uniform ${\rm U}_d$. In practice, the uniform distribution over the unit cube $[0,1]^d$ could be considered as well, yielding similar tests enjoying similar properties, with proofs following along similar lines.}
\end{remark}




 {The next proposition describes the Glivenko--Cantelli property of
  empirical center-outward distribution functions, a result we shall
  heavily rely on.}

\begin{proposition}\label{prop:Hallin4}
(\citealp[Proposition~5.1]{hallin2017distribution},
\citealp[Theorem~3.1]{del2018smooth},
and
\citealp[Proposition~2.3]{hallin2020distribution})
 {Consider the following classes of distributions:
\begin{itemize}[itemsep=-.5ex]
\item the class $\cP_d^{+}$  of distributions $\P\in\cP_d^{\ac}$
  with {\it nonvanishing probability density}, namely,  with Lebesgue density $f$ such that, for all $D>0$ there exist constants
  $\lambda_{D;f}<\Lambda_{D;f}\in(0,\infty)$ such that
  $\lambda_{D;f}\le f(\mz)\le \Lambda_{D;f}$ for all
  $\lVert\mz\rVert\le D$;
%
\item the class  $\cP_d^{\#}$ of all distributions $\P\in\cP_d^{\ac}$ 
such that, denoting by $\fF_{\pms}\n$ the sample distribution function computed from an $n$-tuple $\mZ_1,\ldots,\mZ_n$ of  independent copies of $\mZ\sim\P$,  
\begin{align}\label{eq:GC-new}
\max_{1\le i\le n}\Big\lVert\fF_{\pms}\n(\mZ_{i})-\fF_{\pms}(\mZ_{i})\Big\rVert\stackrel{\sf a.s.}{\longrightarrow}0\quad\text{as $n_R$ and $n_S\to\infty$. }
\end{align}
\end{itemize}
  It holds that
  $\cP_d^{+}\subsetneq\cP_d^{\#}\subsetneq\cP_d^{\ac}$. }
\end{proposition}

\section{Rank-based dependence measures}\label{sec:measure}

We are now ready to present our proposed family of 
dependence measures based on the notions of GSCs and center-outward
ranks and signs.
Throughout, 
$(\mX_1,\mX_2)$ is   a pair of random vectors with~$\P_{\mX_1}\in\cP_{d_1}^{\ac}$ and~$\P_{\mX_2}\in~\!\cP_{d_2}^{\ac}$, and  
$
(\mX_{11},\mX_{21}), (\mX_{12},\mX_{22}), \ldots, (\mX_{1n},\mX_{2n})
$ 
is an $n$-tuple of independent copies of~$(\mX_1,\mX_2)$. Let $\fF_{k,\pms}$ denote the center-outward distribution function of $\mX_k$, and write $\fF_{k,\pms}^{(n)}(\cdot)$ for the sample center-outward distribution function corresponding to  $\{\mX_{ki}\}_{i=1}^n$,   $k=1,2\vspace{1mm}$. 

Our ideas build on \citet{shi2019distribution} and, in slightly
different form, also on \citet{deb2019multivariate},  where the authors introduce 
a multivariate dependence measure  {by applying distance covariance to
$\fF_{1,\pms}(\mX_1)$ and $\fF_{2,\pms}(\mX_2)$, with a sample
counterpart  involving~$\fF_{1,\pms}^{(n)}(\mX_{1i})$ and
$\fF_{2,\pms}^{(n)}(\mX_{2i})$, $i\in\zahl{n}$.}
Our generalization of this particular
dependence measure involves  {\it score functions} and   requires further 
notation.  The  score functions  
are continuous
functions  {$J_1,J_2:[0,1)\to\R_{\geq 0}$}. Classical examples  include the {\it normal} or {\it van der Waerden score function} $J_{\text{\tiny{\rm vdW}}}(u):=\big(F_{\chi^2_d}^{-1}(u)\big)^{1/2}$ (with $F_{\chi^2_d}$ the $\chi^2_d$ distribution function), 
 the {\it Wilcoxon score function} $J_{\text{\tiny{\rm W}}}(u):=u$, and the {\it sign test score function} $J_{\text{\tiny{\rm sign}}}(u):=1$. 
For~$k=1,2$, let  $\fJ_k(\bmu):=J_k(\lVert\bmu\rVert)\bmu/\lVert \bmu\rVert$ if $\bmu\in\bS_{d_k}\backslash\{\bm0_{d_k}\}$ and $\bm0_{d_k}$ if $\bmu=\bm0_{d_k}$. 
Define the population and sample \emph{scored center-outward distribution functions}  as 
$\fG_{k,\pms}(\cdot):=\fJ_k(\fF_{k,\pms}(\cdot))$ and $\fG^{(n)}_{k,\pms}(\cdot):=\fJ_k(\fF^{(n)}_{k,\pms}(\cdot))$, 
respectively. 

\begin{definition}[Rank-based dependence measures]
  \label{def:rGSC}
     {Let  $J_1,J_2$ be two score
      functions.  The {\it (scored) rank-based version}
      of a dependence measure $\mu$ is obtained by applying $\mu$ to
      the pair $(\fG_{1,\pms}(\mX_1),\fG_{2,\pms}(\mX_2))$.  For a GSC
      $\mu=\mu_{f_1,f_2,H}$, the rank-based version is denoted
\begin{equation}\label{GSCpm}
\mu_{\pms}(\mX_1,\mX_2)=\mu_{\pms;J_1,J_2,f_1,f_2,H}(\mX_1,\mX_2):=\mu_{f_1,f_2,H}(\fG_{1,\pms}(\mX_1),\fG_{2,\pms}(\mX_2))
\end{equation}
and termed a {\it rank-based GSC} for short.  The associated {\it
  rank-based SGSC} is
\begin{equation}\label{Wpm}
\tenq{W}\n_{\mu}=\tenq{W}\n_{J_1,J_2,\mu_{f_1,f_2,H}}:= \hat\mu\n \Big(\big[\big(\fG^{(n)}_{1,\pms}(\mX_{1i}),\fG^{(n)}_{2,\pms}(\mX_{2i})\big)\big]_{i=1}^n; f_1,f_2,H\Big).
\end{equation}}
\end{definition}


\begin{remark}
   {
  There is no immediate reason why a rank-based GSC should
  itself by a GSC in the sense of Definition~\ref{def:sc}. In this
  context, an observation of
  \citet{bergsma2006new,bergsma2011nonparametric} is of interest.  For
  distance covariance in the univariate case} 
  (equivalent to
   $4\kappa$ in his notation),  
   Lemma 10 in \cite{bergsma2006new} implies that
\[\frac1{16}\mu_{f_1^{\rm dCov},f_2^{\rm dCov},H_*^4}(\fG_{X_1,\pms}(X_1),\fG_{X_2,\pms}(X_2))
  =\int (F_{(X_1,X_2)}-F_{X_1}F_{X_2})^2 \d F_{X_1}\d F_{X_2}.\]
In other words, for $d_1=d_2=1$ and~$J_1(u)=J_2(u)=u$, the rank-based
distance covariance coincides with $R$ of \citet{MR0125690} up to a
scalar multiple.  Recall that $R$ is 
a GSC, but of higher order than distance covariance; see  {Example~\ref{prop:sgc:1d}\ref{example:R} in
  Appendix~\ref{appendix:sectionB1}}.
\end{remark}

Plugging the   center-outward ranks and signs into the multivariate
dependence measures from Section \ref{sec:sgc} in combination
with various score functions, one immediately obtains  a large variety
of rank-based GSCs and SGSCs, as we exemplify below.   {In particular, the choice}~$f_1=f_1^{\rm dCov}$,
$f_2=f_2^{\rm dCov}$, $J_1(u)=J_2(u)=u$, and $H=H_*^4$   {recovers the
multi\-variate}~rank-based distance covariance  {from
\cite{shi2019distribution}}.


\begin{example}\label{ex:rbsgc} Some rank-based  SGSCs.
\mbox{}
\begin{enumerate}[itemsep=-.5ex,label=(\alph*)]
\item\label{ex:rbsgc-dCov} Rank-based distance covariance 
\[
\tenq{W}\n_{{\dCov}} :=\mbinom{n}{4}^{-1}
\sum_{i_1<\dots<i_4}h_{\dCov}\Big(\big(\fG_{1,\pms}^{(n)}(\mX_{1i_1}),\fG_{2,\pms}^{(n)}(\mX_{2i_1})\big),\dots,\big(\fG_{1,\pms}^{(n)}(\mX_{1i_4}),\fG_{2,\pms}^{(n)}(\mX_{2i_4})\big)\Big)
\]
with $h_{\dCov}:=\overline k_{f_1^{\dCov},f_2^{\dCov},H_{*}^4}$ as given in  {Example}~\ref{prop:sgc}\ref{example:dCov}.  {We have by definition that}
\begin{align*}
 {\tenq{W}\n_{{\dCov}}\! =
\binom{n}{4}^{-1}\!\!\!\!\sum_{i_1\ne \dots\ne i_4}\!\!\frac1{4\cdot 4!}
\Big[\Big\{
   }& {\big\lVert\fG_{1,\pms}^{(n)}(\mX_{1i_1})-\fG_{1,\pms}^{(n)}(\mX_{1i_2})\big\rVert
    -\big\lVert\fG_{1,\pms}^{(n)}(\mX_{1i_1})-\fG_{1,\pms}^{(n)}(\mX_{1i_3})\big\rVert}\\
 {-}& {\big\lVert\fG_{1,\pms}^{(n)}(\mX_{1i_4})-\fG_{1,\pms}^{(n)}(\mX_{1i_2})\big\rVert
    +\big\lVert\fG_{1,\pms}^{(n)}(\mX_{1i_4})-\fG_{1,\pms}^{(n)}(\mX_{1i_3})\big\rVert
    \Big\}}\\
 {\times\Big\{}
    & {\big\lVert\fG_{2,\pms}^{(n)}(\mX_{2i_1})-\fG_{2,\pms}^{(n)}(\mX_{2i_2})\big\rVert
    -\big\lVert\fG_{2,\pms}^{(n)}(\mX_{2i_1})-\fG_{2,\pms}^{(n)}(\mX_{2i_3})\big\rVert}\\
 {-}& {\big\lVert\fG_{2,\pms}^{(n)}(\mX_{2i_4})-\fG_{2,\pms}^{(n)}(\mX_{2i_2})\big\rVert
    +\big\lVert\fG_{2,\pms}^{(n)}(\mX_{2i_4})-\fG_{2,\pms}^{(n)}(\mX_{2i_3})\big\rVert
    \Big\}\Big];}
\end{align*}
\item  {Similarly, Hoeffding's rank-based multivariate marginal ordering   $D$ 
\big(giving $\tenq{W}\n_{M}$\big),
Hoeffding's rank-based multivariate   projection-averaging    $D$ 
\big($\tenq{W}\n_{D}$\big),
Blum--Kiefer--Rosenblatt's rank-based multivariate projection-averaging     $R$ 
\big($\tenq{W}\n_{R}$\big), and
Bergsma--Dassios--Yanagimoto's rank-based multivariate projection-averaging $\tau^*$ 
\big($\tenq{W}\n_{\tau^*}$\big) 
can be defined with kernels
$h_{M}:=\overline k_{f_1^{M},f_2^{M},H_{*}^5}$, 
$h_{D}:=\overline k_{f_1^{D},f_2^{D},H_{*}^5}$, 
$h_{R}:=\overline k_{f_1^{R},f_2^{R},H_{*}^6}$, and 
$h_{\tau^*}:=\overline k_{f_1^{\tau^*},f_2^{\tau^*},H_{*}^4}$ 
as given in  {Example}~\ref{prop:sgc}, respectively.}
\end{enumerate}
\end{example}

Having proposed a general class of dependence measures, we now
examine, for each rank-based GSC, the five desirable properties listed 
in Section~\ref{sec:desiredsec}. To this end, we
first introduce two regularity conditions on the score functions. 

\begin{definition}\label{asp:score}
{\rm 
A score function $J:[0,1)\to \R_{\ge 0}$ is called \emph{weakly regular} if it is continuous over~$[0,1)$ and nondegenerate: $\int_{0}^{1}J^2(u)\d u>0$. 
If, moreover, $J$ is  Lipschitz-continuous, strictly monotone, and satisfies $J(0)=0$, it is called  \emph{strongly regular}. 
}
\end{definition}

\begin{proposition}\label{prop:score}
The {normal} and {sign test} score functions
  are weakly but not strongly regular; the {Wilcoxon} score function is strongly regular. 
\end{proposition}


\begin{proposition}\label{prop:sGSC}
  {Suppose the considered pair $(\mX_1,\mX_2)$ has marginal distributions~$\P_{\mX_1}\in\cP_{d_1}^{\ac}$ and $\P_{\mX_2}\in\cP_{d_2}^{\ac}$.
 Con\-sider any rank-based
 GSC~$\mu_{\pms}:=\mu_{\pms;J_1,J_2,f_1,f_2,H}$ and its rank-based
 SGSC~$\tenq{W}\n_\mu :=\tenq{W}\n_{J_1,J_2,\mu_{f_1,f_2,H}}$ as
 defined in \eqref{GSCpm} and~\eqref{Wpm}.  Further,
 let~$\mu_{*\pms}:=\mu_{\pms;J_1,J_2,f_1,f_2,H_*^{m}}$ be an instance
 using the group from~\eqref{eq:H*m}.} Then, 
\begin{enumerate}[itemsep=-.5ex,label=(\roman*)]
\item \label{prop:sGSC1} (Exact distribution-freeness) Under independence of $\mX_1$ and $\mX_2$, the distribution of~$\tenq{W}\n_{{\mu}}$ does not depend on $\P_{\mX_1}$ nor $\P_{\mX_2}$;
\item \label{prop:sGSC2} (Transformation invariance) If the kernels $f_1$ and $f_2$ 
are   orthogonally invariant, it holds for any orthogonal matrix $\fO_k\in\R^{d_k\times d_k}$, any vector
$\mv_k\in\R^{d_k}$, and any scalar $a_k\in\R_{> 0}$ that
$\mu_{\pms}(\mX_1,\mX_2) = \mu_{\pms}\big(\mv_1+a_1\fO_1\mX_1,\mv_2+a_2\fO_2\mX_2\big)$; 
\item \label{prop:sGSC3} (I- and D-Consistency) 
\begin{enumerate}[itemsep=-.5ex,label=(\alph*)]
\item $\mu_{\pms}$ is I-consistent in the family 
\[
\big\{\P_{(\mX_1,\mX_2)}\big\vert\, \P_{\mX_k}\in\cP_{d_k}^{\ac}\text{\,\,\rm and }
\E \big[f_k\big([\fG_{k,\pms}(\mX_{ki})]_{i=1}^{m}\big)\big]<\infty\text{ for }k=1,2\big\};
\]
\item If the pair of kernels is D-consistent in the class $$
\big\{\P_{(\mX_1,\mX_2)}\in\cP_{d_1+d_2}^{\ac}\big\vert\, \E \big[f_k (\mX_{k1},\ldots,\mX_{km})\big]<\infty\text{ for }k=1,2\big\}$$ (cf.~Lemma~\ref{lem:kernel:kernel2}), then $\mu_{*\pms}
$ is D-consistent in the family 
\[
\cP_{d_1,d_2,\infty}^{\ac} := \big\{\P_{(\mX_1,\mX_2)}\in \cP_{d_1+d_2}^{\ac}\big\vert\,  \E \big[f_k\big([\fG_{k,\pms}(\mX_{ki})]_{i=1}^{m}\big)\big]<\infty\text{ for }k=1,2\big\}
\yestag\label{eq:Dconsifam}
\]
provided that the score functions $J_1$ and $J_2$ are strictly monotone; 
\end{enumerate}
\item \label{prop:sGSC4} (Strong consistency) If 
$f_k\big([\fG_{k,\pms}^{(n)}(\mX_{ki_\ell})]_{\ell=1}^{m}\big)$ and 
$f_k\big([\fG_{k,\pms}(\mX_{ki_\ell})]_{\ell=1}^{m}\big)$ 
 are almost surely bounded, that is, if  there exists a constant $C$ (depending on $f_k$, $J_k$, and $\P_{{\bf X}_k}$) such that for any $n$ and $k=1,2$,  
 $$\P\Big(\, \big\vert f_k\big(\big[\fG_{k,\pms}^{(n)}(\mX_{ki_\ell})\big]_{\ell=1}^{m}\big) \big\vert \leq C\Big)=1=\P\Big(\, \big\vert f_k\big(\big[\fG_{k,\pms}(\mX_{ki_\ell})\big]_{\ell=1}^{m}\big) \big\vert \leq C\Big),$$
   and 
\begin{equation}
(n)_{m}^{-1}\sum_{[i_1,\dots,i_m]\in I_{m}^{n}}
\Big\lvert f_k\Big(\big[\fG_{k,\pms}^{(n)}(\mX_{ki_\ell})\big]_{\ell=1}^{m}\Big)
          -f_k\Big(\big[\fG_{k,\pms}(\mX_{ki_\ell})\big]_{\ell=1}^{m}\Big)\Big\rvert
\stackrel{\sf a.s.}{\longrightarrow} 0, 
\label{condition:as}
\end{equation}
 then
\[
\tenq{W}\n_{{\mu}}=\tenq{W}\n_{J_1,J_2,\mu_{f_1,f_2,H}}
\stackrel{\sf a.s.}{\longrightarrow}\mu_{\pms}(\mX_1,\mX_2).     
\yestag\label{eq:testconsi}
\]
\end{enumerate}
\end{proposition}

\begin{theorem}[Examples]\label{thm:sGSC5}  
As long as $\P_{\mX_1}\in \cP_{d_1}^{\#}$, $\P_{\mX_2}\in \cP_{d_2}^{\#}$, and $J_1, J_2$ are strongly regular, all the kernel functions in  {Example}~\ref{prop:sgc}\ref{example:dCov}--\ref{example:mBDY} satisfy Condition \eqref{condition:as}.
\end{theorem}


\begin{remark} \label{rmk:feuerverger}
 {Unfortunately, Theorem~\ref{thm:sGSC5}  does not imply that the
  rank-based SGSCs  with  normal score
functions satisfy
\eqref{eq:testconsi} although, in view of Proposition
\ref{prop:sGSC}\ref{prop:sGSC3}, their population counterparts are both I-
and D-consistent within a fairly large nonparametric
family of distributions.  A weaker version (replacing a.s. convergence
by convergence in probability) of \eqref{eq:testconsi} holds in the
univariate case with~$d_1=d_2=1$ by
\citet[Sec.~6]{feuerverger1993}. Consistency for normal scores,
however, follows from a recent and yet unpublished result of
\citet[Proposition 4.3]{deb2021efficiency}, which was not available to
us at the time this paper was written and which is obtained via a completely different technique.} 
\end{remark}

We conclude this section with a discussion of computational
issues. Two steps, in the evaluation of  multivariate rank-based SGSCs, are potentially 
costly: (i) calculating the center-outward ranks and signs in \eqref{eq:assignment}, and (ii) computing a~GSC $\hat\mu\n (\cdot)$ 
 with $n$ inputs. 
The optimal matching problem  \eqref{eq:assignment} yielding $[\fG_{1,\pms}\n(\mX_{1i})]_{i=1}^{n}$ and
$[\fG_{2,\pms}\n(\mX_{2i})]_{i=1}^{n}$  
 {can be solved in~$O(n^{5/2}\log(nN))$ time if 
 the 
  costs $\lVert\mz_i-\bmu_j\rVert^2$,~$i,j\in\zahl{n}$ are
  integers bounded by $N$ \citep{MR1015271};  in
  dimension~$d=2$, this can improved to $O(n^{3/2+\delta}\log(N))$
  time for some arbitrarily small constant $\delta>0$
  \citep{MR3205219}.
  The problem can also be solved approximately in 
$O(n^{3/2}\Omega(n,\epsilon,\Delta))$ time if $d\ge3$, 
where 
$$\Omega(n,\epsilon,\Delta):=\epsilon^{-1}\tau(n,\epsilon)\log^{4}(n/\epsilon)\log(\Delta)$$ depends on $n$, $\epsilon$ (the accuracy of the approximation) and $\Delta:={\max c_{ij}}/{\min c_{ij}}$, with~$\tau(n, \epsilon)$ a small term \citep{MR3238983}.
Further details are deferred to
Appendix~\ref{appendix:sectionB3}.}

Once  $[\fG_{1,\pms}^{(n)}(\mX_{1i})]_{i=1}^{n}$ and
$[\fG_{2,\pms}^{(n)}(\mX_{2i})]_{i=1}^{n}$ are obtained, a na\"{\i}ve
 {evaluation} of~$\tenq{W}\n\!$, on the other hand,  requires 
$O(n^m)$ operations. 
Great speedups are possible, however, in particular
cases 
 {such as the rank-based SGSCs from} Example~\ref{ex:rbsgc}.
%
%
 {A detailed summary is provided in
  Proposition~\ref{thm:correlation-computation} of the Appendix.  The
  total computational complexity of the five statistics in
  Example~\ref{ex:rbsgc} is given in the last three rows of Table~\ref{tab:propofsGSC}.}

\section{ {Local power of rank-based tests of independence}} \label{sec:test}
 
Besides quantifying  the dependence between two groups of
random variables, the  {rank-based GSCs} from Section
\ref{sec:measure} allow for constructing tests of the null hypothesis 
\begin{align*}
H_0: \mX_1 \text{ and }\mX_2 \text{ are mutually independent},
\end{align*}
based on a sample $(\mX_{11},\mX_{21}),\dots,(\mX_{1n},\mX_{2n})$ of~$n$ independent copies of $(\mX_1,\mX_2)$.
\citet{shi2019distribution}, and, in a slightly different manner,  
\citet{deb2019multivariate}, studied the particular  case of  a test based on the Wilcoxon version of the   {rank-based} distance covariance $\tenq{W}\n_{{\dCov}}$. 
%
%
Among other results, they derive the limiting null distribution of
$\tenq{W}\n_{{\dCov}}$,
using
combinatorial limit theorems and ``brute-force'' calculation of
permutation statistics.  Although this led to a fairly general 
combinatorial non-central limit theorem \citep[Theorems~4.1 and~4.2]{shi2019distribution},
the derivation is not intuitive and difficult to generalize. 
 {In contrast,}
in this paper, we take a  {new} and more powerful approach to the 
asymptotic analysis of rank-based SGSCs, which
resolves the following three main  issues:
\begin{enumerate}[itemsep=-.5ex,label=(\roman*)]
\item Intuitively, the asymptotic behavior of  rank-based
  dependence measures follows from that of their H\'{a}jek {\it
    asymptotic representations}, which are oracle versions in which
  the observations are transformed using the unknown actual center-outward
  distribution function~${\bf F}_{\pms}$ 
  rather than its sample version ${\bf F}\n_{\pms}$.  Here, we show
the correctness of this intuition by proving asymptotic equivalence
between  {rank-based SGSCs} and their oracle versions. 
\item Previous work does not perform any power analysis for the new
   rank-based tests.  Here, we fill this gap by proving that these 
  tests  {have nontrivial power} in the context of 
  the  {class of quadratic mean differentiable alternatives \citep[Def.~12.2.1]{MR2135927}}.
\item Finally, our rank-based tests allow for the   incorporation of
  score functions, which  {may improve} their performance.   
%
\end{enumerate}
This novel approach 
rests on a
generalization of the classical 
H\'{a}jek  representation method \citep{MR0229351} to the multivariate
setting of center-outward ranks and signs, which simplifies the derivation of
asymptotic null distributions 
and,  via a nontrivial use of Le Cam's third lemma for
 non-normal limits, enables our local
power analysis. 
%

\subsection{
Asymptotic representation}\label{sec:representation}

In order to develop our multivariate asymptotic representation, we
first introduce formally the oracle counterpart to the
rank-based SGSC $\tenq W\n_{\mu}$.

\begin{definition}[Oracle  {rank-based SGSCs}] 
{\rm 
The oracle
   version of the rank-based SGSC $\tenq W\n_{J_1,J_2,\mu_{f_1,f_2,H}}$
  associated with  the GSC 
  $\mu=\mu_{f_1,f_2,H}$    is 
\[
W\n_{{\mu}}=W\n_{J_1,J_2,\mu_{f_1,f_2,H}}:= \hat\mu\n \Big(\big[\big(\fG_{1,\pms}(\mX_{1i}),\fG_{2,\pms}(\mX_{2i})\big)\big]_{i=1}^n; f_1,f_2,H\Big).
\]
}
\end{definition}

Note that the
oracle~$W\n_{\mu}$ 
cannot be computed from the observations as it involves the population scored center-outward distribution functions~$\fG_{1,\pms}$ and $\fG_{2,\pms}$. 
 However, the limiting null
distribution of~$W\n$, unlike that of~$\tenq{W}\n$,    
 follows from 
standard theory for degenerate U-statistics
\citep[Chap.~5.5.2]{MR595165}. This point can be summarized as follows.

\begin{proposition}\label{prop:infea}
  Let $\mu=\mu_{f_1,f_2,H_{*}^{m}}$ be a GSC with $m\ge 4$. 
   Let the kernels 
   $f_1,f_2$ and the score functions~$J_1,J_2$
  satisfy
\[
0<\Var(g_{k}(\mW_{k1},\mW_{k2}))<\infty,~~~k=1,2,
\yestag\label{eq:uvar}
\]
where $\mW_{ki}:=\fJ_{k}(\mU_{ki})$ with  $(\mU_{1i},\mU_{2i})$,
$i\in\zahl{m}$  independent and distributed according to the product
of spherical uniform
distributions $\U_{d_1}\otimes \U_{d_2}$, 
\begin{equation}\label{eq:suey}
g_{k}(\bm{w}_{k1},\bm{w}_{k2}):=
\E \Big[2f_{k,H_{*}^{m}}\Big(\bm{w}_{k1},\bm{w}_{k2},\mW_{k3},\mW_{k4},\dots,\mW_{km}\Big)\Big],
\end{equation}
and
$f_{k,H_{*}^{m}}:=\sum_{\sigma\in H_{*}^{m}}\sgn(\sigma) f_{k}(\mx_{k\sigma(1)},\dots,\mx_{k\sigma(m)})$, $k=1,2$.  
  Then,  under the null hypothesis~$H_0$ that ${\mX_1}\sim\P_{\mX_1}\in\cP_{d_1}^{\ac}$ and
  ${\mX_2}\sim\P_{\mX_2}\in\cP_{d_2}^{\ac}$ are independent, 
\[
n  W\n_{\mu}=nW\n_{J_1,J_2,\mu_{f_1,f_2,H_{*}^{m}}}
\rightsquigarrow  
\sum_{v=1}^{\infty}\lambda_{\mu,v}(\xi_v^2-1),
\]
where $[\xi_v]_{v=1}^\infty$ are independent standard Gaussian random
variables and $[\lambda_{\mu,v}]_{v=1}^\infty$ are the non-zero eigenvalues of the integral equation
\[
\E\big[g_1(\bm{w}_{11},\mW_{12})g_2(\bm{w}_{21},\mW_{22})
         \psi(\mW_{12},\mW_{22})\big]
=\lambda \psi(\bm{w}_{11},\bm{w}_{21}).
\yestag\label{eq:eigen}
\]
\end{proposition}

The tests we are considering reject for large values of test
statistics that 
estimate a nonnegative (I- and
D-)consistent dependence measure.  In all these tests
\[
\text{all eigenvalues of the integral equation \eqref{eq:eigen} are non-negative.} 
\yestag\label{eq:pd}
\]
However, it should be noted that, in view of the following  multivariate representation result, a valid test of $H_0$ can be implemented  also   when  \eqref{eq:pd} does not hold.


 
\begin{theorem}[Multivariate  {H\' ajek} 
  representation]\label{thm:smallop}
  Let $f_1,f_2$ be kernel functions of order~$m\ge~\!4$, and let
  $J_1$, $J_2$ be weakly regular score functions.  Writing
  $\U_{d_k}^{(n)}$ for
  the discrete uniform distribution over the grid $\kG_{\mn}^{d_k}$, let $\mW_{ki}^{(n)}:=\fJ_{k}(\mU_{ki}^{(n)})$ where 
$(\mU_{1i}^{(n)},\mU_{2i}^{(n)})$ for~$i\in\zahl{m}$ are independent
with distribution $\U_{d_1}^{(n)}\otimes \U_{d_2}^{(n)}$. 
  Define  $g_k$, $k=1,2$,
 as in \eqref{eq:suey}, and 
\[
g_{k}^{(n)}(\bm{w}_{k1},\bm{w}_{k2}):=
\E \Big[2f_{k,H_{*}^{m}}\Big(\bm{w}_{k1},\bm{w}_{k2},\mW_{k3}^{(n)},\mW_{k4}^{(n)},\dots,\mW_{km}^{(n)}\Big)\Big],\quad k=1,2.
\yestag\label{eq:WDMlemma}
\]
Assume that 
\begin{align*}
&\text{
$f_{k}$ and $g_{k}$ are Lipschitz-continuous,~~~ 
$g_{k}^{(n)}$ converges uniformly to $g_{k}$,} \yestag\label{eq:key-condition}\\
&
\sup_{i_1,\dots,i_m\in\zahl{m}}\E[f_k([\mW_{ki_\ell}]_{\ell=1}^{m})^2]<\infty,\quad\text{and}\quad 
\int_{0}^{1}J_k^2(u)\d u<\infty, \quad k=1,2.
\end{align*}
  Then,  under the hypothesis $H_0$ that ${\mX_1}\sim\P_{\mX_1}\in\cP_{d_1}^{\ac}$ and
  ${\mX_2}\sim\P_{\mX_2}\in\cP_{d_2}^{\ac}$ are independent,  the
  rank-based SGSC $\tenq{W}\n_{\mu}=\tenq{W}\n_{J_1,J_2,\mu}$ associated to the GSC 
$\mu=\mu_{f_1,f_2,H_{*}^{m}}$ is asymptotically equivalent to its
oracle version $W\n_{\mu}$, i.e., $\tenq{W}\n_{\mu} - W\n_{\mu}=o_{\Pr}(n^{-1})$ as $n_R,n_S\to\infty$.\end{theorem}

\begin{theorem}\label{thm:smallop2}
The conclusion of Theorem~\ref{thm:smallop} still holds with \eqref{eq:key-condition} replaced by
\[\text{$f_k$ is uniformly bounded, and almost everywhere continuous},\quad k=1,2.
\yestag\label{eq:key-condition2}\]
\end{theorem}

\begin{proposition}[Examples]\label{ex:smallop}
If $\mX_1\!\sim\!\P_{\mX_1}\!\in\!\cP_{d_1}^{\ac}$ is independent of 
$\mX_2\!\sim\!\P_{\mX_2}\!\in~\!\cP_{d_2}^{\ac}$  and $J_1,J_2$ are
weakly regular, then the kernel functions from
 {Example}~\ref{prop:sgc}\ref{example:WDM}--\ref{example:mBDY} satisfy~\eqref{eq:uvar},\,\eqref{eq:pd},\,and\,\eqref{eq:key-condition2}.  
If, moreover,   $J_1,J_2$  are    square-integrable (viz.,~$\int_{0}^{1}J_k^2(u)\d u<\infty$ for~$k=1,2$),  
then  \eqref{eq:uvar}, {{\eqref{eq:pd}},}
and~\eqref{eq:key-condition} hold also for 
the kernels in  {Example}~\ref{prop:sgc}\ref{example:dCov}.  
\end{proposition}


\begin{corollary}[Limiting null distribution]\label{thm:null}
Suppose the conditions in Proposition~\ref{prop:infea} and Theorem~\ref{thm:smallop} hold. Then, for $\mu=\mu_{f_1,f_2,H_{*}^{m}}$ with $m\ge 4$,  under the hypothesis~$H_0$ that ${\mX_1}\sim\P_{\mX_1}\in\cP_{d_1}^{\ac}$ and
  ${\mX_2}\sim\P_{\mX_2}\in\cP_{d_2}^{\ac}$ are independent,
\[
n
\tenq{W}\n_{\mu}=n\tenq{W}\n_{J_1,J_2,\mu_{f_1,f_2,H_{*}^{m}}}
\rightsquigarrow  
\sum_{v=1}^{\infty}\lambda_{\mu,v}(\xi_v^2-1)
\yestag\label{eq:limitnulldistrtilde}
\]
with $[\lambda_{\mu,v}]_{v=1}^{\infty}$ and $[\xi_v]_{v=1}^{\infty}$
 as defined in Proposition~\ref{prop:infea}. 
\end{corollary}

\begin{remark} \label{rmk:berry-esseen}
 {Corollary~\ref{thm:null} gives no rate, i.e., no Berry--Ess\' een type bound
 for the convergence in~\eqref{eq:limitnulldistrtilde}.    Indeed,
 deriving such bounds in the present context is quite
 challenging.  Results for the univariate case with $d_1=d_2=1$ were
 established  for simpler statistics such as Spearman's $\rho$ and
 Kendall's $\tau$ by \citet[Chap.~6.2]{MR1472486} and, more recently,
 by \citet{MR3486424}. Extending these results to the multivariate
 measure-transportation-based ranks considered here 
 is highly nontrivial and requires properties of empirical transports
 that have not yet been obtained.  This pertains, in particular, to
 working out
 the rate of
convergence in the Glivenko--Cantelli result for the center-outward
distribution function given in~\eqref{eq:GC-new};  
an open problem in the recent survey by  \citet[Section~5]{hallin2021measure}.
}
\end{remark}


For any significance level $\alpha\in(0,1)$, define the quantile
\[
q_{\mu,1-\alpha}:=\inf\Big\{x\in\R:\P\Big(\sum_{v=1}^{\infty}\lambda_{\mu,v}(\xi_v^2-1)\le x\Big)\ge 1-\alpha \Big\},
\yestag\label{eq:testinfea3}
\]
where $[\lambda_{\mu,v}]_{v=1}^{\infty}$ and $[\xi_v]_{v=1}^{\infty}$
are as in Proposition~\ref{prop:infea}. Let
$\tenq{W}\n_{\mu}$ be as in Theorem \ref{thm:smallop}, and define
the test
\[
  \mathsf{T}\n_{\mu,\alpha}:= \ind\big(n \tenq{W}\n_{\mu} >
q_{\mu,1-\alpha}\big).
\]
The next proposition summarizes the  asymptotic validity and
properties of this test.

\begin{proposition}[Uniform validity and consistency]\label{prop:uvc}
  Let $J_1,J_2$ be weakly regular score functions, and let $\mu=\mu_{f_1,f_2,H_{*}^{m}}$
  be a GSC with $m\ge 4$ such that Conditions \eqref{eq:uvar} 
  and one of 
  \eqref{eq:key-condition} and \eqref{eq:key-condition2} hold.    Then, \vspace{-2mm}
\begin{enumerate}[itemsep=-.5ex]
\item[(i)]  $\lim_{n\to \infty}
  \Pr(\mathsf{T}\n_{\mu,\alpha}=1)=\alpha$ 
  for any $\P\in \cP^{\ac}_{d_1}\otimes \cP^{\ac}_{d_2}$, i.e., for $\mX_1$ and $\mX_2$ independent with~$\mX_1\sim\P_{\mX_1}\in\cP^{\ac}_{d_1}$ and 
   $\mX_2\sim\P_{\mX_2}\in\cP^{\ac}_{d_2}$; 
\item[(ii)] it follows from Proposition \ref{prop:sGSC}\ref{prop:sGSC1} that 
$
\lim_{n\to \infty}\sup_{\P\in\cP^{\#}_{d_1}\otimes\cP^{\#}_{d_2}}\Pr(\mathsf{T}\n_{\mu,\alpha}=1)=\alpha;
$ 
\item[(iii)] if, moreover,   the pair of kernels $(f_1,f_2)$ is D-consistent,
$J_1, J_2$ are strictly monotone, and~\eqref{eq:testconsi} holds,  
$\lim_{n\to \infty}\Pr(\mathsf{T}\n_{\mu,\alpha}=1)=1$ for any fixed alternative 
$\P_{(\mX_1,\mX_2)}\in\cP_{d_1,d_2,\infty}^{\ac}$ as defined in \eqref{eq:Dconsifam}. 
\end{enumerate}\end{proposition}

\subsection{Local power analysis}\label{sec:local-power}

In this section,  {we conduct local power analyses of the proposed
  tests for quadratic mean differentiable  classes of alternatives
  \citep[Def.~12.2.1]{MR2135927}, 
for which  we establish nontrivial power in $n^{-1/2}$ neighborhoods. 
We begin with a model 
$\{q_{\mX}(\mx;\delta)\}_{|\delta|<\delta^*}$ with $\delta^*>0$, 
under which $\mX=(\mX_1,\mX_2)$ has
Lebesgue-density~$q_{\mX}(\mx;\delta)=q_{(\mX_1,\mX_2)}\big((\mx_1,\mx_2);\delta\big)$,
with $q_{\mX_1}(\mx_1;\delta)$ and $q_{\mX_2}(\mx_2;\delta)$ being the marginal densities. 
We then make the following assumptions.}

\begin{assumption} \label{asp:only}\mbox{}
 { 
\begin{enumerate}[itemsep=-.5ex,label=(\roman*)]
\item\label{asp:only1} Dependence of  
 $\mX_1$ and $\mX_2$: 
  $q_{\mX}(\mx;\delta)=q_{\mX_1}(\mx_1;\delta)
  q_{\mX_2}(\mx_2;\delta)$ holds if and only if $\delta=0$. 
\item\label{asp:only2}
  The
family $\{q_\delta(\mx)\}_{|\delta|<\delta^*}$ is quadratic mean differentiable at $\delta=0$ with score function~$\dot{\ell}(\cdot;0)$, that is,
\[
\int\Big(\sqrt{q_{\mX}(\mx;\delta)}-\sqrt{q_{\mX}(\mx;0)}
-\frac12\delta\dot{\ell}(\mx;0)\sqrt{q_{\mX}(\mx;0)}\Big)^2{\rm d}\mx
= o(\delta^2)
\quad\text{as} \;\; \delta\to0.
\]
\item\label{asp:only3} The Fisher information is positive, i.e., $\cI_{\mX}(0)
:=\int \{\dot{\ell}(\mx;0)\}^2q_{\bf X}(\mx,0){\rm d}\mx>0$; of note, 
Assumption~\ref{asp:only}\ref{asp:only2} implies that  
$\cI_{\mX}(0)<\infty$  and~$\int \dot{\ell}(\mx;0) q_{\bf X}(\mx,0){\rm d}\mx=0$. 
\item\label{asp:only4} The score function $\dot{\ell}(\mx;0)$ is not additively separable, i.e., there
  do not exist functions~$h_1$ and $h_2$ such that~$\dot{\ell}(\mx;0)=h_1(\mx_1)+h_2(\mx_2)$.
\end{enumerate}}
\end{assumption}

\begin{remark}
 {For the sake of simplicity, we have restricted ourselves to one-parameter classes. Analogous results hold for families indexed by a multivariate parameter $\mdelta$.}
\end{remark}

For a local power analysis, we consider a sequence of local  alternatives
 {obtained as
\begin{equation}\label{eq:local-alternative}
H_{1}\n(\delta_0):\delta=\delta\n, ~~~\text{where}~\delta\n:= n^{-1/2}\delta_0
\end{equation}
with some constant $\delta_0\ne 0$.} In this local model, testing the null
hypothesis of independence reduces to testing 
$
H_0: \delta_0 = 0$ versus $H_1: \delta_0\ne 0
$. 

\begin{theorem}[Power analysis]\label{thm:power-general}
  Consider a GSC $\mu=\mu_{f_1,f_2,H_{*}^{m}}$ with $m\ge 4$ and kernel
  functions $f_1,f_2$ picked from  {Example}~\ref{prop:sgc}. Assume
  that $J_1,J_2$ are
  weakly regular score functions that satisfy the assumptions of Proposition \ref{ex:smallop}. 
  Then if  {Assumption~\ref{asp:only}} holds, 
  for any~$\beta>0$,
  there exists a 
  constant 
  $C_\beta>0$ depending only on $\beta$ such that, as long as~$|\delta_0|>C_\beta$, 
$
\lim_{n\to\infty}\P\big\{{\mathsf{T}}\n_{\mu,\alpha}=1\big\vert
H_{1}\n(\delta_0)\big\}\ge1-\beta$.
\end{theorem}


%

 {Following the arguments from the proof of Theorem~\ref{thm:power-general}, 
one should be able to obtain similar local power results for the original 
(non-rank-based) tests associated with the kernels listed in
 {Example}~\ref{prop:sgc}. However, to the best of our knowledge,
this analysis has not been performed in the literature,  except for
$d_1=d_2=1$ where results can be found, e.g.,  in \citet{MR3466185} and \cite{shi2020power}.
We also emphasize that, although Theorem \ref{thm:power-general} only
considers the specific cases listed also in Example~\ref{ex:rbsgc},
the proof technique applies more generally. We refrain, however, from
stating a more general version of Theorem~\ref{thm:power-general} as
this would require a number of tedious technical conditions.  }

Combined with the following result, Theorem~\ref{thm:power-general} yields  {nontrivial power of the proposed tests in $n^{-1/2}$ neighborhoods of $\delta=0$}.

\begin{landscape}
{
\renewcommand{\tabcolsep}{4pt}
\renewcommand{\arraystretch}{1}
\begin{table}
\begin{center}
\caption{Properties of the center-outward GSCs in Example~\ref{ex:rbsgc} with   weakly regular score functions $J_k$.}\label{tab:propofsGSC}
\begin{tabular}{cccC{1.4in}C{1.2in}C{1.2in}C{1.2in}C{1.2in}}
\toprule
\multicolumn{3}{c}{$\tenq{W}\n_{\mu}$}                             
        & $\tenq{W}\n_{\dCov}$ 
        & $\tenq{W}\n_{M}$ 
        & $\tenq{W}\n_{D}$ 
        & $\tenq{W}\n_{R}$ 
        & $\tenq{W}\n_{\tau^*}$ \\
\midrule
  (1)   & \multicolumn{2}{C{1.2in}}{Distribution-freeness} 
        & $\P_{(\mX_1,\mX_2)}\in\cP^{\ac}_{d_1}\otimes \cP^{\ac}_{d_2}${\color{blue}$^{(a)}$}
        & $\P_{(\mX_1,\mX_2)}\in\cP^{\ac}_{d_1}\otimes \cP^{\ac}_{d_2}$ 
        & $\P_{(\mX_1,\mX_2)}\in\cP^{\ac}_{d_1}\otimes \cP^{\ac}_{d_2}$ 
        & $\P_{(\mX_1,\mX_2)}\in\cP^{\ac}_{d_1}\otimes \cP^{\ac}_{d_2}$ 
        & $\P_{(\mX_1,\mX_2)}\in\cP^{\ac}_{d_1}\otimes \cP^{\ac}_{d_2}$ \\[2em]
  (2)   & \multicolumn{2}{C{1.2in}}{Transformation invariance}     
        & Orthogonal transf., shifts, and global scales 
        & Shifts and global scales
        & Orthogonal transf., shifts, and global scales  
        & Orthogonal transf., shifts, and global scales  
        & Orthogonal transf., shifts, and global scales  \\[2em]
  (3)   & \multicolumn{2}{C{1.2in}}{D-consistency}        
        & $J_k$  strictly monotone ~~~~~~ and integrable, 
          $\P_{(\mX_1,\mX_2)}\in\cP^{\ac}_{d_1+d_2}${\color{blue}$^{(b)}\!\!\!$} 
        & $J_k$  strictly monotone, 
          $\P_{(\mX_1,\mX_2)}\in\cP^{\ac}_{d_1+d_2}$  
        & $J_k$  strictly monotone, 
          $\P_{(\mX_1,\mX_2)}\in\cP^{\ac}_{d_1+d_2}$  
        & $J_k$  strictly monotone, 
          $\P_{(\mX_1,\mX_2)}\in\cP^{\ac}_{d_1+d_2}$  
        & $J_k$  strictly monotone, 
          $\P_{(\mX_1,\mX_2)}\in\cP^{\ac}_{d_1+d_2}$  \\[2.5em]
  (3')  & \multicolumn{2}{C{1.2in}}{Consistency of test}  
        & $J_k$  strongly regular, 
          $\P_{(\mX_1,\mX_2)}\in\cP^{\#}_{d_1,d_2}${\color{blue}$^{(c)}$}
        & $J_k$  strongly regular, 
          $\P_{(\mX_1,\mX_2)}\in\cP^{\#}_{d_1,d_2}$
        & $J_k$  strongly regular, 
          $\P_{(\mX_1,\mX_2)}\in\cP^{\#}_{d_1,d_2}$
        & $J_k$  strongly regular, 
          $\P_{(\mX_1,\mX_2)}\in\cP^{\#}_{d_1,d_2}$
        & $J_k$  strongly regular, 
          $\P_{(\mX_1,\mX_2)}\in\cP^{\#}_{d_1,d_2}$ \\[2.5em]
  (4)   & \multicolumn{2}{C{1.2in}}{Efficiency}            
        & $J_k$  square-integrable 
        & $J_k$  weakly regular (as assumed)
        & $J_k$  weakly regular (as assumed) 
        & $J_k$  weakly regular (as assumed) 
        & $J_k$  weakly regular (as assumed) \\[2em]
        & \multirow{2}{*}{Exact} 
        & $d_1\vee d_2=2$  
        & $O(n^2)$ 
        & 
          {$O(n^{3/2+\delta}\log N)${\color{blue}$^{(d)}\!\!\!$}}
        & $O(n^3)$
        & $O(n^4)$
        & $O(n^4)$\\[1em]
\multirow{2}{*}
 {(5)}  &                        
        & $d_1\vee d_2=3$  
        & 
          {$O(n^{5/2}\log(nN))${\color{blue}$^{(d)}$}}
        & $O(n^{5/2}\log(nN))$
        & $O(n^3)$
        & $O(n^4)$
        & $O(n^4)$\\[1em]
        & \multicolumn{2}{C{1.2in}}{Fast approximation}    
        & $O(n^{3/2}\Omega\vee nK\log n)${\color{blue}$^{(d)}$} 
        & $O(n^{3/2}\Omega)$ 
        & $O(n^{3/2}\Omega\vee nK\log n)$ 
        & $O(n^{3/2}\Omega\vee nK\log n)$ 
        & $O(n^{3/2}\Omega\vee nK\log n)$ \\
\bottomrule
\end{tabular}
\begin{minipage}{8.3in}
\vspace{0.15cm}
{\color{blue}$^{(a)}$} $\cP^{\ac}_{d_1}\otimes \cP^{\ac}_{d_2}$ is the family of all
  $\P_{(\mX_1,\mX_2)}$ such that $\mX_1,\mX_2$ independent, $\P_{\mX_1}\in\cP^{\ac}_{d_1}$
  and $\P_{\mX_2}\in\cP^{\ac}_{d_2}$ \\
{\color{blue}$^{(b)}$} $\cP_{d_1+d_2}^{\ac}$ is the family of all absolutely continuous distributions on $\R^{d_1+d_2}$ \\
{\color{blue}$^{(c)}$} $\cP^{\#}_{d_1,d_2}:=  \big\{\P_{(\mX_1,\mX_2)}\in\cP_{d_1+d_2}^{\ac}\vert\,  \P_{\mX_1}\in\cP_{d_1}^{\#}, \P_{\mX_2}\in\cP_{d_2}^{\#}\big\}$ \\
{\color{blue}$^{(d)}$} Here we assume without loss of generality that
$c_{ij},~i,j\in\zahl{n}$ are all integers and bounded by integer $N$, 
$\delta$ is some arbitrarily small constant, 
$\Omega$ is defined as $\epsilon^{-1}\tau(n,\epsilon)\log^{4}(n/\epsilon)\log({\max c_{ij}}/{\min c_{ij}})$,
and $K$ is sufficiently large;  as usual, 
$q_1\vee q_2$ stands for the minimum of two quantities $q_1$ and $q_2$.
 {Also refer to 
Propositions \ref{thm:matching-computation} and \ref{thm:correlation-computation} in Section~\ref{sec:help} of the appendix.} 
\vspace{-0.5cm}
\end{minipage}
\end{center}
\end{table}
}
\end{landscape}

\begin{theorem}
\label{thm:opt}
 {Let Assumption~\ref{asp:only} hold. Then, }
for any   $\beta>0$ such that~$\alpha+\beta<~\!1$, 
there exists an absolute constant $c_\beta>0$ 
 such that, as long as $|\delta_0|\le c_\beta$, 
\[
\inf_{\overline{\mathsf{T}}\n_{\alpha}\in\cT\n_{\alpha}}\Pr\big\{\overline{\mathsf{T}}\n_{\alpha}=0\big\vert H_{1}\n(\delta_0)\big\}\geq 1-\alpha-\beta
\]
for all sufficiently large $n$. 
Here the infimum is taken over the class $\cT\n_{\alpha}$ of all size-$\alpha$ tests. 
\end{theorem}

Table \ref{tab:propofsGSC} summarizes our results for the rank-based
SGSCs from Example~\ref{ex:rbsgc}
by giving  an
overview of the five properties listed in the
Introduction.  It also indicates consistency of the tests.  In all
cases, it is assumed that the  score functions involved are weakly
regular.

\subsection{Examples in the quadratic mean differentiable class} \label{sec:example-new}
 {This section presents two specific examples in the quadratic mean
  differentiable class that satisfy Assumption \ref{asp:only}. First,}
we consider parametrized families that extend the bivariate  {\it
  Konijn alternatives} \citep{MR79384}. These alternatives are
classical in the context of testing for multivariate  independence and
have also been considered by  \citet{MR2691505}, \citet{MR1467849},
\citet{MR1965367,MR2088309}, \citet{MR2201019}, and \citet{MR2462206}. 

%
%
%

Konijn families are constructed as follows. Let $\mX ^*_1\sim \P_{\mX ^*_1}\in\cP_{d_1}^{\ac}$ and $\mX ^*_2\sim \P_{\mX ^*_2}\in\cP_{d_2}^{\ac}$ be two (without loss of generality) mean zero (unobserved) independent random vectors with densities $q_1$ and $q_2$, respectively. 
Let $\fG ^*_{1,\pms}$ and $\fG ^*_{2,\pms}$ denote their respective population scored center-outward distribution functions, 
$\P _{\mX ^*}\in\cP_{d_1+d_2}^{\ac} $   their joint distribution,~$q_{\mX ^*}(\mx)=q_{\mX ^*}((\mx_1,\mx_2))=q_1(\mx_1)q_2(\mx_2)$  their 
joint density. 
Define, for~$\delta\in~\!\R$, 
\begin{equation}\label{Kondef}
\mX = \bigg(\begin{matrix}\mX_1 \\ \mX_2\end{matrix}\bigg)
:=\bigg(\begin{matrix}\fI_{d_1} & \delta\, \fM_1\\
                      \delta\, \fM_2 & \fI_{d_2}\end{matrix}\bigg) 
              \bigg(\begin{matrix}\mX ^*_1 \\ \mX ^*_2\end{matrix}\bigg)
= \fA_{\delta}\bigg(\begin{matrix}\mX ^*_1 \\ \mX ^*_2\end{matrix}\bigg)
= \fA_{\delta}\mX ^*,
\end{equation}
where  $\fM_1\in\R^{d_1\times d_2}$ and
$\fM_2\in\R^{d_2\times d_1}$ are two deterministic matrices.  For
$\delta=0$, the
matrix~$\fA_\delta$ is the identity and, thus, invertible. 
By continuity, $\fA_\delta$ is also invertible for $\delta$ in a sufficiently small neighborhood $\Theta$ of $0$. 
 For $\delta\in\Theta$, the density of
$\mX$ can be expressed as~$
q_{\mX}(\mx;\delta)=\big|\det(\fA_{\delta})\big|^{-1}q_{\mX ^*}(\fA_{\delta}^{-1}\mx)$, 
which is differentiable with respect to $\delta$. 
The following additional assumptions will be made on the generating
scheme \eqref{Kondef}. 

\begin{assumption}\label{asp:power-opt}\mbox{}
\begin{enumerate}[itemsep=-.5ex,label=(\roman*)]
\item\label{asp:power-opt1} The distributions of $\mX$ have a common support for all $\delta\in\Theta$. Without loss of generality,  we assume ${\bm\cX}:=\{\mx: q_{\mX}(\mx;\delta) > 0\}$ does not  depend on  $\delta$.
\item\label{asp:power-opt2} 
The map ${\mx}\mapsto \sqrt{q_{\mX ^*}({\mx})}$ is continuously differentiable.
\item\label{asp:power-opt3} 
The Fisher information $\cI_{\mX}(0):=\int \{\dot{\ell}(\mx;0)\}^2 q_{\mX}(\mx;0){\rm d}\mx$ of $\mX$ relative to $\delta$ at~$\delta=0$ is strictly positive and finite.
\end{enumerate}
\end{assumption}


\begin{example}\label{ex:ellip}%
\mbox{}$\,$\vspace{-2mm}
\begin{enumerate}[itemsep=-.5ex,label=(\roman*)]
\item\label{ex:ellip1}   {Suppose} ${\mX ^*_1}$ and ${\mX ^*_2}$ are elliptical with centers $\bm0_{d_1}$ and
  $\bm0_{d_2}$ and   covariances  $\mSigma_1$ and $\mSigma_2$, respectively, that is, 
$
  q_k(\mx_k)
  \propto
  \phi_k\Big(\mx_k^{\top}\mSigma_k^{-1}\mx_k\Big)$,~$k=1,2$, 
  where $\phi_k$ is 
   such that~$\Var(\mX ^*_k)=\mSigma_k$ and 
  $
  \E\left[\lVert\mZ ^*_k \rVert^2
          \rho_k(\lVert\mZ ^*_k\rVert^2)^2\right] <\infty$, 
$k=1,2$ 
 where $\mZ ^*_k$ has density
  function proportional to $\phi_k(\lVert\mz_k\rVert^2)$ and  $\rho_k(t):=\phi_k^{\prime}(t)/\phi_k(t)$.  
   {Then Assumption~\ref{asp:power-opt} is satisfied for any $\fM_1,\fM_2$ such that 
  $\mSigma_1\fM_2^{\top}+\fM_1\mSigma_2\ne \bm0$.}
\item\label{ex:ellip2}   {As a specific example of \ref{ex:ellip1}, if} 
  $\mX ^*_1$ and $\mX ^*_2$ are centered
  multivariate normal or follow centered multivariate $t$-distributions with
  degrees of freedom 
  strictly
  greater than
  two, 
  then Assumption~\ref{asp:power-opt} is satisfied for any~$\fM_1,\fM_2$ such that 
  $\mSigma_1\fM_2^{\top}+\fM_1\mSigma_2\ne \bm0$. 
\end{enumerate}
\end{example}

 {Next,  
consider the following mixture model extending the alternatives treated in
\citet[Sec.~3]{MR3466185}.  Let $q_{1}$ and $q_{2}$ be fixed 
(Lebesgue-)density functions for~$\mX_1$ an $\mX_2$, respectively. The joint density of $\mX=(\mX_1,\mX_2)$ under 
 independence is $q_{1}q_{2}$. Letting $q^{*}\ne q_{1}q_{2}$ denote a fixed joint density, 
 mixture alternatives indexed by $\delta\in[0,1]$ are defined as 
$q_{\mX}(\mx;\delta):=(1-\delta)q_{1}q_{2}+\delta q^*$. 


\begin{assumption} \label{asp:far}
It is assumed that \vspace{-2mm}
\begin{enumerate}[itemsep=-.5ex,label=(\roman*)]
\item \label{asp:far1}
$(1+\delta^*)q_{1}q_{2}-\delta^* q^*$ is a bonafide joint density for some $\delta^*>0$; 
\item \label{asp:far2}
$q^{*}$ and $q_{1}q_{2}$ are mutually absolutely continuous; 
\item \label{asp:far3}
the function $\delta \mapsto \sqrt{q_{\mX}(\mx;\delta)}$ is continuously differentiable in some neighborhood of $0$;
\item \label{asp:far4}
the Fisher information 
$\cI_{\mX}(\delta):=\int (q^{*}-q_1q_2)^2/\{(1-\delta)q_{1}q_{2}+\delta q^*\} {\rm d}\mx$ 
of $\mX$ relative to $\delta$ is finite, strictly positive, and continuous at~$\delta=0$; 
\item \label{asp:far5}
$\dot{\ell}(\mx;0)=q^*(\mx)/\{q_{1}(\mx_1)q_{2}(\mx_2)\}-1$ is 
not additively separable. 
\end{enumerate}
\end{assumption}

\begin{example}\label{ex:mixture}
If $q_k(\mx_k)=1$ for $\mx_k\in[0,1]^{d_k}$, $k=1,2$, 
and $q^*(\mx)\not\equiv1$ is continuous and supported on $[0,1]^{d_1+d_2}$,
then Assumption~\ref{asp:far} holds.
\end{example}

\begin{proposition}\label{rmk:imply1}
Assumption \ref{asp:only}  is satisfied by the Konijn alternatives
under Assumption \ref{asp:power-opt}, and by the mixture alternatives under Assumption \ref{asp:far}.
\end{proposition}
}

\subsection{Numerical experiments}\label{sec:simulation}




 {Extensive simulations of \cite{shi2019distribution} 
  give evidence for the superiority, under non-Gaussian densities, of
  the Wilcoxon versions of our tests over the original
  distance covariance tests. That superiority is more substantial when
  non-Wilcoxon scores, such as the Gaussian ones, are considered (Figure \ref{fig:cauchy}). In view of these results, there is little
  point in pursuing simulations with non-Gaussian densities, and we
  instead focus on Gaussian cases (Figures
  \ref{fig:power1a}--\ref{fig:power1c}) to study the impact on
  finite-sample performance of the dimensions $d_1$ and $d_2$,
  sample size $n$, and within- and between-sample correlations.
 }


\begin{example}\label{ex:sim-power1}
The data are a sample of $n$ independent copies of the multivariate normal vector $(\mX_1,\mX_2)$ in $\R^{d_1+d_2}$, with mean zero and covariance matrix 
$\mSigma$, 
where 
$$\Sigma_{ij}=\Sigma_{ji}=\left\{
\begin{array}{rl}1,&i=j,\\
\tau,&i=1,\, j=2,\\
\rho,&i=1,\, j=d_1+1,\\
0,&\text{otherwise}.
\end{array}
\right.
$$\color{black}
 {Here $\tau$ characterizes the within-group correlation and we consider} (a) $\tau=0$, (b) $\tau=0.5$, and (c) $\tau=0.9$. Independence holds if and only if $\rho$, a between-group correlation, is zero. 
\end{example}
\begin{example}\label{ex:sim-power2}  {The data are $n$ independent copies of $(\mX_1,\mX_2)$ with $X_{1i}=Q_{t(1)}(\Phi(X_{1i}^*))$ and $X_{2j}=Q_{t(1)}(\Phi(X_{2j}^*))$ for $i\in\zahl{d_1}$ and $j\in\zahl{d_2}$; here $Q_{t(1)}$ denotes the quantile function of the standard Cauchy distribution and $(\mX_1^*, \mX_2^*)$ is generated according to Example \ref{ex:sim-power1}(b).}
\end{example}

We compare 
the empirical performance of the following  {five} tests: \vspace{-2mm}
\begin{enumerate}[itemsep=-.5ex,label=(\roman*)]
\item permutation test using the original distance covariance
  \citep{MR3053543}; 
\item permutation test applying original distance
  covariance to marginal ranks \citep{lin2017copula};
\item  center-outward  rank-based   distance covariance test 
with Wilcoxon scores and critical values from  the asymptotic
  distribution \citep{shi2019distribution};
\item new center-outward  rank-based   distance covariance test 
with normal scores and critical values from  the asymptotic
  distribution;
\item  {likelihood ratio test in the Gaussian model  (\citealp[Chap.~9.3.3 \& 8.4.4]{MR1990662}).}
\end{enumerate}
\color{black}
 {The parametric test (v) is tailored for Gaussian densities and
  plays the role of a benchmark.
  Unsurprisingly, in the Gaussian experiments in Figures \ref{fig:power1a}--\ref{fig:power1c}, it uniformly outperforms tests~(i)-(iv). See Figure \ref{fig:cauchy} for its unsatisfactory performance for non-Gaussian densities.}

Figures \ref{fig:power1a}--\ref{fig:cauchy}  report empirical powers 
(rejection frequencies) of these  {five} tests, based on~$1,000$
simulations with nominal significance level $0.05$, dimensions
$d_1=d_2\in\{2,3,5,7\}$, and sample size
$n\in\{216,432,864,1728\}$. The parameter $\rho$ in the true covariance
matrix takes values  $\rho\in\{0,0.005,\dots,0.15\}$. \color{black}The critical values for tests (i) and
(ii)  were computed on the basis of  $n$ random
permutations. 
\color{black}
 {For tests (iii) and (iv), to determine the critical values from the asymptotic distribution given in Corollary \ref{thm:null}, we numerically compute the eigenvalues by adopting the same strategy as in \citet[Sec.~5.2]{shi2019distribution}; see also \citet[p.~3291]{MR3127883}.}

 {It is evident from  Figure~\ref{fig:cauchy} that, in non-Gaussian experiments, the potential benefits of rank-based tests are huge, particularly so when Gaussian scores are adopted (note the very severe bias of the Gaussian likelihood ratio test as $d$ increases). In} Gaussian experiments, the
performance of the normal score--based test (iv) is uniformly better than
that of its Wilcoxon score counterpart~(iii); that superiority   increases  with the dimension  and decreases with  the within-group
dependence $\tau$. The superiority of  both    center-outward  rank-based tests (iii) and (iv)  over the  traditional distance covariance one and its marginal rank version is quite significant for high values of the within-group correlation $\tau$.

%



 {The way the normal-score rank-based test (and also the
  Wilcoxon-score one) outperforms the original distance covariance
  test may come as a surprise.
  However, the original distance covariance does not yield a
  Gaussian parametric test but rather a nonparametric test  
  for which there is no reason to expect superiority over
  its rank-based versions in Gaussian settings. In a different context, we have long been used to the celebrated
Chernoff--Savage phenomenon that normal-score rank statistics may
(uniformly) outperform their pseudo-Gaussian counterparts
\citep{MR100322}. This is best known in the context of two-sample
location problems; see, however, \citet{MR1312324},
\citet{MR2462206}, and \cite{deb2021efficiency}  for Chernoff--Savage
results for linear time series (traditional univariate ranks and
correlogram-based pseudo-Gaussian procedures) and vector independence
(Mahalanobis ranks and signs under elliptical symmetry and Wilks' test
as the pseudo-Gaussian procedure; measure-transportation-based ranks
under elliptical symmetry or independent component assumptions). Although the
present context is different, 
their superiority is another example in which
restricting to rank-based methods brings distribution-freeness at no substantial
cost in terms of efficiency/power.}

\begin{figure}[!htbp]
\centering
\includegraphics[width=\textwidth]{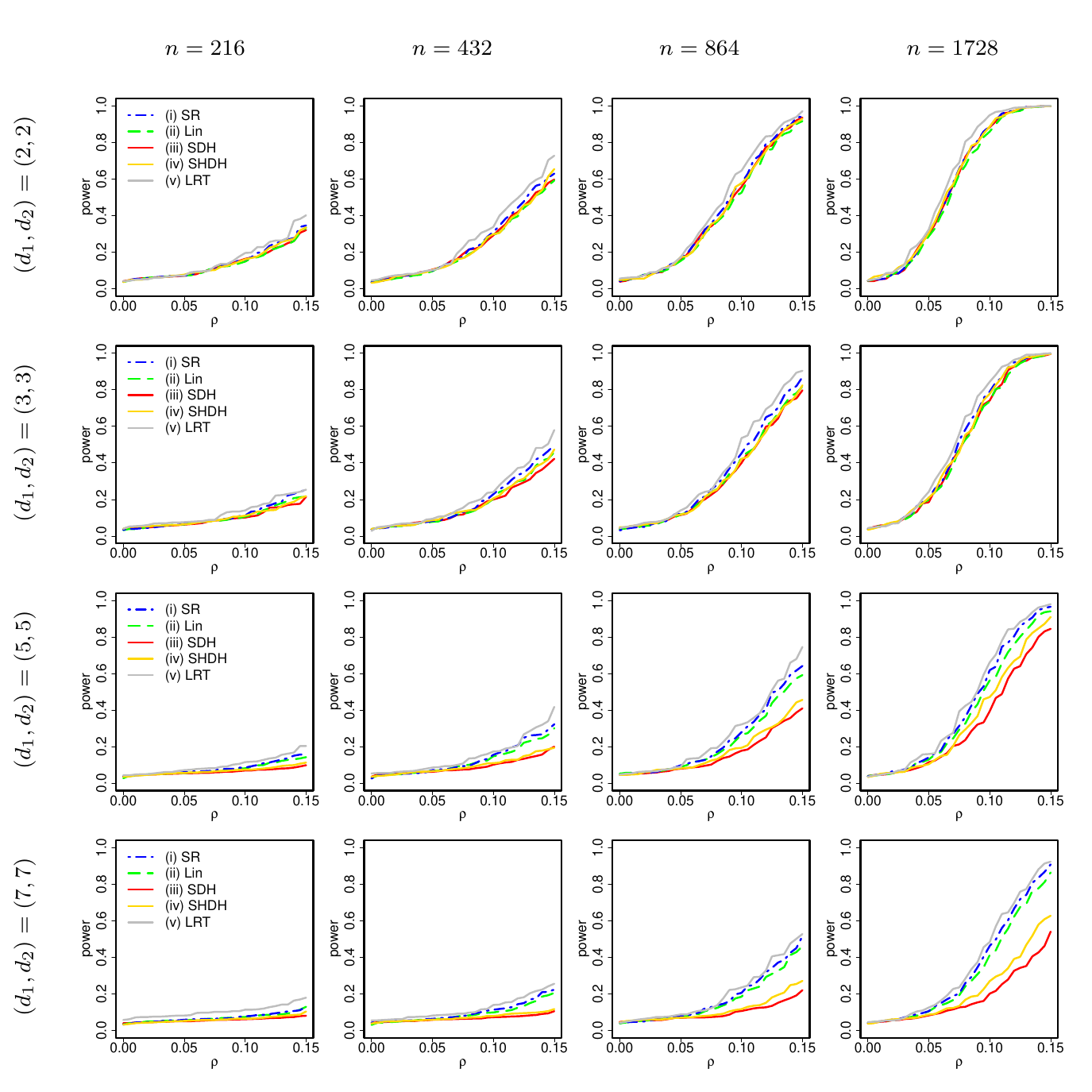}\vspace{-3mm}
\caption{Empirical powers of the  {five} competing tests 
in Example~\ref{ex:sim-power1}(a) ($\tau =0$, no within-group correlation). 
The $y$-axis represents {rejection frequencies} based on 1,000 replicates, 
the~$x$-axis represents $\rho$ (the between-group correlation), 
and the blue, green, red, and gold lines represent the performance of (i) Szekely and Rizzo's original distance covariance test, (ii) Lin's marginal rank version of the distance covariance test, (iii) Shi--Drton--Han's center-outward Wilcoxon version of the distance covariance test, (iv) the center-outward normal-score version of the distance covariance test,  {and (v) the likelihood ratio test,} respectively.\vspace{-5mm}}\label{fig:power1a}
\end{figure}


\begin{figure}[!htbp]
\centering
\includegraphics[width=\textwidth]{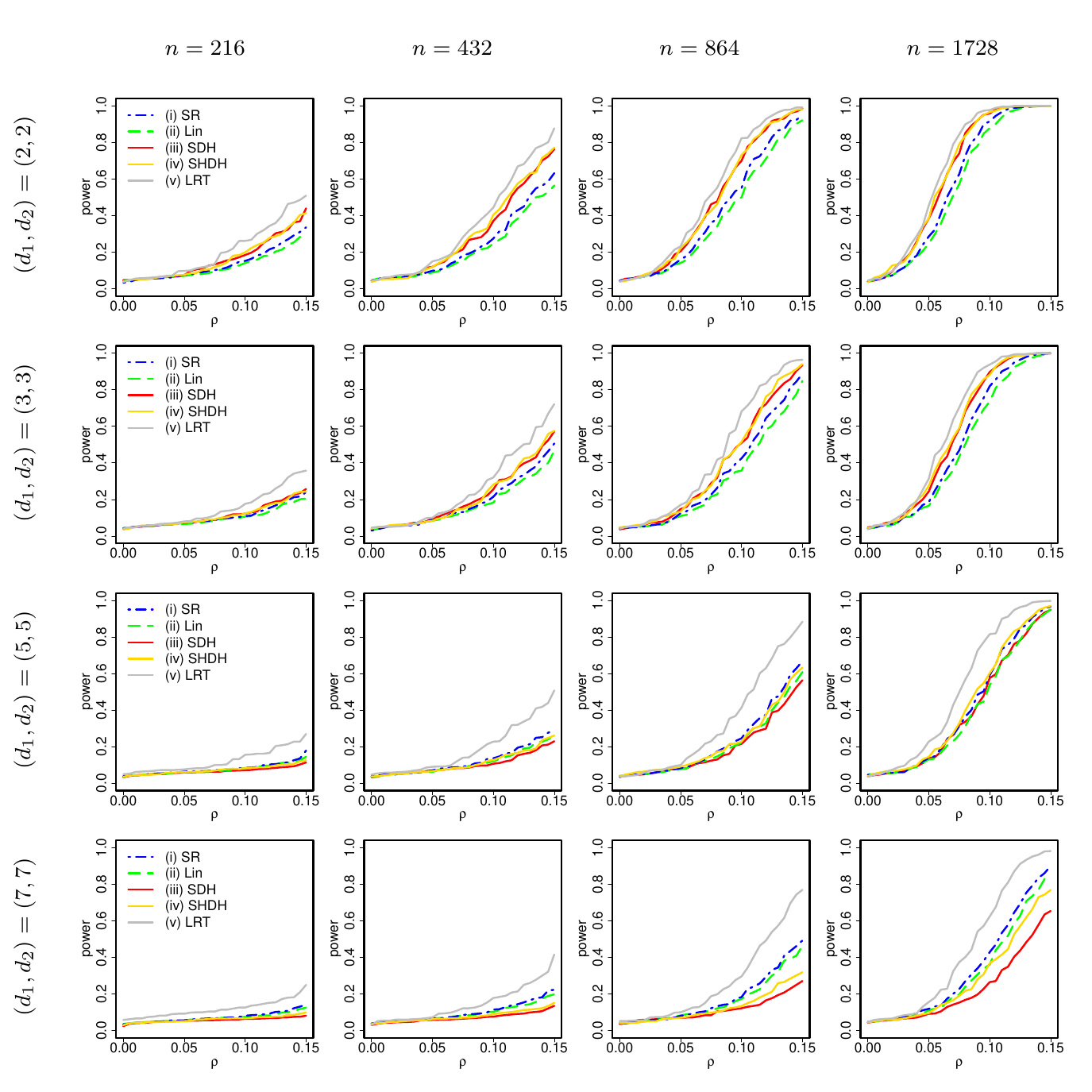}\vspace{-3mm}
\caption{Empirical powers of the  {five} competing tests 
in Example~\ref{ex:sim-power1}(b) ($\tau =0.5$, moderate within-group correlation). 
The $y$-axis represents {rejection frequencies} based on 1,000 replicates, 
the~$x$-axis represents $\rho$ (the between-group correlation), 
and the blue, green, red, and gold lines represent the performance of (i) Szekely and Rizzo's original distance covariance test, (ii) Lin's marginal rank version of the distance covariance test, (iii) Shi--Drton--Han's center-outward Wilcoxon version of the distance covariance test, (iv) the center-outward normal-score version of the distance covariance test,  {and (v) the likelihood ratio test,} respectively.\vspace{-5mm}}\label{fig:power1b}
\end{figure}


\begin{figure}[!htbp]
\centering
\includegraphics[width=\textwidth]{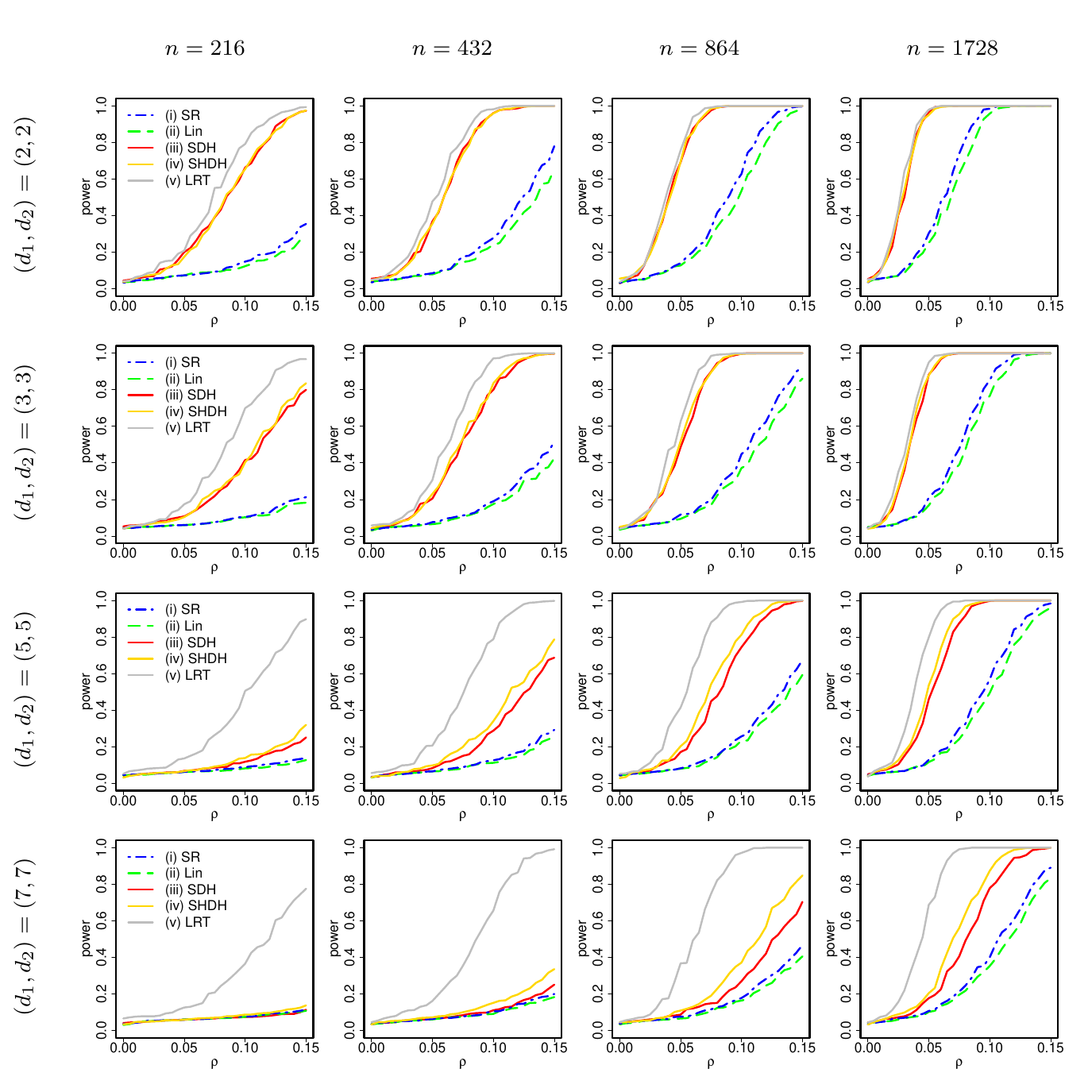}\vspace{-3mm}
\caption{Empirical powers of the  {five} competing tests 
in Example~\ref{ex:sim-power1}(c) ($\tau =0.9$, high within-group correlation). 
The $y$-axis represents {rejection frequencies} based on 1,000 replicates, 
the~$x$-axis represents $\rho$ (the between-group correlation), 
and the blue, green, red, and gold lines represent the performance of (i) Szekely and Rizzo's original distance covariance test, (ii) Lin's marginal rank version of the distance covariance test, (iii) Shi--Drton--Han's center-outward Wilcoxon version of the distance covariance test, (iv) the center-outward normal-score version of the distance covariance test,  {and (v) the likelihood ratio test,} respectively.\vspace{-5mm}}\label{fig:power1c}
\end{figure}


\begin{figure}[!htbp]
\centering
\includegraphics[width=\textwidth]{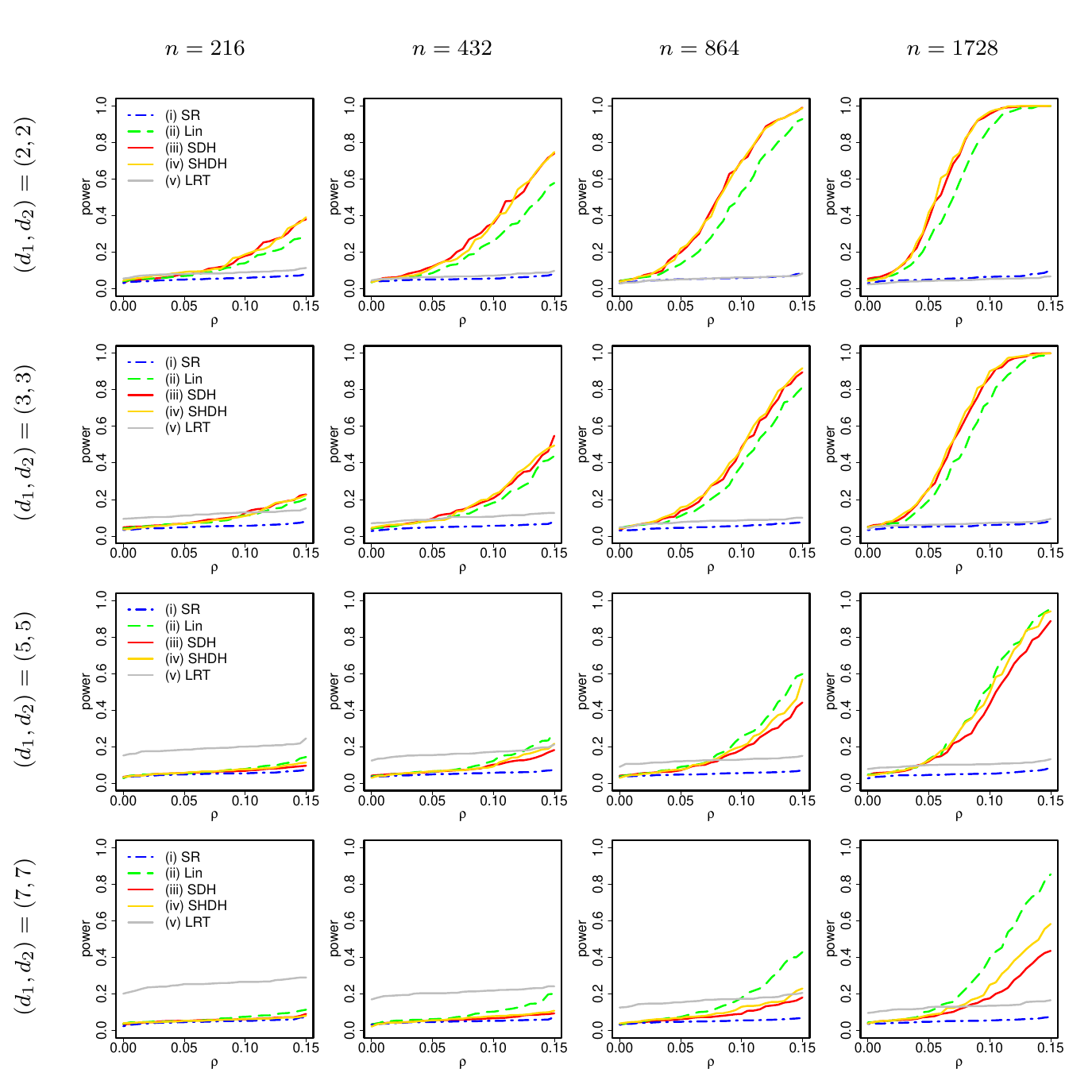}\vspace{-3mm}
\caption{Empirical powers of the  {five} competing tests 
in  {Example~\ref{ex:sim-power2}}. 
The $y$-axis represents {rejection frequencies} based on 1,000 replicates, 
the $x$-axis represents $\rho$ (the between-group correlation), 
and the blue, green, red, and gold lines represent the performance of (i) Szekely and Rizzo's original distance covariance test, (ii) Lin's marginal rank version of the distance covariance test, (iii) Shi--Drton--Han's center-outward Wilcoxon version of the distance covariance test, (iv) the center-outward normal-score version of the distance covariance test,  {and (v) the likelihood ratio test,} respectively.\vspace{-5mm}}\label{fig:cauchy}
\end{figure}



\section{Conclusion}

 {This paper provides a general framework for specifying dependence
measures}  that leverage the new
concept of center-outward ranks and signs.   {The associated independence tests
have the strong appeal of being 
fully}
distribution-free.
 Via the use of a
flexible class of generalized symmetric covariances and the incorporation of score functions, our framework
allows one to construct  a variety  of consistent dependence measures.  This, as our
numerical experiments  demonstrate, 
 can lead to significant gains in
power. 

The theory we develop facilitates the derivation of asymptotic
distributions   yielding easily computable approximate critical
values.  The key result is an asymptotic representation that also
allows us to establish, for the first time, a  {nontrivial local power result}
for tests of vector independence based on  center-outward ranks and signs.


\appendix

\section{Proofs} \label{sec:proof}

Some further concepts and notation concerning U-statistics are needed in this section.
 For any symmetric kernel $h$, any integer~$\ell\in\zahl{m}$, 
and any probability measure $\Pr_{\mZ}$, recall the definition
 \[
h_{\ell}(\mz_1\ldots,\mz_{\ell}; \Pr_{\mZ}):=\E h(\mz_1\ldots,\mz_{\ell},\mZ_{\ell+1},\ldots,\mZ_m),\]
of
the kernel
 and define
\[
\widetilde h_{\ell}(\mz_1,\ldots,\mz_{\ell}; \Pr_{\mZ}) 
:=h_{\ell}(\mz_1,\ldots,\mz_{\ell};\Pr_{\mZ})
  -\E h-\sum_{k=1}^{\ell-1}\sum_{1\leq i_1<\cdots<i_k\leq\ell}
  \widetilde h_{k}(\mz_{i_1},\ldots,\mz_{i_k};\Pr_{\mZ}),
\]
where $\mZ_1,\ldots,\mZ_m$ are $m$ independent copies of $\mZ\sim\Pr_{\mZ}$ and $\E h:=\E h(\mZ_1,\ldots,\mZ_m)$. 
The kernel as well as the corresponding U-statistic are said to be \emph{degenerate} under $\Pr_{\mZ}$ if $h_1(\cdot)$ has variance zero and  \emph{completely degenerate} if the
variances of $h_1(\mZ_1),\ldots,h_{m-1}(\mZ_1,\ldots,\mZ_m)$ all are zero. 
We also have, for  any (possibly dependent) random vectors $\mZ^{\prime}_1,\dots, \mZ^{\prime}_n$,
\[
   \mbinom{n}{m}^{-1}\!\!\!\sum_{1\le i_1<\cdots<i_m\le n} h\Big(\mZ^{\prime}_{i_1},\ldots,\mZ^{\prime}_{i_m}\Big) 
=\E h + \sum_{\ell=1}^{m}\mbinom{n}{\ell}^{-1}\!\!\!\sum_{1\le i_1<\cdots<i_\ell\le n} \mbinom{m}{\ell}\widetilde h_{\ell}\Big(\mZ^{\prime}_{i_1},\ldots,\mZ^{\prime}_{i_\ell};\Pr_{\mZ}\Big),
\]
 (the so-called {\it Hoeffding decomposition} with respect to $\Pr_{\mZ}$).

\vspace{0.2cm}
\noindent {\bf Notation.}
The cardinality of a set $\mathcal{S}$ is denoted as $\card(\mathcal{S})$
and its complement as $\mathcal{S}^{\complement}$. 
We use 
$\rightrightarrows$ to denote uniform convergence of functions
The cumulative distribution function and probability density function
of the univariate standard normal distribution are denoted by $\Phi$ and $\varphi$, respectively. 
Let~$\lVert X\rVert_{\sf L^{r}}:=(\E|X|^r)^{1/r}$ stand for the $\sf L^{r}$-norm of a random variable $X$.  
We use 
$\stackrel{\sf L^{r}}{\longrightarrow}$ 
to denote 
convergence of random variables in the $r$-th mean. 
For random vectors $\mX_n,\mX\in\R^d$, we write~$\mX_n\stackrel{\sf L^{r}}{\longrightarrow}\mX$ if $\lVert\mX_n-\mX\rVert\stackrel{\sf L^{r}}{\longrightarrow}0$.  
Let $(\cX,\cA)$ be a measurable space, and let $\P$ and $\Q$ be two
probability measures on $(\cX,\cA)$: we write $\P\ll\mu$ and
$\Q\ll\mu$ if  $\P$ and $\Q$ are absolutely continuous with respect
to a $\sigma$-finite measure $\mu$ on $(\cX,\cA)$.  The total
variation and Hellinger distances between $\Q$ and~$\P$ are denoted as $\TV(\Q,\P):=\sup_{A\in\cA}|\Q(A)-\P(A)|$ and $\HL(\Q,\P):=\{\int 2(1-\sqrt{{\d\Q}/{\d\P}})\d\P\}^{1/2}$, respectively.
We write $\Q\n\triangleleft \P\n$ for ``$\Q\n$ is contiguous to $\P\n$''. 

\subsection{Proofs for Section~\ref{sec:sgc}}

\subsubsection{Proof of Propostion~\ref{prop:GSC-Iconsistency}}

\begin{proof}[Proof of Propostion~\ref{prop:GSC-Iconsistency}]
The proof is entirely similar to the proof of Proposition~2 in \citet{MR3842884} and hence omitted. 
\end{proof}

%

\subsubsection{Proof of Example~\ref{prop:sgc}}

\begin{proof}[Proof of Example~\ref{prop:sgc}]
Item~\ref{example:dCov} is stated in \citet[Sec.~3.4]{MR3178526}. 
Item~\ref{example:WDM} is given in \citet[Proposition~1]{MR3842884}. 
Item~\ref{example:mD} can be proved using Equation~(3) in \citet{MR3737307}. 
Items~\ref{example:mR} and \ref{example:mBDY} can be proved using 
Proposition~D.5 in \citet{MR4185814supp}
and Theorem~7.2 in \citet{MR4185814}, respectively. 
\end{proof}

\subsubsection{Proof of Lemma~\ref{lem:kernel:kernel2}}

\begin{proof}[Proof of Lemma~\ref{lem:kernel:kernel2}]
Provided that $\E [f_1]$ and $\E [f_2]$ exist and are finite, we have
\begin{align*}
\E\:\Big[&k_{f_1,f_2,H_*^m}\Big((\mX_{11},\mX_{21}),\dots,(\mX_{1m},\mX_{2m})\Big)\Big]\\
&=\E\Big\{\sum_{\sigma\in H}\sgn(\sigma) f_1(\mX_{1\sigma(1)},\dots,\mX_{1\sigma(m)})\Big\}
\Big\{\sum_{\sigma\in H}\sgn(\sigma) f_2(\mX_{2\sigma(1)},\dots,\mX_{2\sigma(m)})\Big\}\\
&=\E\Big\{f_{1}(\mX_{11},\mX_{12},\mX_{13},\mX_{14},\mX_{15},\dots,\mX_{1m})
        -f_{1}(\mX_{11},\mX_{13},\mX_{12},\mX_{14},\mX_{15},\dots,\mX_{1m})\\
       &~~~~~~ -f_{1}(\mX_{14},\mX_{12},\mX_{13},\mX_{11},\mX_{15},\dots,\mX_{1m})
        +f_{1}(\mX_{14},\mX_{13},\mX_{12},\mX_{11},\mX_{15},\dots,\mX_{1m})\Big\}\\
 &~~~\times\Big\{
       f_{2}(\mX_{21},\mX_{22},\mX_{23},\mX_{24},\mX_{25},\dots,\mX_{2m})
        -f_{2}(\mX_{21},\mX_{23},\mX_{22},\mX_{24},\mX_{25},\dots,\mX_{2m})\\
       &~~~~~~ -f_{2}(\mX_{24},\mX_{22},\mX_{23},\mX_{21},\mX_{25},\dots,\mX_{2m})
        +f_{2}(\mX_{24},\mX_{23},\mX_{22},\mX_{21},\mX_{25},\dots,\mX_{2m})\Big\}.
\end{align*}
The  result follows. 
\end{proof}

\subsubsection{Proof of Theorem~\ref{thm:kernel:kernel2}}

\begin{proof}[Proof of Theorem~\ref{thm:kernel:kernel2}]

  The D-consistency 
  of the
  pairs of kernels used in Example
  \ref{prop:sgc}\ref{example:dCov} has been shown in
  \citet[Theorem~3(i)]{MR2382665}, \citet[Theorem~3.11]{MR3127883} and
  \citet[Item~(iv)]{MR3813995}.  The result for
  \ref{prop:sgc}\ref{example:WDM} is given in
  \citet[Theorem~1]{MR3842884}, that for \ref{prop:sgc}\ref{example:mD}
  in \citet[Proposition~1(i)]{MR3737307}, and that for
  \ref{prop:sgc}\ref{example:mR} and \ref{prop:sgc}\ref{example:mBDY}
    in \citet[p.~3435]{MR4185814}.
\end{proof}

\subsubsection{Proof of Lemma~\ref{lem:kernel:kernel1}}

\begin{proof}[Proof of Lemma~\ref{lem:kernel:kernel1}]
The lemma directly follows  from the definition of $\mu_{f_1,f_2,H}$
(cf.~Definition~\ref{def:sc}) and the fact that $f_1$ and $f_2$ are both
orthogonally invariant.
\end{proof}

\subsubsection{Proof of Proposition~\ref{prop:kernel:kernel1}}

\begin{proof}[Proof of Proposition~\ref{prop:kernel:kernel1}]
To verify that the kernels used in Example \ref{prop:sgc}\ref{example:dCov},\ref{example:mD}--\ref{example:mBDY} are orthogonally invariant, it suffices to notice that $
\fO\bm{w}-\fO\mv=\fO(\bm{w}-\mv)$, 
$(\fO\bm{w})^{\top}(\fO\mv)=\bm{w}^{\top}\fO^{\top}\fO\mv=\bm{w}^{\top}\mv,$ 
and 
$\lVert\fO\bm{w}\rVert=
\sqrt{\bm{w}^{\top}\bm{w}}=
\lVert\bm{w}\rVert$ 
for any orthogonal matrix $\fO \in\R^{d\times d}$ and $\bm{w},\mv\in\R^d$. 
\end{proof}



\subsection{Proofs for Section~\ref{sec:measure}}

\subsubsection{Proof of Proposition~\ref{prop:score}}

\begin{proof}[Proof of Proposition~\ref{prop:score}]
The first part is trivial. We next prove the second part.  The function $u\mapsto\big(F_{\chi^2_d}^{-1}(u)\big)^{1/2}$ is continuous over $[0,1)$, and
\[\int_{0}^{1}\Big(\big(F_{\chi^2_d}^{-1}(u)\big)^{1/2}\Big)^2\d u=\int_{0}^{1}F_{\chi^2_d}^{-1}(u)\d u=\E [F_{\chi^2_d}^{-1}(U)]=d,\]
where $U$ is uniformly distributed over $[0,1]$, and thus $F_{\chi^2_d}^{-1}(U)$ is chi-square distributed with $d$ degrees of freedom and expectation $d$. Hence, $J_{\text{\tiny{\rm vdW}}}(u)$ is weakly regular; 
it is not strongly regular, however, since it is unbounded. 
\end{proof}

\subsubsection{Proof of Proposition~\ref{prop:sGSC}}

\paragraph{Proof of Proposition~\ref{prop:sGSC}\ref{prop:sGSC1}}

\begin{proof}[Proof of Proposition~\ref{prop:sGSC}\ref{prop:sGSC1}]
  This  follows
  immediately from Proposition~\ref{prop:Hallin_full}\ref{prop:Hallin3} and
  the independence between $[\fG^{(n)}_{1,\pms}(\mX_{1i})]_{i=1}^{n}$
  and $[\fG^{(n)}_{2,\pms}(\mX_{2i})]_{i=1}^{n}$ under the null hypothesis. 
\end{proof}

\paragraph{Proof of Proposition~\ref{prop:sGSC}\ref{prop:sGSC2}}

\begin{proof}[Proof of Proposition~\ref{prop:sGSC}\ref{prop:sGSC2}]
The desired result follows from combining Lemma~\ref{lem:kernel:kernel1} and Proposition~\ref{prop:Hallin_full}\ref{prop:Hallin5}. 
\end{proof}

\paragraph{Proof of Proposition~\ref{prop:sGSC}\ref{prop:sGSC3}}

\begin{proof}[Proof of Proposition~\ref{prop:sGSC}\ref{prop:sGSC3}]
We only prove the D-consistency part. 
Using Lemma~\ref{lem:kernel:kernel2}, it remains to prove that the
independence of $\fG_{1,\pms}(\mX_1)$ and $\fG_{2,\pms}(\mX_2)$ implies
the independence of $\mX_1$ and~$\mX_2$.  Notice that $\fF_{\pms}$ is $\P$-almost surely invertible for any $\P\in\cP_{d}^{\ac}$ \citep[Section 6.2.3 and Remark 6.2.11]{MR2401600}, and so is $\fG_{\pms}$. 
The independence claim follows.
\end{proof}

\paragraph{Proof of Proposition~\ref{prop:sGSC}\ref{prop:sGSC4}}

\begin{proof}[Proof of Proposition~\ref{prop:sGSC}\ref{prop:sGSC4}]
The main idea of the proof consists in  bounding $|\tenq{W}\n_{\mu}-W_{\mu}|$. 
Let~$\mY^{(n)}_{ki}$ and $\mY_{ki}$ stand for  
$\fG_{k,\pms}^{(n)}(\mX_{ki})$ and 
$\fG_{k,\pms}(\mX_{ki})$, 
respectively. 
Notice that
\begin{align*}
\tenq{W}\n_{J_1,J_2,\mu_{f_1,f_2,H}}
&=(n)_{m}^{-1}
\sum_{[i_1,\dots,i_m]\in I_{m}^{n}}k_{f_1,f_2,H}\Big((\mY^{(n)}_{1i_1},\mY^{(n)}_{2i_1}),\dots,(\mY^{(n)}_{1i_m},\mY^{(n)}_{2i_m})\Big),\\
W_{J_1,J_2,\mu_{f_1,f_2,H}}
&=(n)_{m}^{-1}
\sum_{[i_1,\dots,i_m]\in I_{m}^{n}}k_{f_1,f_2,H}\Big((\mY_{1i_1},\mY_{2i_1}),\dots,(\mY_{1i_m},\mY_{2i_m})\Big),
\end{align*}
where
\begin{align*}
k_{f_1,f_2,H}\Big((&\mx_{11},\mx_{21}),\dots,(\mx_{1m},\mx_{2m})\Big)\\
:=\;&
\Big\{\sum_{\sigma\in H}\sgn(\sigma) f_1(\mx_{1\sigma(1)},\dots,\mx_{1\sigma(m)})\Big\}
\Big\{\sum_{\sigma\in H}\sgn(\sigma) f_2(\mx_{2\sigma(1)},\dots,\mx_{2\sigma(m)})\Big\}.
\end{align*}
Since
$f_k([\mY^{(n)}_{ki_\ell}]_{\ell=1}^{m})$ and $f_k([\mY_{ki_\ell}]_{\ell=1}^{m})$
are almost surely bounded by some constant $C_{J_k,f_k}$, we deduce
\begin{align*}
\Big\lvert k_{f_1,f_2,H}&\Big([(\mY^{(n)}_{1i_\ell},\mY^{(n)}_{2i_\ell})]_{\ell=1}^{m}\Big)-k_{f_1,f_2,H}\Big([(\mY_{1i_\ell},\mY_{2i_\ell})]_{\ell=1}^{m}\Big)\Big\rvert\\
\le\;&\card(H)\cdot C_{J_1,f_1}\cdot \sum_{\sigma\in H}
\Big\lvert f_2\Big([\mY^{(n)}_{2\sigma(i_\ell)}]_{\ell=1}^{m}\Big)
          -f_2\Big([\mY_{2\sigma(i_\ell)}]_{\ell=1}^{m}\Big)\Big\rvert\\
&+\card(H)\cdot C_{J_2,f_2}\cdot \sum_{\sigma\in H}
\Big\lvert f_2\Big([\mY^{(n)}_{1\sigma(i_\ell)}]_{\ell=1}^{m}\Big)
          -f_2\Big([\mY_{1\sigma(i_\ell)}]_{\ell=1}^{m}\Big)\Big\rvert,
\end{align*}
recalling that $\card(H)$ denotes the number of permutations in the subgroup $H$.
Moreover, 
\begin{align*}
\Big\lvert \tenq{W}\n_{J_1,J_2,\mu_{f_1,f_2,H}}\mkern-40mu&\mkern40mu -W_{J_1,J_2,\mu_{f_1,f_2,H}}\Big\rvert\\
\le\;&\card(H)^2\cdot C_{J_1,f_1}\cdot \Big[(n)_{m}^{-1}\sum_{[i_1,\dots,i_m]\in I_{m}^{n}}
\Big\lvert f_2\Big([\mY^{(n)}_{2i_\ell}]_{\ell=1}^{m}\Big)
          -f_2\Big([\mY_{2i_\ell}]_{\ell=1}^{m}\Big)\Big\rvert\Big]\\
&+\card(H)^2\cdot C_{J_2,f_2}\cdot \Big[(n)_{m}^{-1}\sum_{[i_1,\dots,i_m]\in I_{m}^{n}}
\Big\lvert f_1\Big([\mY^{(n)}_{1i_\ell}]_{\ell=1}^{m}\Big)
          -f_1\Big([\mY_{1i_\ell}]_{\ell=1}^{m}\Big)\Big\rvert\Big]
\stackrel{\sf a.s.}{\longrightarrow} 0.
\end{align*}
This, together with the fact that $W_{J_1,J_2,\mu_{f_1,f_2,H}}\stackrel{\sf a.s.}{\longrightarrow}\mu_{\pms}(\mX_1,\mX_2)$ by the strong consistency of U-statistics, yields~$\tenq{W}\n_{J_1,J_2,\mu_{f_1,f_2,H}}\stackrel{\sf a.s.}{\longrightarrow}\mu_{\pms}(\mX_1,\mX_2)$.
\end{proof}

\subsubsection{Proof of Theorem~\ref{thm:sGSC5}}

We first fix some notation and prove a property that will hold for
all GSCs $\mu$ and associated kernel functions considered in Example \ref{prop:sgc}\ref{example:dCov}--\ref{example:mBDY}. 
For $k=1,2$, let $\bm{y}^{(n)}_{ki}=\fJ(\bmu^{*(n)}_{ki})$,
where $\bmu^{*(n)}_{ki}$, $i\in\zahl{n}$ are the deterministic points
forming the grid $\kG^{d_k}_{\mn}$. Writing $\mY^{(n)}_{ki}$ and
$\mY_{ki}$ for~$\fG_{k,\pms}^{(n)}(\mX_{ki})$ and~$\fG_{k,\pms}(\mX_{ki})$, 
respectively, let us show that
\begin{align}\label{eq:han-1}
\Xi_{k}^{(n)}:=\sup_{1\le i\le n}\lVert \mY^{(n)}_{ki}-\mY_{ki}\rVert\stackrel{\sf a.s.}{\longrightarrow}0, ~~~k=1,2. 
\end{align}

Recall that, by definition of strong regularity, $J_k$ is Lipschitz-continuous  with some constant~$L_k$, strictly monotone, and satisfies $J_k(0) = 0$. 
Then we immediately have $|J_k(u)|\le L_k$ for all $u\in[0,1)$, and thus 
$\mY^{(n)}_{ki}$ and $\mY_{ki}$ are almost surely bounded by $L_k$. As long as $\P_{\mX_k}\in \cP_{d_k}^{\#}$,  in order to prove that~$\Xi_{k}^{(n)}\stackrel{\sf a.s.}{\longrightarrow}0$, it suffices to show that $\lVert\fJ_k(\bmu_{k1})-\fJ_k(\bmu_{k2})\rVert\le 2L_k\lVert\bmu_{k1}-\bmu_{k2}\rVert$ for any $\bmu_{k1},\bmu_{k2}\in\R^{d_k}$ with $\lVert\bmu_{k1}\rVert,\lVert\bmu_{k2}\rVert<1$. Without loss of generality, assume that $\lVert\bmu_{k2}\rVert\le\lVert\bmu_{k1}\rVert$. If $\lVert\bmu_{k2}\rVert=0$, the claim is obvious by noticing $|J_k(u)|\le L_k u$ for $u\in[0,1)$ and then $\lVert\fJ_k(\bmu_{k1})\rVert\le L_k\lVert\bmu_{k1}\rVert$; otherwise we have
\begin{align*}
\lVert\fJ_k(\bmu_{k1})-\fJ_k(\bmu_{k2})\rVert
&\le\Big\lVert\fJ_k(\bmu_{k1})-\fJ_k\Big(\frac{\lVert\bmu_{k2}\rVert}{\lVert\bmu_{k1}\rVert}\bmu_{k1}\Big)\Big\rVert
+\Big\lVert\fJ_k\Big(\frac{\lVert\bmu_{k2}\rVert}{\lVert\bmu_{k1}\rVert}\bmu_{k1}\Big)-\fJ_k(\bmu_{k2})\Big\rVert\\
&=\Big\lvert J_k(\lVert\bmu_{k1}\rVert)-J_k(\lVert\bmu_{k2}\rVert)\Big\rvert
+\frac{J_k(\lVert\bmu_{k2}\rVert)}{\lVert\bmu_{k2}\rVert}\cdot\Big\lVert\frac{\lVert\bmu_{k2}\rVert}{\lVert\bmu_{k1}\rVert}\bmu_{k1}-\bmu_{k2}\Big\rVert\\
&\le L_k\Big\lvert \lVert\bmu_{k1}\rVert - \lVert\bmu_{k2}\rVert\Big\rvert
+L_k\Big\lVert\frac{\lVert\bmu_{k2}\rVert}{\lVert\bmu_{k1}\rVert}\bmu_{k1}-\bmu_{k2}\Big\rVert
\le 2L_k\lVert \bmu_{k1}-\bmu_{k2}\rVert.
\end{align*}
This completes the proof of \eqref{eq:han-1}.

\paragraph{Proof of Theorem~\ref{thm:sGSC5} ($h=h_{\dCov^2}$)}

\begin{proof}[Proof of Theorem~\ref{thm:sGSC5} ($h=h_{\dCov^2}$)]
Recall that
\[f_1^{\dCov}([\mw_i]_{i=1}^{4})=\frac12\lVert
\mw_1-\mw_2\rVert\quad\text{and}\quad f_2^{\dCov}([\mw_i]_{i=1}^{4})=\frac12\lVert
\mw_1-\mw_2\rVert,
\]
 with possibly different dimension for the inputs.
Now,
$\frac12\lVert\mY^{(n)}_{ki_1}-\mY^{(n)}_{ki_2}\rVert$ and 
$\frac12\lVert\mY_{ki_1}-\mY_{ki_2}\rVert$
are almost surely bounded by $L_k$, 
since $\mY^{(n)}_{ki}$ and $\mY_{ki}$ are. 
 Next,
\[
\Big\lvert \frac12\lVert\mY^{(n)}_{ki_1}-\mY^{(n)}_{ki_2}\rVert
          -\frac12\lVert\mY_{ki_1}-\mY_{ki_2}\rVert\Big\rvert
\le \frac12\lVert\mY^{(n)}_{ki_1}-\mY_{ki_1}\rVert
   +\frac12\lVert\mY^{(n)}_{ki_2}-\mY_{ki_2}\rVert
\le \sup_{1\le i\le n}\lVert \mY^{(n)}_{ki}-\mY_{ki}\rVert,
\]
and we deduce that
\[
(n)_{4}^{-1}\sum_{[i_1,\dots,i_4]\in I_{4}^{n}}
\Big\lvert \frac12\lVert\mY^{(n)}_{ki_1}-\mY^{(n)}_{ki_2}\rVert
          -\frac12\lVert\mY_{ki_1}-\mY_{ki_2}\rVert\Big\rvert
\le \sup_{1\le i\le n}\lVert \mY^{(n)}_{ki}-\mY_{ki}\rVert
\stackrel{\sf a.s.}{\longrightarrow} 0.
\]
Both conditions in \eqref{condition:as} are satisfied, and  the proof is thus completed. 
\end{proof}

\paragraph{Proof of Theorem~\ref{thm:sGSC5} ($h=h_{M}$)}

\begin{proof}[Proof of Theorem~\ref{thm:sGSC5} ($h=h_{M}$)]
Recall that 
$f_1^{M}([\mw_i]_{i=1}^{5})=f_2^{M}([\mw_i]_{i=1}^{5})=\frac12\ind(\mw_{1},\mw_{2}\preceq\mw_{5})$,
up to a change in input dimension for the two functions.
It is obvious that
$f_k(\{\mY^{(n)}_{ki_\ell}\}_{\ell=1}^{m})$ and~$f_k(\{\mY_{ki_\ell}\}_{\ell=1}^{m})$
are almost surely bounded. 
Next we verify the second condition in \eqref{condition:as}.  

We have for $k=1,2$,
\[
\Big\lvert\ind(\mY^{(n)}_{ki_1},\mY^{(n)}_{ki_2}\preceq\mY^{(n)}_{ki_5})-\ind(\mY_{ki_1},\mY_{ki_2}\preceq\mY_{ki_5})\Big\rvert
\le\ind(\cB^{\complement}_{k;i_1,i_2,i_3,i_4,i_5}),
\]
where
\[
\cB_{k;i_1,i_2,i_3,i_4,i_5}:=\Big\{
\lVert \mY^{(n)}_{ki_1}-\mY^{(n)}_{ki_5}\rVert\ge2\,\Xi_{k}^{(n)},~~~
\lVert \mY^{(n)}_{ki_2}-\mY^{(n)}_{ki_5}\rVert\ge2\,\Xi_{k}^{(n)}\Big\}.
\]
Accordingly, 
\begin{align*}
\; (n)_5^{-1}&\!\!\!\!\!\sum_{[i_1,\dots,i_5]\in I^n_{5}}
\Big\lvert\ind(\mY^{(n)}_{ki_1},\mY^{(n)}_{ki_2}\preceq\mY^{(n)}_{ki_5})-\ind(\mY_{ki_1},\mY_{ki_2}\preceq\mY_{ki_5})\Big\rvert\\
\le\;&(n)_3^{-1}\card\Big\{[i_1,i_2,i_5]\in I^n_{3}: 
\lVert \mY^{(n)}_{ki_1}-\mY^{(n)}_{ki_5}\rVert\!<\!2\,\Xi_{k}^{(n)}~\text{or}~~~
\lVert \mY^{(n)}_{ki_2}-\mY^{(n)}_{ki_5}\rVert\!<\!2\,\Xi_{k}^{(n)}\Big\}\\
=\;&(n)_3^{-1}\card\Big\{[i_1,i_2,i_5]\in I^n_{3}: 
\lVert \bm{y}^{(n)}_{ki_1}-\bm{y}^{(n)}_{ki_5}\rVert<2\,\Xi_{k}^{(n)}~~\text{or}~~~
\lVert \bm{y}^{(n)}_{ki_2}-\bm{y}^{(n)}_{ki_5}\rVert<2\,\Xi_{k}^{(n)}\Big\}
\stackrel{\sf a.s.}{\longrightarrow}0,\yestag\label{eq:smallo1as_hoem}
\end{align*}
which completes the proof. 
\end{proof}

\paragraph{Proof of Theorem~\ref{thm:sGSC5} ($h=h_{D}$)}

\begin{proof}[Proof of Theorem~\ref{thm:sGSC5} ($h=h_{D}$)]
Recall that 
$f_1^{D}([\mw_i]_{i=1}^{5})\! =\! f_2^{D}([\mw_i]_{i=1}^{5})\!=\!\frac12\Arc(\mw_1-\mw_5,\mw_2-\mw_5)$,
up to a change in input dimension for the two functions.
Obviously, 
$f_k([\mY^{(n)}_{ki_\ell}]_{\ell=1}^{m})$ and~$f_k([\mY_{ki_\ell}]_{\ell=1}^{m})$
are almost surely bounded. 
To verify the second condition in \eqref{condition:as}, 
we start by bounding the difference between $\Arc(\mY^{(n)}_{ki_1}-\mY^{(n)}_{ki_5}, \mY^{(n)}_{ki_2}-\mY^{(n)}_{ki_5})$ and $\Arc(\mY_{ki_1}-\mY_{ki_5}, \mY_{ki_2}-\mY_{ki_5})$. 

For $k=1,2$, 
consider $(\bm{y}_{k1},\bm{y}_{k2},\bm{y}_{k5})\in(\R^{d_k})^3$ such that
\[\min\{\lVert \bm{y}_{k1}-\bm{y}_{k5}\rVert, \lVert \bm{y}_{k2}-\bm{y}_{k5}\rVert\}\ge\eta
~~~\text{and}~~~
\zeta\le\Arc(\bm{y}_{k1}-\bm{y}_{k5}, \bm{y}_{k2}-\bm{y}_{k5})\le\frac12-\zeta,
\]
where $\eta$ and $\zeta$ will be specified later on. 
For $(\bm{y}^{\prime}_{k1},\bm{y}^{\prime}_{k2},\bm{y}^{\prime}_{k5})\in(\R^{d_k})^3$ satisfying $\lVert \bm{y}_{ki}-\bm{y}^{\prime}_{ki}\rVert\le\delta$ for~$i=1,2,5$,
\begin{align*}
\Arc(\bm{y}_{k1}-\bm{y}_{k5}, \bm{y}_{k1}-\bm{y}^{\prime}_{k5})&
\le \frac{1}{2\pi}\arcsin\frac{\lVert\bm{y}_{k5}-\bm{y}^{\prime}_{k5}\rVert}{\lVert\bm{y}_{k1}-\bm{y}_{k5}\rVert}
\le \frac{1}{2\pi}\arcsin\frac{\delta}{\eta},\\
\Arc(\bm{y}_{k1}-\bm{y}^{\prime}_{k5}, \bm{y}^{\prime}_{k1}-\bm{y}^{\prime}_{k5})&
\le \frac{1}{2\pi}\arcsin\frac{\lVert\bm{y}_{k1}-\bm{y}^{\prime}_{k1}\rVert}{\lVert\bm{y}_{k1}-\bm{y}^{\prime}_{k5}\rVert}
\le \frac{1}{2\pi}\arcsin\frac{\delta}{\eta-\delta},\\
\Arc(\bm{y}_{k2}-\bm{y}_{k5}, \bm{y}_{k2}-\bm{y}^{\prime}_{k5})&
\le \frac{1}{2\pi}\arcsin\frac{\lVert\bm{y}_{k5}-\bm{y}^{\prime}_{k5}\rVert}{\lVert\bm{y}_{k2}-\bm{y}_{k5}\rVert}
\le \frac{1}{2\pi}\arcsin\frac{\delta}{\eta},\\
\text{and}~~~
\Arc(\bm{y}_{k2}-\bm{y}^{\prime}_{k5}, \bm{y}^{\prime}_{k2}-\bm{y}^{\prime}_{k5})&
\le \frac{1}{2\pi}\arcsin\frac{\lVert\bm{y}_{k2}-\bm{y}^{\prime}_{k2}\rVert}{\lVert\bm{y}_{k2}-\bm{y}^{\prime}_{k5}\rVert}
\le \frac{1}{2\pi}\arcsin\frac{\delta}{\eta-\delta}.
\end{align*}
Assuming that
\[\frac{1}{2\pi}\Big(2\arcsin\frac{\delta}{\eta}+2\arcsin\frac{\delta}{\eta-\delta}\Big)\le\zeta,\yestag\label{eq:verbelow1}\]
we obtain
\[
\lvert\Arc(\bm{y}_{k1}-\bm{y}_{k5}, \bm{y}_{k2}-\bm{y}_{k5})-\Arc(\bm{y}^{\prime}_{k1}-\bm{y}^{\prime}_{k5}, \bm{y}^{\prime}_{k2}-\bm{y}^{\prime}_{k5})\rvert\le \frac{1}{2\pi}\Big(2\arcsin\frac{\delta}{\eta}+2\arcsin\frac{\delta}{\eta-\delta}\Big).
\]
For $\delta\le 1/4$, take $\eta=\sqrt{\delta}$ and $\zeta=3\sqrt{\delta}/2$ such that \eqref{eq:verbelow1} holds,
\begin{align*}
     &\frac{1}{2\pi}\Big(2\arcsin\frac{\delta}{\eta}+2\arcsin\frac{\delta}{\eta-\delta}\Big)
  =\; \frac{1}{2\pi}\Big(2\arcsin\sqrt{\delta}+2\arcsin\frac{\sqrt{\delta}}{1-\sqrt{\delta}}\Big)\\
\le\;&\frac{1}{2\pi}\Big(2\arcsin\sqrt{\delta}+2\arcsin2\sqrt{\delta}\Big)
\le\; \frac{1}{2\pi}\Big(2\frac{\pi}{2}\sqrt{\delta}+2\frac{\pi}{2}(2\sqrt{\delta})\Big)=\frac{3}{2}\sqrt{\delta}=\zeta.
\end{align*}
It follows that for $\delta\le 1/4$ and 
$(\bm{y}_{k1},\bm{y}_{k2},\bm{y}_{k5}),(\bm{y}^{\prime}_{k1},\bm{y}^{\prime}_{k2},\bm{y}^{\prime}_{k5})\in(\R^{d_k})^3$ such that
\begin{align*}
\min\{\lVert \bm{y}_{k1}-\bm{y}_{k5}\rVert, \lVert \bm{y}_{k2}-\bm{y}_{k5}\rVert\}&\ge\sqrt{\delta},
~~~\frac{3}{2}\sqrt{\delta}\le\Arc(\bm{y}_{k1}-\bm{y}_{k5}, \bm{y}_{k2}-\bm{y}_{k5})\le\frac12-\frac{3}{2}\sqrt{\delta},\\
~~~\text{and}~~~~~~\lVert \bm{y}_{ki}-\bm{y}^{\prime}_{ki}\rVert &\le\delta~~~~~~~\text{for }i=1,2,5, 
\end{align*}
we have
\[
\lvert\Arc(\bm{y}_{k1}-\bm{y}_{k5}, \bm{y}_{k2}-\bm{y}_{k5})-\Arc(\bm{y}^{\prime}_{k1}-\bm{y}^{\prime}_{k5}, \bm{y}^{\prime}_{k2}-\bm{y}^{\prime}_{k5})\rvert\le \frac{3}{2}\sqrt{\delta}.
\]
Then,  for $k=1,2$,
\begin{align*}
\;&\Big\lvert\Arc(\mY^{(n)}_{ki_1}-\mY^{(n)}_{ki_5}, \mY^{(n)}_{ki_2}-\mY^{(n)}_{ki_5})-\Arc(\mY_{ki_1}-\mY_{ki_5}, \mY_{ki_2}-\mY_{ki_5})\Big\rvert\\
\le\;&\frac{3}{2}\sqrt{\Xi_{k}^{(n)}}\cdot\ind(\cA_{k;i_1,i_2,i_3,i_4,i_5})+\Big(\frac12+\frac12\Big)\cdot\ind(\cA^{\complement}_{k;i_1,i_2,i_3,i_4,i_5})
\le\frac{3}{2}\sqrt{\Xi_{k}^{(n)}}+\ind(\cA^{\complement}_{k;i_1,i_2,i_3,i_4,i_5}),
\end{align*}
where
\begin{multline*}
\cA_{k;i_1,i_2,i_3,i_4,i_5}:=\Big\{\Xi_{k}^{(n)}\le\frac14,~~~
\lVert \mY^{(n)}_{ki_1}-\mY^{(n)}_{ki_5}\rVert\ge\sqrt{\Xi_{k}^{(n)}},~~~
\lVert \mY^{(n)}_{ki_2}-\mY^{(n)}_{ki_5}\rVert\ge\sqrt{\Xi_{k}^{(n)}},\\
\text{and}~~~
\frac{3}{2}\sqrt{\Xi_{k}^{(n)}}
\le\Arc(\mY^{(n)}_{ki_1}-\mY^{(n)}_{ki_5}, \mY^{(n)}_{ki_2}-\mY^{(n)}_{ki_5})
\le\frac{1}{2}-\frac{3}{2}\sqrt{\Xi_{k}^{(n)}}\Big\},
\end{multline*}
and, accordingly, 
\begin{align*}
(n)_5^{-1}&\!\!\!\!\sum_{[i_1,\dots,i_5]\in I^n_{5}}
\Big\lvert\frac12\Arc(\mY^{(n)}_{ki_1}-\mY^{(n)}_{ki_5}, \mY^{(n)}_{ki_2}-\mY^{(n)}_{ki_5})-\frac12\Arc(\mY_{ki_1}-\mY_{ki_5}, \mY_{ki_2}-\mY_{ki_5})\Big\rvert\\
\le\;&\frac12\Big(\frac{3}{2}\sqrt{\Xi_{k}^{(n)}}+\ind\Big\{\Xi_{k}^{(n)}>\frac14\Big\}\\
&\quad+(n)_3^{-1}\card\Big\{[i_1,i_2,i_5]\in I^n_{3}: 
\lVert \mY^{(n)}_{ki_1}-\mY^{(n)}_{ki_5}\rVert<\sqrt{\Xi_{k}^{(n)}},~~~\text{or}~~~
\lVert \mY^{(n)}_{ki_2}-\mY^{(n)}_{ki_5}\rVert<\sqrt{\Xi_{k}^{(n)}},\\
&\qquad\qquad\qquad~~~\text{or}~~~
\Arc(\mY^{(n)}_{ki_1}-\mY^{(n)}_{ki_5}, \mY^{(n)}_{i_2}-\mY^{(n)}_{ki_5})
\in\Big[0,\frac{3}{2}\sqrt{\Xi_{k}^{(n)}}\Big)\cup\Big(\frac12-\frac{3}{2}\sqrt{\Xi_{k}^{(n)}},\frac12\Big]\Big\}\Big)\\
=\;&\frac12\Big(\frac{3}{2}\sqrt{\Xi_{k}^{(n)}}+\ind\Big\{\Xi_{k}^{(n)}>\frac14\Big\}\\
&\quad+(n)_3^{-1}\card\Big\{[i_1,i_2,i_5]\in I^n_{3}: 
\lVert \bm{y}^{(n)}_{ki_1}-\bm{y}^{(n)}_{ki_5}\rVert<\sqrt{\Xi_{k}^{(n)}},~~~\text{or}~~~
\lVert \bm{y}^{(n)}_{ki_2}-\bm{y}^{(n)}_{ki_5}\rVert<\sqrt{\Xi_{k}^{(n)}},\\
&\qquad\qquad\qquad~~~\text{or}~~~
\Arc(\bm{y}^{(n)}_{ki_1}-\bm{y}^{(n)}_{ki_5}, \bm{y}^{(n)}_{ki_2}-\bm{y}^{(n)}_{ki_5})
\in\Big[0,\frac{3}{2}\sqrt{\Xi_{k}^{(n)}}\Big)\cup\Big(\frac12-\frac{3}{2}\sqrt{\Xi_{k}^{(n)}},\frac12\Big]\Big\}\Big).\yestag\label{eq:smallo1as}
\end{align*}
Since, for any sequence $[\delta\n]_{n=1}^{\infty}$ tending to $0$, it holds that
\begin{align*}
\;&(n)_3^{-1}\card\Big\{[i_1,i_2,i_5]\in I^n_{3}: 
\lVert \bm{y}^{(n)}_{ki_1}-\bm{y}^{(n)}_{ki_5}\rVert<\sqrt{\delta\n},~~~\text{or}~~~
\lVert \bm{y}^{(n)}_{ki_2}-\bm{y}^{(n)}_{ki_5}\rVert<\sqrt{\delta\n},\\
&\qquad\qquad~~~\text{or}~~~
\Arc(\bm{y}^{(n)}_{ki_1}-\bm{y}^{(n)}_{ki_5}, \bm{y}^{(n)}_{ki_2}-\bm{y}^{(n)}_{ki_5})
\in\Big[0,\frac{3}{2}\sqrt{\delta\n}\Big)\cup\Big(\frac12-\frac{3}{2}\sqrt{\delta\n},\frac12\Big]\Big\}\to 0,
\end{align*}
we have shown that \eqref{eq:smallo1as} converges to $0$ almost surely. 
This completes the proof.
\end{proof}

\paragraph{Proof of Theorem~\ref{thm:sGSC5} ($h=h_{R},\, h_{\tau^*}$)}

\begin{proof}[Proof of Theorem~\ref{thm:sGSC5} ($h=h_{R},\, h_{\tau^*}$)]
The proof is similar to the proof of Theorem~\ref{thm:sGSC5} ($h=h_{D}$) and hence omitted.
\end{proof}

%

\subsection{Proofs for Section~\ref{sec:test}}

\subsubsection{Proof of Proposition~\ref{prop:infea}}

\begin{proof}[Proof of Proposition~\ref{prop:infea}]
In view of Lemma~3 in \citet{MR3842884}, the claim readily follows from the theory of degenerate U-statistics \citep[Chap.~5.5.2]{MR595165}.
\end{proof}

\subsubsection{Proof of Theorem~\ref{thm:smallop}}

\begin{proof}[Proof of Theorem~\ref{thm:smallop}]
For $k=1,2$,
let $\P_{J_k,d_k}^{(n)}$ and $\P_{J_k,d_k}$ denote the distributions of $\mW_{k1}^{(n)}$ and~$\mW_{k1}$, respectively, 
and let again $\mY^{(n)}_{ki}$ and $\mY_{ki}$  stand for $\fG_{k,\pms}^{(n)}(\mX_{ki})$ and $\fG_{k,\pms}(\mX_{ki})$, respectively. 
Consider the Hoeffding decomposition 
\begin{equation}\label{eqn:Hdec-utildeW}
\tenq{W}\n_{\mu}
=\sum_{\ell=1}^{m}\underbrace{\mbinom{m}{\ell}\mbinom{n}{\ell}^{-1}
  \sum_{1\le i_1<\cdots< i_{\ell}\le n}\widetilde h_{\mu,\ell}\Big(
  (\mY^{(n)}_{1i_1},\mY^{(n)}_{2i_1}),\dots,(\mY^{(n)}_{1i_\ell},\mY^{(n)}_{2i_\ell});
  \P_{J_1,d_1}^{(n)}\otimes \P_{J_2,d_2}^{(n)}\Big)}_{\displaystyle \tenq H_{n,\ell}},
\end{equation}
of $\tenq{W}\n_{\mu}$ with respect to the product measure $\P_{J_1,d_1}^{(n)}\otimes \P_{J_2,d_2}^{(n)}$ and the Hoeffding decomposition 
\begin{equation}\label{eqn:Hdec-W}
 W_{\mu}
=\sum_{\ell=1}^{m}\underbrace{\mbinom{m}{\ell}\mbinom{n}{\ell}^{-1}
  \sum_{1\le i_1<\cdots< i_{\ell}\le n}\widetilde h_{\mu,\ell}\Big(
  (\mY_{1i_1},\mY_{2i_1}),\dots,(\mY_{1i_\ell},\mY_{2i_\ell});
  \P_{J_1,d_1}\otimes \P_{J_2,d_2}\Big)}_{\displaystyle H_{n,\ell}}.
\end{equation}
of $W_{\mu}$ with respect to product measure $\P_{J_1,d_1}\otimes \P_{J_2,d_2}$.

The proof is divided into three steps. 
The first step shows that $n\tenq H_{n,1}=nH_{n,1}=0$, the second step   that 
$n\tenq H_{n,2}-nH_{n,2}=o_{\Pr}(1)$.
The third step verifies that $n\tenq H_{n,\ell}$ and $nH_{n,\ell}$, $\ell=3,4,\dots,m$ all  are $o_{\Pr}(1)$ terms.\medskip

{\bf Step I.}
Lemma~3 in \citet{MR3842884} confirms that 
\[
\widetilde h_{\mu,1}(\cdot;\P_{J_1,d_1}^{(n)}\otimes \P_{J_2,d_2}^{(n)})=0=\widetilde h_{\mu,1}(\cdot;\P_{J_1,d_1}\otimes \P_{J_2,d_2}),
\]
and thus 
$n\tenq H_{n,1}=nH_{n,1}=0$. \medskip

{\bf Step II.}
Lemma~3 in \citet{MR3842884} shows that, 
\begin{align*}
\mbinom{m}{2}\cdot\widetilde h_{\mu,2}\Big((\bm{y}_{11},\bm{y}_{21}),(\bm{y}_{12},\bm{y}_{22});\P_{J_1,d_1}^{(n)}\otimes \P_{J_2,d_2}^{(n)}\Big)
&= g_{1}^{(n)}(\bm{y}_{11},\bm{y}_{12})g_{2}^{(n)}(\bm{y}_{21},\bm{y}_{22}),\\
\text{and}~~~
\mbinom{m}{2}\cdot\widetilde h_{\mu,2}\Big((\bm{y}_{11},\bm{y}_{21}),(\bm{y}_{12},\bm{y}_{22});\P_{J_1,d_1}\otimes \P_{J_2,d_2}\Big)
&= g_{1}(\bm{y}_{11},\bm{y}_{12})g_{2}(\bm{y}_{21},\bm{y}_{22}),
\end{align*}
where $g_{k}^{(n)}$ and $g_{k}$ are defined in \eqref{eq:WDMlemma} and \eqref{eq:suey}. 
To prove that $n\tenq H_{n,2}-nH_{n,2}=o_{\Pr}(1)$, it suffices to show~that
\begin{align*}
   \E\big[(&n\tenq H_{n,2}\!\! -nH_{n,2})^2\big]\\
&=\E\Big[\Big(\frac{1}{n-1}\sum_{(i,j)\in I_{2}^{n}}g_{1}^{(n)}(\mY^{(n)}_{1i},\mY^{(n)}_{1j})
                                                  g_{2}^{(n)}(\mY^{(n)}_{2i},\mY^{(n)}_{2j})
           -\frac{1}{n-1}\sum_{(i,j)\in I_{2}^{n}}g_{1}(\mY_{1i},\mY_{1j})
                                                  g_{2}(\mY_{2i},\mY_{2j})\Big)^2\Big]\\
&=o(1).\yestag\label{eq:dCovHallin}
\end{align*}
We proceed in three sub-steps.

{\bf Step II-1.}
The theory of degenerate U-statistics (cf.\ Equation (7) of Section
1.6 in \citet{MR1075417}) yields that
\[
 \E\big[(nH_{n,2})^2\big]
=\frac{2n}{n-1}\E\big[g_{1}(\mY_{11},\mY_{12})^2\big]\E\big[g_{2}(\mY_{21},\mY_{22})^2\big].
\yestag\label{eq:dCovtailterm}
\]

{\bf Step II-2.}
We next deduce that
\begin{align*}
   \E\big[(&n\tenq H_{n,2})(nH_{n,2})\big]\\
=\;&\E\Big[\Big(\frac{1}{n-1}\sum_{(i,j)\in I_{2}^{n}}g_{1}^{(n)}(\mY^{(n)}_{1i},\mY^{(n)}_{1j})
                                                      g_{2}^{(n)}(\mY^{(n)}_{2i},\mY^{(n)}_{2j})\Big)
       \Big(\frac{1}{n-1}\sum_{(i,j)\in I_{2}^{n}}g_{1}(\mY_{1i},\mY_{1j})
                                                  g_{2}(\mY_{2i},\mY_{2j})\Big)\Big]\\
\to&\; 2\E\big[g_{1}(\mY_{11},\mY_{12})^2\big]\E\big[g_{2}(\mY_{21},\mY_{22})^2\big].\yestag\label{eq:dCovcrossterm}
\end{align*}
By   symmetry, we have
\begin{align*}
 \E\big[g_{k}^{(n)}(\mY^{(n)}_{ki},\mY^{(n)}_{kj})g_{k}(\mY_{ki},\mY_{kj})\big]
&=\E\big[g_{k}^{(n)}(\mY^{(n)}_{k1},\mY^{(n)}_{k2})g_{k}(\mY_{k1},\mY_{k2})\big]=:A_{k}^{(n)}, \\
 \E\big[g_{k}^{(n)}(\mY^{(n)}_{k\ell},\mY^{(n)}_{kj})g_{k}(\mY_{ki},\mY_{kj})\big]&=\E\big[g_{k}^{(n)}(\mY^{(n)}_{ki},\mY^{(n)}_{kr})g_{k}(\mY_{ki},\mY_{kj})\big]\\
&=\E\big[g_{k}^{(n)}(\mY^{(n)}_{k1},\mY^{(n)}_{k3})g_{k}(\mY_{k1},\mY_{k2})\big]=:B_{k}^{(n)}, \\
 \E\big[g_{k}^{(n)}(\mY^{(n)}_{k\ell},\mY^{(n)}_{kr})g_{k}(\mY_{ki},\mY_{kj})\big]
&=\E\big[g_{k}^{(n)}(\mY^{(n)}_{k3},\mY^{(n)}_{k4})g_{k}(\mY_{k1},\mY_{k2})\big]=:C_{k}^{(n)}
\end{align*}
for all distinct $i,j,\ell,r$, and also
\begin{align*}
A_{k}^{(n)}
   =\;&\E\big[g_{k}^{(n)}(\mY^{(n)}_{k1},\mY^{(n)}_{k2})g_{k}(\mY_{k1},\mY_{k2})\big],\yestag\label{eq:convAn}\\
A_{k}^{(n)}+(n-2)B_{k}^{(n)}
   =\;&\E\big[g_{k}^{(n)}(\mY^{(n)}_{k1},\mY^{(n)}_{k2})g_{k}(\mY_{k1},\mY_{k2})\big]
   +\sum_{\ell:\ell\ne 1,2}
       \E\big[g_{k}^{(n)}(\mY^{(n)}_{k\ell},\mY^{(n)}_{k2})g_{k}(\mY_{k2},\mY_{k2})\big] \\
  =\;&-\E\big[g_{k}^{(n)}(\mY^{(n)}_{k2},\mY^{(n)}_{k2})g_{k}(\mY_{k1},\mY_{k2})\big],\yestag\label{eq:convBn}\\
2B_{k}^{(n)}+(n-3)C_{k}^{(n)}
  =\;&\E\big[g_{k}^{(n)}(\mY^{(n)}_{k3},\mY^{(n)}_{k1})g_{k}(\mY_{k1},\mY_{k2})\big]
     +\E\big[g_{k}^{(n)}(\mY^{(n)}_{k3},\mY^{(n)}_{k2})g_{k}(\mY_{k1},\mY_{k2})\big]\\
  \;&+\sum_{\ell:\ell\ne 1,2,3}
       \E\big[g_{k}^{(n)}(\mY^{(n)}_{k3},\mY^{(n)}_{k\ell})g_{k}(\mY_{k1},\mY_{k2})\big]\\
  =\;&-\E\big[g_{k}^{(n)}(\mY^{(n)}_{k3},\mY^{(n)}_{k3})g_{k}(\mY_{k1},\mY_{k2})\big].\yestag\label{eq:convCn}
\end{align*}

We claim that
\begin{align*}
A_{k}^{(n)}\to\;&\E\big[g_{k}(\mY_{k1},\mY_{k2})^2\big], \yestag\label{eq:convA}\\
A_{k}^{(n)}+(n-2)B_{k}^{(n)}\to\;&-\E\big[g_{k}(\mY_{k2},\mY_{k2})g_{k}(\mY_{k1},\mY_{k2})\big]=0,\yestag\label{eq:convB}\\
2B_{k}^{(n)}+(n-3)C_{k}^{(n)}\to\;&-\E\big[g_{k}(\mY_{k3},\mY_{k3})g_{k}(\mY_{k1},\mY_{k2})\big]=0.\yestag\label{eq:convC}
\end{align*}
We only prove \eqref{eq:convA}, as \eqref{eq:convB} and \eqref{eq:convC} are quite similar.

If Condition \eqref{eq:key-condition} holds,
we obtain, since~$\E\big[f_k([\mW_{ki_\ell}]_{\ell=1}^{m})^2\big]<\infty$,  that
\[\lVert g_{k}(\mY_{k1},\mY_{k2})\rVert_{\sf L^{1}}\le
\lVert g_{k}(\mY_{k1},\mY_{k2})\rVert_{\sf L^{2}}<\infty .\]   
To prove \eqref{eq:convA}, we still need to show that $\mY^{(n)}_{ki}\stackrel{\sf L^{2}}{\longrightarrow}\mY_{ki}$ for $k=1,2$. 
Since the scores   $J_k$, $k=1,2$ are weakly regular     (cf.~Definition~\ref{asp:score}) and square-integrable, we obtain 
\[\lim_{n\to\infty}n^{-1}\sum_{r=1}^{n}J^2\Big(\frac{r}{n+1}\Big)=\int_{0}^{1}J^2(u)\d u,\]
and thus $\E\lVert\mY^{(n)}_{ki}\rVert^2\to \E\lVert\mY_{ki}\rVert^2$. 
Notice also that $\mY^{(n)}_{ki}\stackrel{\sf a.s.}{\longrightarrow}\mY_{ki}$. 
Using Vitali's theorem \citep[Chap.~3,~Theorem~5.5]{MR3701383}
yields $\E\lVert\mY^{(n)}_{ki}-\mY_{ki}\rVert^2\to 0$. 

Because
$g_{k}^{(n)}(\bm{y}_{k1},\bm{y}_{k2})\rightrightarrows g_{k}(\bm{y}_{k1},\bm{y}_{k2})$, we have
\begin{align*}
&\E\big[\lvert g_{k}^{(n)}(\mY^{(n)}_{k1},\mY^{(n)}_{k2})-g_{k}(\mY^{(n)}_{k1},\mY^{(n)}_{k2})\rvert\cdot\lvert g_{k}(\mY_{k1},\mY_{k2})\rvert\big]\\
\le\;&\lVert g_{k}^{(n)}(\mY^{(n)}_{k1},\mY^{(n)}_{k2})-g_{k}(\mY^{(n)}_{k1},\mY^{(n)}_{k2})\rVert_{\sf L^{\infty}}\cdot\lVert g_{k}(\mY_{k1},\mY_{k2})\rVert_{\sf L^{1}}
\to 0.\yestag\label{eq:convA1}
\end{align*}
Next, since $g_{k}$ is Lipschitz-continuous, by the fact that $\mY^{(n)}_{ki}\stackrel{\sf L^{2}}{\longrightarrow}\mY_{ki}$,
\begin{align*}
&\E\big[\lvert g_{k}(\mY^{(n)}_{k1},\mY^{(n)}_{k2})-g_{k}(\mY_{k1},\mY_{k2})\rvert\cdot
\lvert g_{k}(\mY_{k1},\mY_{k2})\rvert \big]\\
\le\;&\lVert g_{k}(\mY^{(n)}_{k1},\mY^{(n)}_{k2})-g_{k}(\mY_{k1},\mY_{k2})\rVert_{\sf L^{2}}\cdot\lVert g_{k}(\mY_{k1},\mY_{k2})\rVert_{\sf L^{2}}
\to 0;\yestag\label{eq:convA2}
\end{align*}
Combining \eqref{eq:convA1} and \eqref{eq:convA2} yields
\eqref{eq:convA}.


Having established \eqref{eq:convA}--\eqref{eq:convC}, we obtain 
that
\[
A_{k}^{(n)}\to\E\big[g_{k}(\mY_{k1},\mY_{k2})^2\big],~~~
B_{k}^{(n)}=O(n^{-1})~~~\text{and}~~~
C_{k}^{(n)}=o(n^{-1}).
\yestag\label{eq:dCovcrosstermU}
\]
Plugging \eqref{eq:dCovcrosstermU} into the left-hand side of \eqref{eq:dCovcrossterm} gives
\begin{align*}
&\E\Big[\Big(\frac{1}{n-1}\sum_{(i,j)\in I_{2}^{n}}g_{1}^{(n)}(\mY^{(n)}_{1i},\mY^{(n)}_{1j})
                                                      g_{2}^{(n)}(\mY^{(n)}_{2i},\mY^{(n)}_{2j})\Big)
       \Big(\frac{1}{n-1}\sum_{(i,j)\in I_{2}^{n}}g_{1}(\mY_{1i},\mY_{1j})
                                                  g_{2}(\mY_{2i},\mY_{2j})\Big)\Big]\\
=\;&\frac{n(n-1)}{(n-1)^2}\Big\{2 A_{1}^{(n)}A_{2}^{(n)}+4(n-2)B_{1}^{(n)}B_{2}^{(n)}+(n-2)(n-3)C_{1}^{(n)}C_{2}^{(n)}\Big\}\\
\to\;&2\E\big[g_{1}(\mY_{11},\mY_{12})^2\big]\E\big[g_{2}(\mY_{21},\mY_{22})^2\big].
\end{align*}
This completes the proof of \eqref{eq:dCovcrossterm}.

{\bf Step II-3.}
In order to prove \eqref{eq:dCovHallin}, it remains to show that
\[
 \E\big[(n\tenq H_{n,2})^2\big]
\to 2\E\big[g_{1}(\mY_{11},\mY_{12})^2\big]\E\big[g_{2}(\mY_{21},\mY_{22})^2\big].\yestag\label{eq:dCovheadterm}\]
Notice that $n\tenq H_{n,2}$ is a double-indexed permutation statistic. 
Applying Equations (2.2)--(2.3) in \citet{MR857081} yields
 $
\E\big[n\tenq H_{n,2}\big]
=\; n\mu_{1}^{(n)}\mu_{2}^{(n)},
$ 
and
\begin{align*}
\Var(n\tenq H_{n,2})
=\;&\frac{4n^2}{(n-1)^3(n-2)^2}
\Big(\frac{\sum_{i=1}^{n}\{\zeta_{1i}^{(n)}\}^2}{n}\Big)
\Big(\frac{\sum_{i=1}^{n}\{\zeta_{2i}^{(n)}\}^2}{n}\Big)\\
   &+\frac{2n}{n-3}
\Big(\frac{\sum_{i\ne j}\{\eta_{1ij}^{(n)}\}^2}{n(n-1)}\Big)
\Big(\frac{\sum_{i\ne j}\{\eta_{2ij}^{(n)}\}^2}{n(n-1)}\Big),
\end{align*}
where for $k=1,2$, 
\begin{align*}
\mu_{k}^{(n)}:=\;&\frac{1}{n(n-1)}\sum_{i\ne j}g_{k}^{(n)}(\bm{y}^{(n)}_{ki}, \bm{y}^{(n)}_{kj}),\\
\zeta_{ki}^{(n)}:=\;&\sum_{j:j\ne i}\Big\{g_{k}^{(n)}(\bm{y}^{(n)}_{ki}, \bm{y}^{(n)}_{kj})-\mu_{k}^{(n)}\Big\},\\
\eta_{kij}^{(n)}:=\;&g_{k}^{(n)}(\bm{y}^{(n)}_{i}, \bm{y}^{(n)}_{j})
-\frac{\zeta_{ki}^{(n)}}{n-2}-\frac{\zeta_{kj}^{(n)}}{n-2}-\mu_{k}^{(n)}.
\end{align*}

Direct computation gives
\begin{align*}
\mu_{k}^{(n)}=\;&-\frac{1}{n(n-1)}\sum_{i=1}^{n}g_{k}^{(n)}(\bm{y}^{(n)}_{ki},\bm{y}^{(n)}_{ki}),\\
\zeta_{ki}^{(n)}=\;&-g_{k}^{(n)}(\bm{y}^{(n)}_{ki},\bm{y}^{(n)}_{ki})+\frac{1}{n}\sum_{j=1}^{n}g_{k}^{(n)}(\bm{y}^{(n)}_{kj},\bm{y}^{(n)}_{kj}),\\
\eta_{kij}^{(n)}=\;&g_{k}^{(n)}(\bm{y}^{(n)}_{i}, \bm{y}^{(n)}_{j})
+\frac{g_{k}^{(n)}(\bm{y}^{(n)}_{ki}, \bm{y}^{(n)}_{ki})}{n-2}
+\frac{g_{k}^{(n)}(\bm{y}^{(n)}_{kj}, \bm{y}^{(n)}_{kj})}{n-2}
-\frac{1}{(n-1)(n-2)}\sum_{i=1}^{n}g_{k}^{(n)}(\bm{y}^{(n)}_{ki},\bm{y}^{(n)}_{ki}).
\end{align*}

Moreover, we can write $\E\big[n\tenq H_{n,2}\big]$ and $\Var(n\tenq H_{n,2})$ in terms of $\mY_{k1}^{(n)}$ and~$\mY_{k2}^{(n)}$: 
\begin{align*}
\E\big[n\tenq H_{n,2}\big]
=\;&\frac{n}{(n-1)^2}\E\big[g_{1}^{(n)}(\mY^{(n)}_{11},\mY^{(n)}_{11})\big]\E\big[g_{2}^{(n)}(\mY^{(n)}_{21},\mY^{(n)}_{21})\big],\\
\Var(n\tenq H_{n,2})
=\;&\frac{4n^2}{(n-1)^3(n-2)^2}
\Var\big[g_{1}^{(n)}(\mY^{(n)}_{11},\mY^{(n)}_{11})\big]\Var\big[g_{2}^{(n)}(\mY^{(n)}_{21},\mY^{(n)}_{21})\big]\\
   &+\frac{2n}{n-3}
\Var\Big[g_{1}^{(n)}(\mY^{(n)}_{11},\mY^{(n)}_{12})
  +\frac{g_{1}^{(n)}(\mY^{(n)}_{11},\mY^{(n)}_{11})}{n-2}
  +\frac{g_{1}^{(n)}(\mY^{(n)}_{12},\mY^{(n)}_{12})}{n-2}\Big]\\
&\mkern60mu\times 
\Var\Big[g_{2}^{(n)}(\mY^{(n)}_{21},\mY^{(n)}_{22})
  +\frac{g_{2}^{(n)}(\mY^{(n)}_{21},\mY^{(n)}_{21})}{n-2}
  +\frac{g_{2}^{(n)}(\mY^{(n)}_{22},\mY^{(n)}_{22})}{n-2}\Big].
\end{align*}

Using once again  Condition \eqref{eq:key-condition}, and by a similar 
argument as in the  proof of~\eqref{eq:convA}, we obtain
\begin{align*}
\E\big[n\tenq H_{n,2}\big]
\to\;&\frac{n}{(n-1)^2}\E\big[g_{1}(\mY_{11},\mY_{11})\big]\E\big[g_{2}(\mY_{21},\mY_{21})\big]
\to 0,\yestag\label{eq:barbE}\\
\Var(n\tenq H_{n,2})
\to\;&\frac{4n^2}{(n-1)^3(n-2)^2}
\Var\big[g_{1}(\mY_{11},\mY_{11})\big]\Var\big[g_{2}(\mY_{21},\mY_{21})\big]\\
   &+\frac{2n}{n-3}
\Var\Big[g_{1}(\mY_{11},\mY_{12})
  +\frac{g_{1}(\mY_{11},\mY_{11})}{n-2}
  +\frac{g_{1}(\mY_{12},\mY_{12})}{n-2}\Big]\\
&\mkern60mu\times 
\Var\Big[g_{2}(\mY_{21},\mY_{22})
  +\frac{g_{2}(\mY_{21},\mY_{21})}{n-2}
  +\frac{g_{2}(\mY_{22},\mY_{22})}{n-2}\Big]\\
\to\;&2\E\big[g_{1}(\mY_{11},\mY_{12})^2\big]\E\big[g_{2}(\mY_{21},\mY_{22})^2\big].
\yestag\label{eq:barbVar-half}
\end{align*}
Combining \eqref{eq:barbE} and \eqref{eq:barbVar-half}, we deduce that \eqref{eq:dCovheadterm} holds. \smallskip

Finally, Step II is completed by combining \eqref{eq:dCovtailterm},
\eqref{eq:dCovcrossterm}, and \eqref{eq:dCovheadterm} to deduce \eqref{eq:dCovHallin}.\medskip

{\bf Step III.} 
Notice that $\sup_{i_1,\dots,i_m\in\zahl{m}}\E\big[f_k([\mW_{ki_\ell}]_{\ell=1}^{m})^2\big]<\infty$. 
Proving that $\E\big[(n\tenq H_{n,\ell})^2\big]=o(1)$ for~$\ell=3,4,\dots,m$ goes along the same steps as   the proof of Theorem 4.2 in the supplement of
\citet{shi2019distribution}; it is omitted here. 
The fact that $\E\big[(nH_{n,\ell})^2\big]=o(1)$, $\ell=3,4,\dots,m$ follows directly from the theory of degenerate U-statistics (cf.~Equation (7) of Section 1.6 in \citet{MR1075417}). The proof is thus complete.
\end{proof}

\subsubsection{Proof of Theorem~\ref{thm:smallop2}}

\begin{proof}[Proof of Theorem~\ref{thm:smallop2}]
The proof is similar to that of Theorem~\ref{thm:smallop}. The only difference lies in proving \eqref{eq:convA}--\eqref{eq:convC} and \eqref{eq:barbE}--\eqref{eq:barbVar-half}. 
\textcolor{black}{By the 
continuous mapping theorem \citep[Theorem~2.3]{MR1652247}
and the Skorokhod construction \citep[Chap.~3,~Theorem~5.7(viii)]{MR3701383}, we can assume,  without loss of generality, that $\mW^{(n)}_{ki}\stackrel{\sf a.s.}{\longrightarrow}\mW_{ki}$. If Condition \eqref{eq:key-condition2} holds, then
  \eqref{eq:convA} immediately follows from the dominated convergence
  theorem and the definitions of $g_k^{(n)}$ and $g_k$ in~\eqref{eq:WDMlemma} and \eqref{eq:suey}. 
The proofs for \eqref{eq:convB},  \eqref{eq:convC}, \eqref{eq:barbE}, and~\eqref{eq:barbVar-half}  are similar. }
\end{proof}

\subsubsection{Proof of Proposition~\ref{ex:smallop}}

\paragraph{Proof of Proposition~\ref{ex:smallop} ($h=h_{\dCov^2}$)}

\begin{proof}[Proof of Proposition~\ref{ex:smallop} ($h=h_{\dCov^2}$)]
Condition \eqref{eq:uvar} is obvious. 
Condition \eqref{eq:pd} is satisfied in view of Theorem~5 in \citet{MR2382665}.
We next verify that condition \eqref{eq:key-condition} is satisfied. To do so, let us 
first show that $g_{k}^{(n)}(\bm{y}_{k1},\bm{y}_{k2})\rightrightarrows
g_{k}(\bm{y}_{k1},\bm{y}_{k2})$ for $k=1,2$. By definitions 
\eqref{eq:suey} and \eqref{eq:WDMlemma},
\begin{align*}
g_{k}^{(n)}(\bm{y}_{k1},\bm{y}_{k2})&:=
\lVert\bm{y}_{k1}-\bm{y}_{k2}\rVert - \E\lVert\bm{y}_{k1}-\mW_{k3}^{(n)}\rVert
- \E\lVert\mW_{k4}^{(n)}-\bm{y}_{k2}\rVert + \E\lVert\mW_{k4}^{(n)}-\mW_{k3}^{(n)}\rVert,\\
\text{and}~~~
g_{k}(\bm{y}_{k1},\bm{y}_{k2})&:=
\lVert\bm{y}_{k1}-\bm{y}_{k2}\rVert - \E\lVert\bm{y}_{k1}-\mW_{k3}\rVert
- \E\lVert\mW_{k4}-\bm{y}_{k2}\rVert + \E\lVert\mW_{k4}-\mW_{k3}\rVert.
\end{align*}
Noting that $J_k$, $k=1,2$ are continuous, we can assume, as in the proof of  Theorem~\ref{thm:smallop2},   that 
$\mW^{(n)}_{ki}\stackrel{\sf a.s.}{\longrightarrow}\mW_{ki}$.
Since the scores $J_k$ are square-integrable, we obtain that
$\E\lVert\mW^{(n)}_{ki}\rVert^2\to \E\lVert\mW_{ki}\rVert^2$. 
Using Vitali's theorem \citep[Chap.~3,~Theorem~5.5]{MR3701383}
yields $\mW^{(n)}_{ki}\stackrel{\sf L^{2}}{\longrightarrow}\mW_{ki}$. 
Therefore, we obtain
\begin{align*}
\Big\lvert\E\lVert\bm{y}_{k1}-\mW^{(n)}_{k3}\rVert-\E\lVert\bm{y}_{k1}-\mW_{k3}\rVert\Big\rvert&\le\E\lVert\mW^{(n)}_{k3}-\mW_{k3}\rVert,\\
\Big\lvert\E\lVert\mW^{(n)}_{k4}-\bm{y}_{k2}\rVert-\E\lVert\mW_{k4}-\bm{y}_{k2}\rVert\Big\rvert&\le\E\lVert\mW^{(n)}_{k4}-\mW_{k4}\rVert,\\
\Big\lvert\E\lVert\mW_{k4}^{(n)}-\mW_{k3}^{(n)}\rVert-\E\lVert\mW_{k4}-\mW_{k3}\Big\rvert&\le\E\lVert\mW^{(n)}_{k3}-\mW_{k3}\rVert+\E\lVert\mW^{(n)}_{k4}-\mW_{k4}\rVert,
\end{align*}
and, furthermore,
\[
\Big\lvert g_{k}^{(n)}(\bm{y}_{k1},\bm{y}_{k2})-g_{k}(\bm{y}_{k1},\bm{y}_{k2})\Big\rvert
\le 2\Big(\E\lVert\mW^{(n)}_{k3}-\mW_{k3}\rVert+\E\lVert\mW^{(n)}_{k4}-\mW_{k4}\rVert\Big).
\]
The uniform convergence $g_{k}^{(n)}(\bm{y}_{k1},\bm{y}_{k2})\rightrightarrows g_{k}(\bm{y}_{k1},\bm{y}_{k2})$ follows. 
It is obvious that $g_{k}(\bm{y}_{k1},\bm{y}_{k2})$ is Lipschitz-continuous, and $\E\big[f_k(\mW_{ki_1},\dots,\mW_{ki_4})^2\big]<\infty$ for all $i_1,\dots,i_4\in\zahl{4}$ as long as $J_1,J_2$ are weakly regular. 
\end{proof}

\paragraph{Proof of Proposition~\ref{ex:smallop} ($h=h_{M},\, h_{D}, \, h_{R}, \, h_{\tau^*}$)}

\begin{proof}[Proof of Proposition~\ref{ex:smallop} ($h=h_{M},\, h_{D}, \, h_{R}, \, h_{\tau^*}$)]
Condition \eqref{eq:uvar} is obvious. 
Condition \eqref{eq:pd} is satisfied for $h_{D}$ by Theorem 3(i) in \citet{MR3737307}. 
For $h=h_{M},\, h_{R}, \, h_{\tau^*}$, we can prove condition \eqref{eq:pd} holds as well in a similar way. 
It is clear that condition \eqref{eq:key-condition2} is satisfied for all these four kernel functions. 
\end{proof}

\subsubsection{Proof of Corollary~\ref{thm:null}}

\begin{proof}[Proof of Corollary~\ref{thm:null}]
Combining Proposition~\ref{prop:infea} and Theorem~\ref{thm:smallop},
one immediately obtains the  limiting null distribution of the rank-based 
statistic $\tenq{W}\n_\mu$. 
\end{proof}

\subsubsection{Proof of Proposition~\ref{prop:uvc}}

\begin{proof}[Proof of Proposition~\ref{prop:uvc}]
Validity is a direct corollary of Corollary~\ref{thm:null}. 
Uniform validity  then follows from  validity and   exact distribution-freeness.  
For any fixed alternative in $\cP_{d_1,d_2,\infty}^{\ac}$, it holds that $\tenq W_\mu\n\stackrel{\sf a.s.}{\longrightarrow}\mu_{\pms}(\mX_1,\mX_2)>0$ as $n\to\infty$. Thus, $n\,\tenq W_\mu\n\stackrel{\sf a.s.}{\longrightarrow}\infty$ and the result follows. 
\end{proof}

\subsubsection{Proof of Theorem~\ref{thm:power-general}}

\begin{proof}[Proof of Theorem~\ref{thm:power-general}]

Let $\mX ^*_{ni}$ and $\mX_{ni}$, $i\in\zahl{n}$ be independent copies of $\mX ^*$ and $\mX$ with $\delta=\delta\n$, respectively. 
Let $\P\n:=\otimes_{i=1}^{n}\P\n_{i}$,
$\Q\n:=\otimes_{i=1}^{n}\Q\n_{i}$, where $\P\n_{i}$ and $\Q\n_{i}$ are
the distributions of~$\mX ^*_{ni}$ and $\mX_{ni}$, respectively. 
Define 
\[\Lambda\n:=\log\frac{\d\Q\n}{\d\P\n}=\sum_{i=1}^{n}\log \frac{q_{\mX}(\mX ^*_{ni};\delta\n)}{q_{\mX}(\mX ^*_{ni};0)}\quad\text{and}\quad
T\n:=\delta\n\sum_{i=1}^{n}\dot{\ell}(\mX ^*_{ni};0).
\]

We proceed in three steps. 
First, we clarify that $\Q\n$ is contiguous to $\P\n$ in order for  Le~Cam's third lemma  \citep[Theorem~6.6]{MR1652247} to be applicable. 
Next, we derive the joint limiting null distribution of $(n\tenq{W}\n_{\mu}, \Lambda\n)^{\top}$. 
Lastly, we employ Le~Cam's third lemma to obtain the asymptotic distribution of $(n\tenq{W}\n_{\mu}, \Lambda\n)^{\top}$ under contiguous  alternatives.

{\bf Step I.} In view of 
\citet[Example~12.3.7]{MR2135927}, 
Assumption~\ref{asp:only} entails  the contiguity $\Q\n\triangleleft \P\n$.

{\bf Step II.} Next, we   derive the limiting joint distribution of $(n\tenq{W}\n_{\mu}, \Lambda\n)^{\top}$ under the null hypothesis.
To this end, we first obtain the limiting null distribution of $(nH_{n,2}, T\n)^{\top}$, where $H_{n,2}$ is defined in \eqref{eqn:Hdec-W}. 
By condition \eqref{eq:uvar}, we write
\[
H_{n,2}=
\frac{1}{n(n-1)}\sum_{i\ne j}
  \sum_{v=1}^{\infty}
  \lambda_{v}
  \psi_{v}(\mY_{1i},\mY_{2i})
  \psi_{v}(\mY_{1j},\mY_{2j}),
\]
where $\psi_{v}$ is the normalized eigenfunction associated with $\lambda_{v}$ and $\mY_{ki}=\fG ^*_{k,\pms}(\mX ^*_{ki})$ for $k=1,2$.  
For each positive integer $K$, consider the ``truncated'' U-statistic 
\[
H_{n,2,K}:=
\frac{1}{n(n-1)}\sum_{i\ne j}
  \sum_{v=1}^{K}
  \lambda_{v}
  \psi_{v}(\mY_{1i},\mY_{2i})
  \psi_{v}(\mY_{1j},\mY_{2j}).
\]
Note that $nH_{n,2}$ and $nH_{n,2,K}$ can be written as
\begin{align*}
n H_{n,2}&=\frac{n}{n-1}\Big\{\sum_{v=1}^{\infty}\lambda_{v} \Big(\sum_{i=1}^n \frac{\psi_{v}(\mY_{1i},\mY_{2i})}{\sqrt{n}}\Big)^2
-\sum_{v=1}^{\infty}\lambda_{v} \Big(\frac{\sum_{i=1}^n \{\psi_{v}(\mY_{1i},\mY_{2i})\}^2}{n}\Big)\Big\},\\
n H_{n,2,K}&=\frac{n}{n-1}\Big\{\sum_{v=1}^{K}\lambda_{v} \Big(\sum_{i=1}^n \frac{\psi_{v}(\mY_{1i},\mY_{2i})}{\sqrt{n}}\Big)^2
-\sum_{v=1}^{K}\lambda_{v} \Big(\frac{\sum_{i=1}^n \{\psi_{v}(\mY_{1i},\mY_{2i})\}^2}{n}\Big)\Big\}.
\end{align*}

To obtain the limiting null distribution of $(nH_{n,2}, T\n)^{\top}$, 
  first consider the limiting null distribution,  for fixed $K$, of $(nH_{n,2,K}, T\n)^{\top}$.
Let $S_{n,v}$ be a shorthand for $n^{-1/2}\sum_{i=1}^n {\psi_{v}(\mY_{1i},\mY_{2i})}$ and  observe that 
\[
\E[S_{n,v}]=\E[T\n]=0,  ~~~ 
\Var[S_{n,v}]=1,        ~~~
\Var[T\n]=\cI_{\mX}(0), ~~~\text{and}~~~
\Cov[S_{n,v}, T\n]\to \gamma_v\delta_0.
\]
where $\gamma_v:=\Cov\big[\psi_{v}(\mY_1,\mY_2),\dot{\ell}\big((\fG_{1,\pms}^{-1}(\mY_{1}), \fG_{2,\pms}^{-1}(\mY_{2}));0\big)\big]$. 
There exists at least one $v\geq 1$ such that~$\gamma_v\neq 0$. 
Indeed, applying Lemma~4.2 in \citet{MR3541972} yields 
\[
 \big\{\psi_{v}(\my)\big\}_{v\in\Z_{>0}}
=\big\{\psi_{1,v_1}(\my_1)
       \psi_{2,v_2}(\my_2)\big\}_{v_1,v_2\in\Z_{>0}},
\]
where $\psi_{k,v}(\my_k), v\in\Z_{>0}$ 
are eigenfunctions associated with the non-zero eigenvalues of the integral equations
$
\E[g_k(\my_{k},\mW_{k2})\psi_k(\mW_{k2})]=\lambda_{k}\psi_k(\my_{k})
$
for $k=1,2$. 
Since
$
\big\{\psi_{1,v_1}(\my_1)
      \psi_{2,v_2}(\my_2)\big\}_{v_1,v_2\in\Z_{\ge0}}
$
(where $\psi_{k,v}(\my_k):=1$ for $v=0$, $k=1,2$) forms a complete orthogonal basis of the set of square integrable functions, 
$\gamma_v=0$ for all $v\ge 1$ thus entails 
that $\dot{\ell}(\mx;0)$ is additively separable, 
which contradicts Assumption~\ref{asp:only}\ref{asp:only4}.
Therefore,  $\gamma_{v^*}\ne0$ for some $v^*$. 
Applying the multivariate central limit theorem \citep[Equation~(18.24)]{MR855460}, we deduce
\[
(S_{n,1}, \dots,S_{n,K}, T\n)^{\top}
\stackrel{\P\n}{\rightsquigarrow}
(\xi_{1}, \dots,\xi_{K}, V_K)^{\top}
\sim N_{K+1}\bigg(\bigg(\begin{matrix}\bm0_{K}\\ 0\end{matrix}\bigg),
\bigg(\begin{matrix}\fI_p & \delta_0\mv\\
\delta_0\mv^{\top} & \delta_0^2\cI\end{matrix}\bigg)\bigg), 
\]
where $\cI:=\cI_{\mX}(0)$ and $\mv=(\gamma_1,\dots,\gamma_K)^{\top}$.  Thus, $V_K$ can be expressed as 
\[\Big(\delta_0^2\cI\Big)^{1/2}\Big\{\sum_{v=1}^{K}c_v\xi_{v} + \Big(1-\sum_{v=1}^{K}c_v^2\Big)^{1/2}\xi_0\Big\}\]
where $c_v:=\cI^{-1/2}\gamma_v$, 
and $\xi_0$ is standard Gaussian,  independent of $\xi_{1},\dots,\xi_{K}$.
Then, by the continuous mapping theorem \citep[Theorem~2.3]{MR1652247} 
and Slutsky's theorem \citep[Theorem~2.8]{MR1652247},
\[
(nH_{n,2,K}, T\n)^{\top}
\stackrel{\P\n}{\rightsquigarrow}
\bigg(\sum_{v=1}^{K}\lambda_v(\xi_{v}^2-1),\Big(\delta_0^2\cI\Big)^{1/2}\Big\{\sum_{v=1}^{K}c_v\xi_{v} + \Big(1-\sum_{v=1}^{K}c_v^2\Big)^{1/2} \xi_0\Big\}\bigg)^{\top}
\yestag\label{eq:jointKlimit}
\]
for any $K$. This entails 
\[
(nH_{n,2}, T\n)^{\top}
\stackrel{\P\n}{\rightsquigarrow}
\bigg(\sum_{v=1}^{\infty}\lambda_v(\xi_{v}^2-1),\Big(\delta_0^2\cI\Big)^{1/2}\Big\{\sum_{v=1}^{\infty}c_v\xi_{v} + \Big(1-\sum_{v=1}^{\infty}c_v^2\Big)^{1/2}\xi_0\Big\}\bigg)^{\top}.
\yestag\label{eq:jointlimit}
\]
Indeed, putting 
\begin{align*}
M_K&:= \sum_{v=1}^{K}\lambda_v(\xi_{v}^2-1), &  
V_K&:= \Big(\delta_0^2\cI\Big)^{1/2}\Big\{\sum_{v=1}^{K}c_v \xi_{v} + \Big(1-\sum_{v=1}^{K}c_v^2\Big)^{1/2} \xi_0\Big\},\\
M&:= \sum_{v=1}^{\infty}\lambda_v(\xi_{v}^2-1), &  \text{and}~~~
V&:= \Big(\delta_0^2\cI\Big)^{1/2}\Big\{\sum_{v=1}^{\infty}c_v \xi_{v} + \Big(1-\sum_{v=1}^{\infty}c_v^2\Big)^{1/2} \xi_0\Big\},
\end{align*}
 it suffices, in order to to prove \eqref{eq:jointlimit}, to show that, for any $a,b\in\R$,  
\[
\Big\lvert\E\Big[\exp\Big\{\sfi anH_{n,2}+\sfi bT\n\Big\}\Big] -
\E\Big[\exp\Big\{\sfi aM+\sfi bV\Big\}\Big]\Big\rvert \to 0~~~\text{as}~n\to\infty.
\yestag\label{eq:charlimit}
\]
We have
\begin{align*}
&\Big\lvert\E\Big[\exp\Big\{\sfi anH_{n,2}+\sfi bT\n\Big\}\Big] -
\E\Big[\exp\Big\{\sfi aM+\sfi bV\Big\}\Big]\Big\rvert \\
\le\;&\Big\lvert\E\Big[\exp\Big\{\sfi anH_{n,2}+\sfi bT\n\Big\}\Big] -
\E\Big[\exp\Big\{\sfi anH_{n,2,K}+\sfi bT\n\Big\}\Big]\Big\rvert \\
&+\Big\lvert\E\Big[\exp\Big\{\sfi anH_{n,2,K}+\sfi bT\n\Big\}\Big] -
\E\Big[\exp\Big\{\sfi aM_K+\sfi bV_K\Big\}\Big]\Big\rvert \\
&+\Big\lvert\E\Big[\exp\Big\{\sfi aM_K+\sfi bV_K\Big\}\Big] -
\E\Big[\exp\Big\{\sfi aM+\sfi bV\Big\}\Big]\Big\rvert=:I+I\!I+I\!I\!I,\quad\text{say,}
\end{align*}
where it follows from  page 82 of \citet{MR1075417}
and Equation~(4.3.10) in \citet{MR1472486} that 
\begin{align*}
I 
&\le\E\Big\lvert \exp\Big\{\sfi an(H_{n,2}-H_{n,2,K})\Big\}-1\Big\rvert
\le\Big\{\E\Big\lvert an(H_{n,2}-H_{n,2,K})\Big\rvert^2\Big\}^{1/2}
=\Big\{\frac{2na^2}{n-1}\sum_{v=K+1}^{\infty}\lambda_v^2\Big\}^{1/2}
\end{align*}
and
\begin{align*}
I\!I\!I 
&\le\E\Big\lvert \exp\Big\{\sfi a(M_K-M)+\sfi b(V_K-V)\Big\}-1\Big\rvert
\le\Big\{\E\Big\lvert a(M_K-M)+b(V_K-V)\Big\rvert^2\Big\}^{1/2}\\
&\le\Big\{2\Big(2a^2\sum_{v=K+1}^{\infty}\lambda_v^2 + 
2b^2\delta_0^2\cI\sum_{v=K+1}^{\infty}c_v^2\Big)\Big\}^{1/2}.
\end{align*}
Since by condition \eqref{eq:uvar}
\[
\sum_{v=1}^{\infty}\lambda_v^2=\Var(g_{1}(\mW_{11},\mW_{12}))\cdot\Var(g_{2}(\mW_{21},\mW_{22}))\in(0,\infty)
~~~\text{and}~~~
\sum_{v=1}^{\infty}c_v^2=\cI^{-1}\sum_{v=1}^{\infty}\gamma_v^2\le 1, 
\]
we conclude that, for any $\epsilon>0$, there exists $K_0$ such that $I<\epsilon/3$ and $I\!I\!I<\epsilon/3$ for all $n$ and all~$K\ge K_0$. For this $K_0$, we also have, by \eqref{eq:jointKlimit}, that  $I\!I<\epsilon/3$ for all $n$ sufficiently large;  
 \eqref{eq:charlimit}, hence \eqref{eq:jointlimit},  follow.
 
Now, as in  
\citet[Theorem~7.2]{MR1652247}, 
\[\Lambda\n-T\n+\delta_0^2\cI/2\stackrel{\P\n}{\longrightarrow} 0. \yestag\label{eq:firstapprox}\]
Combining \eqref{eq:jointlimit} and \eqref{eq:firstapprox} yields
\[
(nH_{n,2}, \Lambda\n)^{\top}
\stackrel{\P\n}{\rightsquigarrow}
\bigg(\sum_{v=1}^{\infty}\lambda_v(\xi_{v}^2-1),\Big(\delta_0^2\cI\Big)^{1/2}\Big\{\sum_{v=1}^{\infty}c_v\xi_{v} + \Big(1-\sum_{v=1}^{\infty}c_v^2\Big)^{1/2}\xi_0\Big\}-\frac{\delta_0^2\cI}{2}\bigg)^{\top}.
\yestag\label{eq:jointlimitL}
\]
 Equation (1.6.7) in \citet[p.~30]{MR1075417}, along with  the fact that $H_{n,1}=0$, implies that $(n\tenq W\n_\mu, \Lambda\n)^{\top}$ has the same limiting distribution as \eqref{eq:jointlimitL} under $\P\n$. 

{\bf Step III.} 
Finally we employ the general form \citep[Theorem~6.6]{MR1652247} of Le Cam's third lemma, which by condition \eqref{eq:pd} entails  
\begin{align*}
   \;& \Q\n(n\tenq{W}\n_{\mu}\le q_{1-\alpha})\\
\to\;& \E \Big[\ind\Big(\sum_{v=1}^{\infty}\lambda_v(\xi_{v}^2-1)\le q_{1-\alpha}\Big) \cdot\exp\Big\{\Big(\delta_0^2\cI\Big)^{1/2}\Big(\sum_{v=1}^{\infty}c_v \xi_{v} + \Big(1-\sum_{v=1}^{\infty}c_v^2\Big)^{1/2} \xi_0\Big)-\frac{\delta_0^2\cI}{2}\Big\}\Big]\\
\le\;& \E \Big[\ind\Big\{\Big\lvert \xi_{v^*}\Big\rvert\le \Big(\frac{q_{1-\alpha}+\sum_{v=1}^{\infty}\lambda_v}{\lambda_{v^*}}\Big)^{1/2}\Big\}  
\cdot\exp\Big\{\Big(\delta_0^2\cI\Big)^{1/2}\Big(\sum_{v=1}^{\infty}c_v \xi_{v} + \Big(1-\sum_{v=1}^{\infty}c_v^2\Big)^{1/2} \xi_0\Big)-\frac{\delta_0^2\cI}{2}\Big\}\Big]\\
=\;& \E \Big[\ind\Big\{\Big\lvert \xi_{v^*}\Big\rvert\le \Big(\frac{q_{1-\alpha}+\sum_{v=1}^{\infty}\lambda_v}{\lambda_{v^*}}\Big)^{1/2}\Big\} 
\cdot\exp\Big\{\Big(\delta_0^2\cI\Big)^{1/2}\Big(c_{v^*} \xi_{v^*} + \Big(1-c_{v^*}^2\Big)^{1/2} \xi_0\Big)-\frac{\delta_0^2\cI}{2}\Big\}\Big]\\
=\;& \Phi\Big(\Big(\frac{q_{1-\alpha}+\sum_{v=1}^{\infty}\lambda_v}{\lambda_{v^*}}\Big)^{1/2} -c_{v^*}\Big(\delta_0^2\cI\Big)^{1/2}\Big) - \Phi\Big(-\Big(\frac{q_{1-\alpha}+\sum_{v=1}^{\infty}\lambda_v}{\lambda_{v^*}}\Big)^{1/2} -c_{v^*}\Big(\delta_0^2\cI\Big)^{1/2}\Big)\\
\le\;& 2\Big(\frac{q_{1-\alpha}+\sum_{v=1}^{\infty}\lambda_v}{\lambda_{v^*}}\Big)^{1/2} \varphi\Big(\Big\{\lvert c_{v^*}\rvert\cdot\Big(\delta_0^2\cI\Big)^{1/2}-\Big(\frac{q_{1-\alpha}+\sum_{v=1}^{\infty}\lambda_v}{\lambda_{v^*}}\Big)^{1/2}\Big\}_+\Big),
\end{align*}
a quantity which is arbitrarily small for   large enough $\delta_0$, irrespective of the sign of $c_{v^*}$. 
\end{proof}

\subsubsection{Proof of Theorem~\ref{thm:opt}}

\begin{proof}[Proof of Theorem~\ref{thm:opt}]

This result is a standard result connecting the Fisher information to
the usual lower bound of rate $n^{-1/2}$ \textcolor{black}{\citep[Chap.~6]{MR3445293}}. Recall that $\mX ^*_{ni}$ and~$\mX_{ni}$, $i\in\zahl{n}$ are independent copies of~$\mX ^*$ and~$\mX$, respectively, with $\delta=\delta\n=n^{-1/2}\delta_0$. 
Recall $\P\n:=\otimes_{i=1}^{n}\P\n_{i}$,
$\Q\n:=\otimes_{i=1}^{n}\Q\n_{i}$, where~$\P\n_{i}$ and~$\Q\n_{i}$ are
the distributions of $\mX ^*_{ni}$ and $\mX_{ni}$, respectively. 
It suffices to prove that for any small~$0<~\!\beta~\!<1~\!-~\!\alpha$, there exists $|\delta_0|=c_\beta$ such that, for all sufficiently large $n$, $\TV(\Q\n,\P\n)<\beta$, which is implied by~$\HL(\Q\n,\P\n)<\beta$ using the fact that  total variation and Hellinger distances satisfy 
\[
\TV(\Q\n,\P\n)\le\HL(\Q\n,\P\n)
\]
\citep[Equation~(2.20)]{MR2724359}. 
It is also known \citep[p.~83]{MR2724359}  that
\[
1-\frac{\HL^2(\Q\n,\P\n)}{2}=\prod_{i=1}^{n}\Big(1-\frac{\HL^2(\Q\n_{i},\P\n_{i})}{2}\Big).
\]
\citet[Example~13.1.1]{MR2135927} 
shows that, under Assumption~\ref{asp:only}, 
\[
n\times \HL^2(\Q\n_{i},\P\n_{i})
\to \frac{\delta_0^2\cI_{\mX}(0)}{4};
\]
notice that here the definition of $\HL^2(\Q,\P)$ differs with that in \citet[Definition~13.1.3]{MR2135927} by a factor of $2$. 
Therefore, 
\[
1-\frac{\HL^2(\Q\n,\P\n)}{2} \longrightarrow \exp\Big\{-\frac{\delta_0^2\cI_{\mX}(0)}{8}\Big\}.
\]
The desired result follows by taking $c_{\beta}>0$ such that
\[
\exp\Big\{-\frac{c_{\beta}^2\cI_{\mX}(0)}{8}\Big\}=1-\frac{\beta^2}{8}.
\]
This completes the proof.
\end{proof}

\subsubsection{Proof of Example~\ref{ex:ellip}}


\paragraph{Proof of Example~\ref{ex:ellip}\ref{ex:ellip1}}
\begin{proof}[Proof of Example~\ref{ex:ellip}\ref{ex:ellip1}]
We need to verify Assumption~\ref{asp:power-opt}. 
Items \ref{asp:power-opt1} and \ref{asp:power-opt2} are obvious. 
For~\ref{asp:power-opt3}, following the proof of Lemma~3.2.1 in \citet{MR2691505}, when $\mX ^*_1$ and $\mX ^*_2$ are elliptically symmetric with   parameters $\bm0_{d_1},\, \mSigma_1$ and $\bm0_{d_2},\, \mSigma_2$, respectively, we obtain
\[
\dot{\ell}(\mx;0)=
-2(\fM_1\mx_2)^{\top}\mSigma_1^{-1}\mx_1 \cdot \rho_1\Big(\mx_1^{\top}\mSigma_1^{-1}\mx_1\Big)
-2(\fM_2\mx_1)^{\top}\mSigma_2^{-1}\mx_2 \cdot \rho_2\Big(\mx_2^{\top}\mSigma_2^{-1}\mx_2\Big).
\]
Consequently,  the condition that $\E\left[\lVert\mZ ^*_k \rVert^2
          \rho_k(\lVert\mZ ^*_k\rVert^2)^2\right] <\infty$ for $k=1,2$ 
 is sufficient for $\cI_{\mX}(0)=\E\big[\dot{\ell}(\mX;0)^2\big]<\infty$. If $\cI_{\mX}(0)=0$, then we must have
\[
\rho_1\Big(\mx_1^{\top}\mSigma_1^{-1}\mx_1\Big)
=\rho_2\Big(\mx_2^{\top}\mSigma_2^{-1}\mx_2\Big)=C_{\rho}
\]
for some constant $C_{\rho}\ne 0$ and
\[(\fM_1\mx_2)^{\top}\mSigma_1^{-1}\mx_1+(\fM_2\mx_1)^{\top}\mSigma_2^{-1}\mx_2
=\mx_1^{\top}\mSigma_1^{-1}(\fM_1\mSigma_2+\mSigma_1\fM_2^{\top})\mSigma_2^{-1}\mx_2=0\]
for all $\mx_1,\mx_2$. This contradicts the assumption that $\mSigma_1\fM_2^{\top}+\fM_1\mSigma_2\ne \zero$ and completes the proof. 
\end{proof}

\paragraph{Proof of Example~\ref{ex:ellip}\ref{ex:ellip2}}
\begin{proof}[Proof of Example~\ref{ex:ellip}\ref{ex:ellip2}]
For the multivariate normal, $\phi_k(t)=\exp(-t/2)$ and $\rho_k(t)=-1/2$, so that all conditions in Example~\ref{ex:ellip}\ref{ex:ellip1} are satisfied. 
For a multivariate $t$-distribution with $\nu_k$ degrees of freedom,
\[\phi_k(t)=(1+t/\nu_k)^{-(\nu_k+d_k)/2}\quad\text{and}\quad 
\rho_k(t)=-2^{-1}(1+d_k/\nu_k)(1+t/\nu_k)^{-1}.
\]
 It is easily checked that all conditions in Example~\ref{ex:ellip}\ref{ex:ellip1} are satisfied when $\nu_k>2$; see \citet[p.~44--46]{MR2691505}. 
\end{proof}

\subsubsection{Proof of Example~\ref{ex:mixture}}

\begin{proof}[Proof of Example~\ref{ex:mixture}]
Since $q^*$ is continuous and has compact support, it is upper bounded by some constant, say $C_q > 1$, and then Assumption~\ref{asp:far}\ref{asp:far1} holds with $\delta^*=C_q^{-1}$. The rest of Assumption~\ref{asp:far} can be easily verified. 
\end{proof}

\subsubsection{Proof of Proposition~\ref{rmk:imply1}}

\begin{proof}[Proof of Propositiion~\ref{rmk:imply1}]
(1) Konijn family. It is clear that Assumption~\ref{asp:only}\ref{asp:only1},\ref{asp:only3} is satisfied. 
\citet[Appendix~B, p.~105--107]{MR2691505} shows that
Assumption~\ref{asp:power-opt} implies 
Assumption~\ref{asp:only}\ref{asp:only2}. 
To verify Assumption~\ref{asp:only}\ref{asp:only4}, notice that
\[
\dot{\ell}(\mx;0)=
-2(\fM_1\mx_2)^{\top}\Big(\nabla q_1(\mx_1)\big/q_1(\mx_1)\Big)
-2(\fM_2\mx_1)^{\top}\Big(\nabla q_2(\mx_2)\big/q_2(\mx_2)\Big)
\]
following the proof of Lemma~3.2.1 in \citet{MR2691505}. 

(2) Mixture family. Direct computation yields that
\[\dot{\ell}(\mx;\delta)=\frac{q^*(\mx)-q_{1}(\mx_1)q_{2}(\mx_2)}{(1-\delta)\{q_{1}(\mx_1)q_{2}(\mx_2)\}+\delta q^*(\mx)}.\]
The rest directly follows from Theorem~12.2.1 in \citet{MR2135927}.
\end{proof}

\section{Auxiliary results} \label{sec:help}

\subsection{Auxiliary results for Section~\ref{sec:sgc}}\label{appendix:sectionB1}

The concept of GSC unifies a surprisingly large number of well-known
dependence measures.  Moreover, only two types of subgroups are
needed, namely,
$H_{\tau}^m:=\langle(1~2)\rangle=\{(1),(1~2)\}\subseteq\kS_m$ for~$m=2$ and
$H_*^m:=\langle(1~4),(2~3)\rangle=\{(1),(1~4),(2~3),(1~4)(2~3)\}\subseteq\kS_m$
for $m\ge 4$. The following result illustrates this fact with 
four classical examples of  univariate 
dependence measures, namely, the tau of \cite{Kendall1938}, the $D$ of
\cite{MR0029139}, the $R$ of \cite{MR0125690}, and the  $\tau^*$ of
\cite{MR3178526} which, as shown by \cite{MR4185806}, is connected to the work of
\cite{yanagimoto1970measures}. \textcolor{black}{Below, we write $\mw
  = (w_1,\ldots, w_m)\mapsto f_k (\mw)$, $k=1,2$ for the kernel
  functions of an  $m$th order univariate GSC; note that not all components of $\mw$ need  {to} have an impact on $ f_k (\mw)$: see, for instance the kernel $f_1$ of the 6th order Blum--Kiefer--Rosenblatt GSC, which is mapping~$\mw=(w_1,\ldots,w_6)$ to $\R_{\geq 0}$ but does not depend on $w_6$ ($f_2$ does). }

\begin{example}[Examples of univariate GSCs]\label{prop:sgc:1d} \mbox{ }
\begin{enumerate}[itemsep=-.5ex,label=(\alph*)]
\item\label{example:kendall}
  Kendall's tau is a 2nd order GSC with $H=H_{\tau}^2$ and $$f_1(\mw)=f_2(\mw)=\ind(w_1<w_2) \text{ on } \
  \mathbbm{R}^2 {,}$$
 {which can be proved as follows: 
\begin{align*}
&\mu_{f_1,f_2,H_{\tau}^2}(X_1,X_2)
:= \E [k_{f_1,f_2,H_{\tau}^2}((X_{11},X_{21}),(X_{12},X_{22}))]\\
&\quad 
= \E [\{\ind(X_{11}<X_{12})-\ind(X_{12}<X_{11})\}
      \{\ind(X_{21}<X_{22})-\ind(X_{22}<X_{21})\}]\\
&\quad 
= \E [\sign(X_{11}-X_{12})
      \sign(X_{21}-X_{22})] =: \tau;
\end{align*}
see also Example 5 in \citet[Chapter~1.2]{MR1075417} for the last expression of Kendall's $\tau$;}
\item\label{example:D} Hoeffding's $D$ is a 5th order GSC with $H=H_*^5$ and 
$$f_1(\mw)=f_2(\mw)=\frac12\ind(\max\{w_1,w_2\}\le w_5)\ \text{ on
} \ \mathbbm{R}^5;$$
\item\label{example:R} Blum--Kiefer--Rosenblatt's $R$ is a 6th order GSC
  with $H=H_*^6$ and
$$f_1(\mw)=\frac12\ind(\max\{w_1,w_2\}\le w_5), \qquad
f_2(\mw)=\frac12\ind(\max\{w_1,w_2\}\le w_6)\ \text{ on
} \ \mathbbm{R}^6;$$
\item\label{example:BDY} Bergsma--Dassios--Yanagimoto's $\tau^*$ is a
  4th order GSC with $H=H_*^4$ and
  $$f_1(\mw)=f_2(\mw)=\ind(\max\{w_1,w_2\}< \min\{w_3,w_4\}) \text{ on } \
  \mathbbm{R}^4.$$
\end{enumerate}
\end{example} 

\begin{remark}\label{rmk:uniq}
Distinct choices of the kernels $f_1$ and $f_2$ do not necessarily imply distinct GSCs. For example, \citet[Proposition~1(ii)]{MR3842884} showed that
 Hoeffding's~$D$ in  {Example}~\ref{prop:sgc:1d}\ref{example:D}  is a 5th order GSC with $H=H_*^5$
also for $$f_1(\mw)=f_2(\mw)=\frac12\ind(\max\{w_1,w_2\}\le w_5<\max\{w_3,w_4\}) \text{ on } \
  \mathbbm{R}^5;$$
similarly, for Blum--Kiefer--Rosenblatt's $R$,   the
kernels  in  {Example}~\ref{prop:sgc:1d}\ref{example:R} can be replaced with 
  \begin{align*}
    f_1(\mw)&=\frac12\ind(\max\{w_1,w_2\}\le w_5<\max\{w_3,w_4\}) \text{ on } \
  \mathbbm{R}^6, \\
    f_2(\mw)&=\frac12\ind(\max\{w_1,w_2\}\le w_6<\max\{w_3,w_4\}) \text{ on } \
  \mathbbm{R}^6.
  \end{align*}
\end{remark}

\subsection{Auxiliary results for Section~\ref{sec:hallin}}\label{appendix:sectionB2}

The next proposition collects several properties of center-outward  distribution functions.

\begin{proposition} \label{prop:Hallin_full}
  Let $\fF_{\pms}$ be the
  center-outward distribution function of 
  $\P\in\cP_d^{\ac}$.  Then, 
\begin{enumerate}[itemsep=-.5ex,label=(\roman*)]
\item\label{prop:Hallin1}
  (\citealp[Proposition~4.2(i)]{hallin2017distribution},
  \citealp[Proposition~2.1(i),(iii)]{hallin2020distribution})
  $\fF_{\pms}$ is a probability integral transformation of $\R^d$,
  namely, $\mZ\sim \P$ if and only if $\fF_{\pms}(\mZ)\sim \U_d$; 
\item\label{prop:Hallin2}
  \citep[Proposition~2.1(ii)]{hallin2020distribution}  if $\mZ\sim\P$,
   $\|\fF_{\pms}(\mZ)\|$ is uniform over $[0,1)$,
  $\fF_{\pms}(\mZ)/\|\fF_{\pms}(\mZ)\|$ is uniform over the sphere
  $\cS_{d-1}$, and they are mutually independent.
\end{enumerate}
Writing $\fF_{\pms}^{\mZ}$ for the center-outward distribution function
of 
$\mZ\sim\P\in~\!\cP_d^{\ac}$, 
 \begin{enumerate}[itemsep=-.5ex,label=(\roman*)]
\setcounter{enumi}{2}
 \item\label{prop:Hallin5}   \citep[Proposition~2.2]{hallin2020efficient} 
  for any $\mv\in\R^d, a\in\R_{>0},$ and  orthogonal $d\times d$ matrix~$\fO$,
\[
\fF_{\pms}^{\mv+a\fO\mZ}(\mv+a\fO\mz)=\fO\fF_{\pms}^{\mZ}(\mz) \text{ for all }\mz\in\R^d.
\]
\end{enumerate}
Letting $\mZ_1,\ldots,\mZ_n$ be independent copies of $\mZ\sim\P\in~\!\cP_d^{\ac}$ with center-outward distribution function~$\fF_{\pms}$,
\begin{enumerate}[label=(\roman*)]
\setcounter{enumi}{3}
\item\label{prop:Hallin3} (\citealp[Proposition~6.1(ii)]{hallin2017distribution}, \citealp[Proposition~2.5(ii)]{hallin2020distribution}) for any decomposition $n_0,n_R,{n_S}$ of $n$, the random
vector $[\fF_{\pms}^{(n)}(\mZ_1),\dots,\fF_{\pms}^{(n)}(\mZ_n)]$ is
uniformly distributed over all distinct arrangements of the grid
$\kG^{d}_{\mn}$;
\item\label{prop:Hallin:add} (\citealp[Proof of Theorem~3.1]{del2018smooth}, \citealp[Proof of Proposition 3.3]{hallin2020distribution}) as $n_R$ and $n_S\to\infty$, for every $i\in\zahl{n}$, 
\[
\Big\lVert\fF_{\pms}^{(n)}(\mZ_{i})-\fF_{\pms}(\mZ_{i})\Big\rVert\stackrel{\sf a.s.}{\longrightarrow}0.
\]
\end{enumerate}
\end{proposition}

\begin{proof}[Proof of Proposition~\ref{prop:Hallin_full}]
We give an independent proof of part \ref{prop:Hallin5}. In view of Definition \ref{def:centerdistr}, there exists a convex function $\Psi$ such that $\fF_{\pms}^{\mZ}=\nabla\Psi$. It is obvious that $\fF_{\pms}^{\mv+a\fO\mZ}$ defined implicitly by
$$\fF_{\pms}^{\mv+a\fO\mZ}(\mv+a\fO\mz)=\fO\fF_{\pms}^{\mZ}(\mz),$$
satisfies (ii) and (iii) in Definition \ref{def:centerdistr}. It  only remains, thus,  to construct a convex function $\Psi^{*}$ such that $\fF_{\pms}^{\mv+a\fO\mZ}=\nabla\Psi^{*}$. Noting that $\fF_{\pms}^{\mv+a\fO\mZ}(\mz)=\fO\fF_{\pms}^{\mZ}(a^{-1}\fO^{-1}(\mz-\mv))$, it is easy to check that 
 $\mz\mapsto\Psi^{*}(\mz):=a\Psi\left(a^{-1}\fO^{-1}(\mz-\mv)\right)$ 
 is convex, and thus  continuous and almost everywhere differentiable,  
with $\nabla\Psi^{*}(\mv+a\fO\mZ)=\fO\nabla\Psi(\mz)$. 
\end{proof}


\begin{proposition} 
  (\citealp[Proposition~5.1]{hallin2017distribution},
\citealp[Theorem~3.1]{del2018smooth},
\citealp[Theorem~2.5]{MR4147635}, and
\citealp[Proposition~2.3]{hallin2020distribution})
  \label{prop:Hallin4_full}
Consider the following classes of distributions:
\begin{itemize}[itemsep=-.5ex]
\item the class $\cP_d^{+}$  of distributions $\P\in\cP_d^{\ac}$
  with {\it nonvanishing probability density}, namely,  with Lebesgue density $f$ such that, for all $D>0$ there exist constants
  $\lambda_{D;f}<\Lambda_{D;f}\in(0,\infty)$ such that
  $\lambda_{D;f}\le f(\mz)\le \Lambda_{D;f}$ for all
  $\lVert\mz\rVert\le D$;
\item the class $\cP_d^{\mathrm{conv}}$  of distributions
  $\P\in\cP_d^{\ac}$ with convex support 
  $\overline{\mathrm{supp}}(\P)$ and a density  that is
  nonvanishing over this support, namely,  with density $f$ such that, for all $D>0$ there exist constants
  $\lambda_{D;f}< \Lambda_{D;f}\in(0,\infty)$ such that
  $\lambda_{D;f}\le f(\mz)\le \Lambda_{D;f}$ for all $\mz\in {\mathrm{supp}}(\P)$ with~$\lVert\mz\rVert\le D$;
%
\item the class $\cP_d^{\pms}$  of distributions $\P\in\cP_d^{\ac}$
  that are push-forwards of $\U_d$ of the form
  $\P=\nabla\Upsilon\sharp\U_d$ ($\nabla\Upsilon$  the
  gradient of a convex function)   
 and   a homeomorphism from the punctured ball~$\bS_d\backslash\{\bm0_d\}$ to~$\nabla\Upsilon(\bS_d\backslash\{\bm0_d\})$ such that $\nabla\Upsilon(\{\bm0_d\})$ is  compact, convex, and has Lebesgue measure zero;
\item the class  $\cP_d^{\#}$ of all distributions $\P\in\cP_d^{\ac}$ 
such that, denoting by $\fF_{\pms}^{(n)}$ the sample distribution function computed from an $n$-tuple $\mZ_1,\ldots,\mZ_n$ of  independent copies of $\mZ\sim\P$,  
\[
\max_{1\le i\le n}\Big\lVert\fF_{\pms}^{(n)}(\mZ_{i})-\fF_{\pms}(\mZ_{i})\Big\rVert\stackrel{\sf a.s.}{\longrightarrow}0
\]
as $n_R$ and $n_S\to\infty$ (a Glivenko-Cantelli property). 
\end{itemize}
  It holds that
  $\cP_d^{+}\subsetneq\cP_d^{\mathrm{conv}}\subsetneq\cP_d^{\pms}\subseteq\cP_d^{\#}\subsetneq\cP_d^{\ac}$. 
\end{proposition}

\subsection{Auxiliary results for Section~\ref{sec:measure}}\label{appendix:sectionB3}


The time complexity of computing the optimal matching and nearly optimal matchings 
is summarized in the following proposition.

\begin{proposition}\label{thm:matching-computation}
The optimal matching problem  \eqref{eq:assignment} yielding $[\fG_{1,\pms}^{(n)}(\mX_{1i})]_{i=1}^{n}$ and
$[\fG_{2,\pms}^{(n)}(\mX_{2i})]_{i=1}^{n}$ can be solved in $O(n^3)$ time  via the refined Hungarian algorithm \citep{MR0254076,MR0297347,MR0266680,edmonds1972theoretical}. Moreover, 
\begin{enumerate}[itemsep=-.5ex,label=(\roman*)]
\item if we assume that $c_{ij},~i,j\in\zahl{n}$  all are integers and bounded by some (positive) integer $N$, which can be achieved by scaling and rounding, then there exists an  optimal matching algorithm  solving the problem in $O(n^{5/2}\log(nN))$ time \citep{MR1015271}; 
\item if $d=2$ and $c_{ij},~i,j\in\zahl{n}$ all are  integers and bounded by some (positive) integer $N$, there exists an exact an  optimal matching algorithm  solving the problem in   $O(n^{3/2+\delta}\log(N))$ time for any arbitrarily small constant $\delta>0$ \citep{MR3205219};
\item if $d\ge3$, there is an algorithm  computing a $(1+\epsilon)$-approximate perfect matching in 
$$O\left(n^{3/2}\epsilon^{-1}\tau(n,\epsilon)\log^{4}(n/\epsilon)\log\left({\max c_{ij}}/{\min c_{ij}}\right)\right)\quad\text{ time},$$ where a $(1+\epsilon)$-approximate perfect matching for $\epsilon>0$ is a bijection $\sigma$ from $\zahl{n}$ to itself such that $\sum_{i=1}^{n}c_{i\sigma(i)}$ is no larger than $(1+\epsilon)$ times the cost of 
the optimal matching 
 and $\tau(n,\epsilon)$ is the query and update time of an $\epsilon/c$-approximate nearest neighbor data structure for some constant $c>1$ \citep{MR3238983}. 
\end{enumerate}
\end{proposition}


Once  $[\fG_{1,\pms}^{(n)}(\mX_{1i})]_{i=1}^{n}$ and
$[\fG_{2,\pms}^{(n)}(\mX_{2i})]_{i=1}^{n}$ are obtained, a naive
approach to the computation of~$\tenq{W}\n\!$, on the other hand,  requires at most   a $O(n^m)$ time complexity.
Great speedups are possible, however, in particular
cases and    
the next proposition summarizes the results for  the various center-outward rank-based statistics listed in Example~\ref{ex:rbsgc}.

\begin{proposition}\label{thm:correlation-computation}
Assuming that $[\fG_{1,\pms}^{(n)}(\mX_{1i})]_{i=1}^{n}$ and $[\fG_{2,\pms}^{(n)}(\mX_{2i})]_{i=1}^{n}$ have been previously obtained, one can compute 
\begin{enumerate}[itemsep=-.5ex,label=(\roman*)]
\item $\tenq{W}\n_{{\dCov}}$  in $O(n^2)$ time (\citealp[Definition~1]{MR3053543}, \citealp[Definition~2, Proposition~1]{MR3269983}, \citealp[Lemma~3.1]{MR3556612})
\item $\tenq{W}\n_{{{M}}}$    in $O(n (\log n)^{d_1+d_2-1})$ time \citep[p.~557, end of Sec.~5.2]{MR3842884},
\item $\tenq{W}\n_{{D}}$      in $O(n^3)$ time \citep[Theorem~1]{MR3737307},
\item $\tenq{W}\n_{{R}}$      in $O(n^4)$ time as proved in Section~A.3.4 of the supplement,
\item $\tenq{W}\n_{{\tau^*}}$ in $O(n^4)$ time by definition. 
\end{enumerate}
If, moreover, approximate values are allowed, one can compute 
\begin{enumerate}[itemsep=-.5ex,label=(\roman*)]
\item approximate $\tenq{W}\n_{{\dCov}}$    in $O(nK\log n)$ time (\citealp[Theorem~4.1]{MR3556612}, \citealp[Theorem~3.1]{MR3913146}),
\item approximate $\tenq{W}\n_{{D}}$      in $O(nK\log n)$ time \citep[p.~557]{MR3842884},
\item approximate $\tenq{W}\n_{{R}}$      in $O(nK\log n)$ time (\citealp[Equation~(6.1)]{MR4185806}, \citealp[p.~557]{MR3842884}, \citealp[Corollary~4]{doi:10.1137/1.9781611976465.136}), 
\item approximate $\tenq{W}\n_{{\tau^*}}$ in $O(nK\log n)$ time \citep[Corollary~4]{doi:10.1137/1.9781611976465.136}.
\end{enumerate}
These approximations consider random projections to speed up
computation;  $K$ stands for the number of random projections. See
also
\citet[Sec.~3.1]{huang2017statistically}.
\end{proposition}


\begin{proof}[Proof of Proposition~\ref{thm:correlation-computation}]

We only illustrate how to efficiently compute U-statistic estimates of 
Hoeffding's multivariate projection-averaging $D$ and
Blum--Kiefer--Rosenblatt's  multivariate projection-averaging $R$.  The other claims   straightforwardly follow from the sources provided in the proposition. 

\citet{MR3737307} showed how to efficiently compute a V-statistic
estimate of Hoeffding's multivariate projection-averaging $D$. Let us show how to efficiently compute the corresponding 
U-statistic.
We define arrays $(a_{k\ell rs})_{\ell,r,s\in\zahl{n}}$ for $k=1,2$ as
\[
\begin{cases}
a_{k\ell rs}:=\Arc(\bm{y}_{k\ell}-\bm{y}_{ks},\bm{y}_{kr}-\bm{y}_{ks})&~~~\text{if }[\ell,r,s]\in I^{n}_{3},\\
a_{k\ell rs}:=0&~~~\text{otherwise}.
\end{cases}
\]
Their U-centered versions $(A_{k\ell rs})_{\ell,r,s\in\zahl{n}}$ for $k=1,2$ are
\begin{align*}
A_{k\ell rs}&:=\begin{cases}\displaystyle a_{k\ell rs}-\frac1{n-3}\sum_{i=1}^{n}a_{ki rs}-\frac1{n-3}\sum_{j=1}^{n}a_{k\ell js}+\frac1{(n-2)(n-3)}\sum_{i,j=1}^{n}a_{kijs}
  &~~\text{if }[\ell,r,s]\in I^{n}_{3},\\
  0&~~\text{otherwise}.
\end{cases}
\end{align*}
Then, 
\[
\mbinom{n}{5}^{-1}
\sum_{i_1<\cdots<i_5}h_{D}\Big((\bm{y}_{1i_1},\bm{y}_{2i_1}),\dots,(\bm{y}_{1i_5},\bm{y}_{2i_5})\Big)=\frac{1}{n(n-1)(n-4)}\sum_{[\ell,r,s]\in I^{n}_{3}} A_{1\ell rs}A_{2\ell rs},
\]
\textcolor{black}{which clearly has $O(n^3)$ complexity.}

Turning to Blum--Kiefer--Rosenblatt's  multivariate projection-averaging $R$,   define,  for $k=1,2$, the arrays~$(b_{k\ell rst})_{\ell,r,s,t\in\zahl{n}}$  as
\[
\begin{cases}
b_{1\ell rst}:=\Arc(\bm{y}_{1\ell}-\bm{y}_{1s},\bm{y}_{1r}-\bm{y}_{1s})~\text{and}~
b_{2\ell rst}:=\Arc(\bm{y}_{2\ell}-\bm{y}_{2t},\bm{y}_{2r}-\bm{y}_{2t})&~~~\text{if }[\ell,r,s,t]\in I^{n}_{4},\\
b_{1\ell rst}:=0~\text{and}~
b_{2\ell rst}:=0&~~~\text{otherwise}.
\end{cases}
\]
Their U-centered versions $(B_{k\ell rs})_{k,\ell,r,s\in\zahl{n}}$ for $k=1,2$ are 
\begin{align*}
B_{k\ell rst}&:=\begin{cases}\displaystyle b_{k\ell rst}-\frac1{n-4}\sum_{i=1}^{n}b_{kirst}-\frac1{n-4}\sum_{j=1}^{n}b_{k\ell jst}+\frac1{(n-3)(n-4)}\sum_{i,j=1}^{n}b_{kijst}
  &\text{if }[\ell,r,s,t]\in I^{n}_{4},\\
  0&\text{otherwise}.
  \end{cases}
\end{align*}
Then, 
\[
\mbinom{n}{6}^{-1}
\sum_{i_1<\cdots<i_6}h_{R}\Big((\bm{y}_{1i_1},\bm{y}_{2i_1}),\dots,(\bm{y}_{1i_6},\bm{y}_{2i_6})\Big)=\frac{1}{n(n-1)(n-2)(n-5)}\sum_{[\ell,r,s,t]\in I^{n}_{4}} B_{1\ell rst}B_{2\ell rst},
\]
\textcolor{black}{which clearly has $O(n^4)$ complexity.} 
This completes the proof.
\end{proof}

\nocite{ghosal2019multivariate,MR3737306,MR4147635}

{\small 
\bibliographystyle{apalike}
\bibliography{AMS}
}

\end{document}